%% file: final.tex
\documentclass{amsart}
\usepackage[margin=1.2in]{geometry}

\usepackage{xcolor}
\usepackage{url}
\usepackage{amssymb}
\usepackage{amsmath}
\usepackage{amsthm}
\usepackage{amsfonts}
\usepackage{graphicx}
\usepackage{float}
\usepackage{wrapfig}
\usepackage{enumerate}
\usepackage{listings}
\usepackage{dsfont}
\usepackage{mathtools}
\usepackage{siunitx}
\usepackage{tikz}
\usepackage{circuitikz}
\usepackage{tcolorbox}
\usepackage{extarrows}
\usetikzlibrary{arrows.meta, decorations.pathreplacing, calligraphy}

\newtheorem{theorem}{Theorem}[section]

\newtheorem{lemma}[theorem]{Lemma}
\newtheorem{proposition}[theorem]{Proposition}
\newtheorem{definition}[theorem]{Definition}
\newtheorem{example}[theorem]{Example}

\newtheorem{remark}[theorem]{Remark}

\newtheorem{assumption}[theorem]{Assumption}

\newcommand{\E}{\mathbb{E}}
\renewcommand{\P}{\mathbb{P}}

\newcommand{\R}{\mathbb{R}}
\newcommand{\N}{\mathbb{N}}

\newcommand{\slices}[1]{\widehat{\bold{#1}}}
\renewcommand{\epsilon}{\varepsilon}

\DeclareFontFamily{U}{mathx}{}
\DeclareFontShape{U}{mathx}{m}{n}{<-> mathx10}{}
\DeclareSymbolFont{mathx}{U}{mathx}{m}{n}
\DeclareMathAccent{\widecheck}{0}{mathx}{"71}

\title[Convergence of adapted empirical measure for mixing observations]{Convergence of the adapted empirical measure for mixing observations}

\author{Ruslan Mirmominov}
\address{Ruslan Mirmominov\newline
\hbox{}\hspace{0.33cm} Carnegie Mellon University\newline
\hbox{}\hspace{0.33cm} Department of Mathematics}
\email{rmirmomi@andrew.cmu.edu}

\author{Johannes Wiesel}
\address{Johannes Wiesel\newline
\hbox{}\hspace{0.33cm} University of Copenhagen \newline
\hbox{}\hspace{0.33cm} Department of Mathematics}
\email{wiesel@math.ku.dk}

\begin{document}
\maketitle
\begin{abstract}
The adapted Wasserstein distance $\mathcal{AW}$ is a modification of the classical Wasserstein metric, that provides robust and dynamically consistent comparisons of laws of stochastic processes, and has proved particularly useful in the analysis of stochastic control problems, model uncertainty, and mathematical finance.
In applications, the law of a stochastic process $\mu$ is not directly observed, and has to be inferred from a finite number of samples. As the empirical measure is not $\mathcal{AW}$-consistent, Backhoff, Bartl, Beiglb\"ock and Wiesel introduced the adapted empirical measure $\widehat{\mu}^N$, a suitable modification, and proved its $\mathcal{AW}$-consistency when observations are i.i.d. 

In this paper we study $\mathcal{AW}$-convergence of the adapted empirical measure $\widehat{\mu}^N$ to the population distribution $\mu$, for observations satisfying a generalization of the $\eta$-mixing condition introduced by Kontorovich and Ramanan. We establish moment bounds and sub-exponential concentration inequalities for $\mathcal{AW}(\mu,\widehat{\mu}^N)$, and prove consistency of $\widehat{\mu}^N$. In addition, we extend the Bounded Differences inequality of Kontorovich and Ramanan for $\eta$-mixing observations to uncountable spaces, a result that may be of independent interest. Numerical simulations illustrating our theory are also provided.
\end{abstract}

\section{Introduction}

In recent years, the field of statistical optimal transport (OT) has gained a lot of traction. While originally concerned with the problem of quantifying convergence of the empirical measure $\mu^N=1/N\sum_{n=1}^N \delta_{X_n}$ of i.i.d.~samples $X_1, \dots, X_N$ to its population counterpart $\mu$ in the Wasserstein metric $
\mathcal{W}$, statistical OT has by now made important contributions to many central questions in statistics and machine learning. We refer to \cite{chewi2024statistical} and the references therein for an overview of recent developments.

In this article we set out to study a variant of the original estimation problem $\mathcal{W}(\mu, \mu^N)$, on the Wasserstein space of discrete time stochastic processes. More specifically, assuming that we observe a (not necessarily i.i.d.)~sample $\slices{X}:=(X^n)_{n = 1}^N \in (\R^d)^T$ from a stochastic process with finite time horizon $T>1$ (see Figure \ref{fig:1}), we aim to estimate the convergence of this sample to its data generating distribution in the adapted Wasserstein distance $\mathcal{AW}$. 

\begin{figure}[h!]
\begin{center}
    \input{figures/observations.tikz}    
\end{center}
    \label{fig:1}
    \caption{Observations $\slices{X}:=(X^n)_{n = 1}^N \in (\R^d)^T$.}
\end{figure}
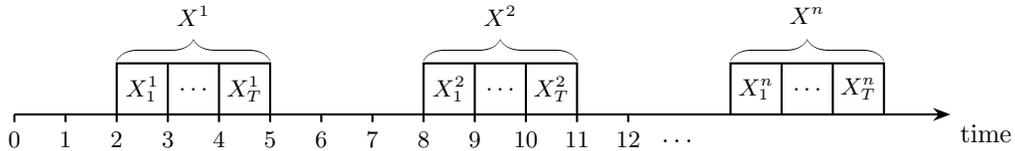

As explained e.g.~in \cite{backhoff2020adapted, backhoff2020all, blanchet2024bounding, bartl2021wasserstein}, the adapted (or nested) Wasserstein distance is a strengthening of the usual Wasserstein distance $\mathcal{W}$. It generates a topology, that makes time-dependent optimization problems of the type $$\mu \mapsto V(\mu):=\sup_{a\in \mathcal{A}} \int f(x,a)\,\mu(dx)$$ continuous; here $\mathcal{A}$ is the set of predictable strategies or the set of stopping times. Importantly, $\mu \mapsto V(\mu)$ is not continuous in $\mathcal{W}$, so it is not sufficient to consider the Wasserstein metric. $\mathcal{AW}$ offers attractive properties for statistical estimation of time-dependent optimization problems, e.g.~in mathematical finance and operations research. We refer to \cite{bartl2021wasserstein} for a survey of the theory of adapted optimal transport in discrete time. It is worth pointing out that many time-dependent optimization problems $V$ classically studied in mathematical finance are actually Lipschitz continuous wrt.~$\mathcal{AW}$; see \cite{backhoff2020adapted}. In particular, bounds in the adapted Wasserstein distance immediately yield strongly consistent estimators for \emph{all} such functionals, rendering tailor made ad hoc constructions unnecessary.

Statistical inference on the adapted Wasserstein space goes back at least to \cite{pflug2016empirical, glanzer2019incorporating}, showing that $\lim_{N\to \infty} \mathcal{AW}(\mu, \mu^N)\neq 0$ in general. This follows from the simple observation that for continuous laws $\mu$, a sample path drawn from the empirical measure $\mu^N$ can almost surely be reconstructed from its value at the initial time; in other words, $(X^n_1)_{n = 1}^N$ already contains the full information structure of $(X^n)_{n=1}^N$ under $\mu^N$.  This is why \cite{backhoff2022estimating} introduce the so-called \emph{adapted empirical measure} $\widehat{\mu}^N$. Assuming that $(X^n)_{n=1}^N$ are i.i.d, \cite{backhoff2022estimating} show that $\lim_{n\to \infty} \mathcal{AW}( \mu, \widehat{\mu}^N)=0$ for compactly supported measures $\mu$ and quantify the speed of convergence for suitably regular measures; see also the subsequent work \cite{acciaio2024convergence} for an extension to probability measures $\mu$ with polynomial moments. Convergence of empirical measures in the adapted weak topology has also been studied in \cite{hou2024convergence, larsson2025fast, blanchet2024bounding}.

While the above works lay the basis for statistical inference in the adapted Wasserstein space, they do not lend themselves to practical applications in mathematical finance or operations research: in reality, a statistician typically only has access to \emph{one single} time series of observations, that is often assumed to be stationary. In other words, the observations $(X^n)_{n = 1}^N$ are \emph{not} independent. Well-known examples of such settings include observations from stationary mixing Markov processes and more complicated time series models. In consequence, the consistency results of \cite{acciaio2024convergence, backhoff2022estimating} cannot be applied in these cases.

In this article we address this gap in the literature by deriving the first rates of convergence for $\mathcal{AW}(\mu, \widehat{\mu}^N)$, where $(X^n)_{n = 1}^N$ are mixing observations with law $\mu$. For simplicity of exposition in this Introduction, we state our three main results for compactly supported measures $\mu$ and refer to Section \ref{sec:main_results} below for more general statements.

Our first main result is the bound
\begin{align} \label{eq:sample_complex_intro}
\mathbb{E}\,\mathcal{AW}(\mu, \widehat \mu^N) \leq C \cdot \sqrt{1 + 2 \sum_{s = 1}^{N-1} \eta_{\slices{X}}(s)} \cdot \begin{cases}
N^{- 1 / (T + 1)}, &d = 1,\\
N^{- 1 / (2T)} \log (N + 1), &d = 2,\\
N^{- 1 / (dT)}, &d \geq 3,
\end{cases}
\end{align}
on the sample complexity, assuming that $\mu$ has regular kernels (see Def. \ref{def:Lipschitz_kernels}). Here the constant $C > 0$ depends on $L$, $d$, $T$ and $\operatorname{spt}(\mu)$; see Theorems \ref{thm:moment_estimate_compact} and \ref{thm:rate_general}. Our second main result is the concentration inequality
\begin{align}\label{eq:concentration_intro}
\mathbb{P}\left(\left|\mathcal{AW}(\mu, \widehat \mu^N) - \mathbb{E}\, \mathcal{AW}(\mu, \widehat \mu^N)\right| > \varepsilon\right) \leq 2 \exp\left(-c \cdot N \cdot \frac{\varepsilon^2}{\operatorname{diam}(\operatorname{spt}(\mu))^2 \cdot (1 + \sum_{s = 1}^{N-1} \bar \eta_{\slices{X}}(s))^2}\right),
\end{align}
where the constant $c > 0$ depends on $T$; see Theorems \ref{thm:AW_concentration_compact} and \ref{thm:AW_concentration_general}. Our last main result is the consistency of the adapted empirical measure 
\begin{align}\label{eq:consistency_intro}
\lim_{N\to \infty} \mathcal{AW}(\mu, \widehat{\mu}^N) =0
\end{align}
under an appropriate mixing condition (stated in terms of $\eta_{\slices{X}}(s)$ and $\bar \eta_{\slices{X}}(s)$ for $s>0$);
see Theorems \ref{thm:consistency_general} and \ref{thm:consistency_general_noncompact}. 

In the statements above, the mixing coefficients $\eta_{\slices{X}}(s)$ and $\bar \eta_{\slices{X}}(s)$ play a central role. They are formally defined in Section \ref{sec:mixing_conditions} as
\begin{align*}
\eta_{\slices{X}}(s) &:= \sup_{\substack{1\le n<N-s+1 \\ A_{1:n} \in \mathcal{B}(\R^d)^{n}\\\mathbb{P}(X^{1:n} \in A_{1:n}) > 0}} \operatorname{TV}(\operatorname{Law}(X^{n+s} \mid X^{1:n} \in A_{1:n}), \operatorname{Law}(X^{n+s} \mid X^{1:n-1} \in A_{1:n-1})),\\
\bar \eta_{\slices{X}}(s) &:= \sup_{\substack{1\le n<N-s+1\\ A_{1:n} \in \mathcal{B}(\R^d)^{n}\\\mathbb{P}(X^{1:n} \in A_{1:n}) > 0}} \operatorname{TV}(\operatorname{Law}(X^{n+s:N} \mid X^{1:n} \in A_{1:n}), \operatorname{Law}(X^{n+s:N} \mid X^{1:n-1} \in A_{1:n-1}))
\end{align*}
for $1\le s<N.$
In words, $\eta_{\slices{X}}(s)$ measures the total variation distance between the conditional laws $\text{Law}(X^{n+s}|X^{1:n})$ and $\text{Law}(X^{n+s}|X^{1:n-1})$ uniformly over $X^{1:n} \in A_{1:n}$; a similar interpretation holds for $\bar \eta_{\slices{X}}(s)$. Intuitively one expects $\eta_{\slices{X}}(s)$ to decrease when the time lag $s$ increases; see Figure \ref{fig:1}. As we will discuss in Section \ref{sec:mixing_conditions} below, this is indeed the case in many classical examples. In particular $\eta_{\slices{X}}(s)=0$ for all $s>0$ if and only if $\slices{X}$ consists of independent observations (see Proposition \ref{prop:independent}). 

The coefficients $\eta_{\slices{X}}(s)$ and $\bar\eta_{\slices{X}}(s)$ can be interpreted as an extension of the $\eta$-mixing coefficient of \cite{kontorovich2008concentration} to uncountable probability spaces. To the best of our knowledge, this extension is new. We will see below that the coefficients $\eta_{\slices{X}}(s)$ and $\bar\eta_{\slices{X}}(s)$ are widely applicable in the study of adapted empirical convergence. They also offer a direct connection of the sample complexity result \eqref{eq:sample_complex_intro} to the i.i.d.~case: in fact, as soon as $\eta_{\slices{X}}(s)$ is summable, \eqref{eq:sample_complex_intro} yields the same bound on the sample complexity as in the independent case; cf.~\cite{acciaio2024convergence, backhoff2022estimating}.

We obtain \eqref{eq:sample_complex_intro} with a similar methodology as in \cite{acciaio2024convergence} and \cite{backhoff2022estimating}. Indeed, going back to \cite{backhoff2022estimating}, such bounds are proved by reducing the estimation of the adapted Wasserstein distance to a bound of the form $\mu^N(G) \cdot \mathcal{W}(\mu_G, (\mu^N)_G)$, where $\mu_G(\cdot) := \mu(\cdot \mid x_{1:t} \in G)$ for some Borel set $G$, i.e.~on linear combinations of the classical Wasserstein distance between conditional laws. These terms are then analyzed using moment estimates for $\mathcal{W}$ from \cite{fournier2015rate}. Importantly, this method relies on the fact that $(\mu^N)_G \sim (\mu_G)^{N \mu^N(G)}$, i.e.~the conditional law of $
\mu^N$ has the same distribution as the empirical measure of the conditional law, with the number of samples adjusted to satisfy the disintegration condition. 
However, this identity in distribution is \emph{false} in general, if the observations $\slices{X}$ are mixing. Hence it is not possible in our setting to use moment estimates from \cite{fournier2015rate} directly. Instead, we bound the Wasserstein distance by a linear combination of the differences of probabilities of Borel sets; see \cite[Lemma 2]{dereich2013constructive} and \cite[Lemma 6]{fournier2015rate}. These differences can be estimated by a covariance bound under the $\eta$-mixing condition. As a comparison, if $\slices{X}$ are independent, we recover the same rate as in \cite[Theorem 1.5]{backhoff2022estimating} for compactly supported probability measures $\mu$, and the same rate as in \cite[Theorem 2.16, (i)]{acciaio2024convergence} for general $\mu$.

In order to obtain \eqref{eq:concentration_intro}, we first establish general concentration inequalities via the \emph{martingale differences} method \cite[Lemma 1 (Azuma's inequality)]{azuma1967weighted}, and then apply them to $\mathcal{AW}(\mu, \widehat\mu^N)$. Our approach differs from \cite{acciaio2024convergence} and \cite{backhoff2022estimating}: as we have already noted above, independence of $\slices{X}$ implies $(\mu^N)_G \sim (\mu_G)^{N \mu^N(G)}$. Hence concentration inequalities for the Wasserstein distance from \cite{fournier2015rate} or the Bounded Differences inequality (see e.g.~\cite{mcdiarmid1989method} and \cite[Theorem 3]{kontorovich2014concentration}) can be used to obtain concentration for the terms $\mathcal{W}(\mu_G, (\mu^N)_G)$. 
In more detail, \cite{fournier2015rate} uses delicate estimates relying on the fact that $\mu^N(A)$ has a binomial distribution. Naturally, this fails if the observations $\slices X$ are not independent. Similarly, we cannot use the \emph{Poissonization trick} of \cite{fournier2015rate}, where $N$ is replaced by a Poisson random variable with expected value $N$ and one shows that the error for such a replacement is negligible.
Instead, we achieve concentration via the distribution-free martingale differences method. Our new concentration inequalities extend the classical Bounded Differences inequality. They are potentially useful in other applications, e.g.~for bounds of the Wasserstein distance or the Conditional Value-At-Risk. 

In order to establish \eqref{eq:consistency_intro}, we combine smoothed TV bounds as in \cite{hou2024convergence} with our newly established concentration inequalities. 
In  fact, if $\mu$ admits continuous disintegrations, then we can follow the arguments of \cite{backhoff2022estimating}. However, the subsequent extension to general $\mu$ via Lusin's theorem breaks down in our setting, as it is greatly affected by the weak dependencies of $\slices{X}$. In order to resolve this issue, we convolve $\widehat\mu_N$ with independent Gaussian noise and use the methodology of \cite{hou2024convergence} to bound the Adapted Wasserstein distance by the smoothed TV distance. 
A potentially different approach would use consistency for the independent case via a \textit{decoupling} argument: here one would replace $\slices{X} = (X^n)_{n = 1}^N$ with an independent version $\widehat{\bold Y} = (Y^n)_{n = 1}^N$ such that $X^n \sim Y^n$ and $\mathbb{P}\left(X^n \neq Y^n\right) \leq \phi_n$ for a sequence $(\phi_n)_{n
\in \N}$ that depends on the mixing properties of $\bold X$. However, we were unable to identify a sequence $(\phi_n)_{n\in \N}$ that decays with $n$; the closest result we could find is \cite{ruschendorf1985wasserstein}.

The remainder of the article is structured as follows: Section \ref{sec:notation} introduces notation. In Section \ref{sec:setting} we define the adapted Wasserstein distance $\mathcal{AW}$, the adapted empirical measure $\widehat{\mu}^N$ and the mixing conditions for the observations $(X^n)_{n = 1}^N$ formally, and make some preliminary observations. We state the main results of the paper in Section \ref{sec:main_results}. Section \ref{sec:numerics} is dedicated to numerical experiments. In Section \ref{sec:moment_estimate} we prove \eqref{eq:sample_complex_intro} and some auxiliary results. We state and prove general concentration inequalities and their applications to \eqref{eq:concentration_intro} in Section \ref{sec:concentration}. Afterwards we prove \eqref{eq:consistency_intro} in Section \ref{sec:consistency}. The remaining technical results are stated and proved in Section \ref{sec:remaining}.

\section{Notation}\label{sec:notation}
For a Polish space $(\mathcal{X}, d_{\mathcal{X}})$ we denote by $\mathcal{P}(\mathcal{X})$ the set of Borel probability measures on $\mathcal{X}$. We denote the Borel $\sigma$-algebra on $\mathcal{X}$ by $\mathcal{B}(\mathcal{X})$. For $p \ge 1$ we write $\mathcal{P}_p(\mathcal{X})$ for the set of Borel probability measures $\mu$ on $\mathcal{X}$ satisfying $\int d_{\mathcal{X}}(x_0, x)^p \, \mu(dx) < \infty$ for some $x_0 \in \mathcal{X}$. For two probability measures $\mu, \nu \in \mathcal{P}(\mathcal{X})$ we denote by $\Pi(\mu, \nu)$ the set of probability measures $\pi \in \mathcal{P}(\mathcal{X} \times \mathcal{X})$ such that $\pi(\cdot \times \mathcal{X}) = \mu(\cdot)$ and $\pi(\mathcal{X} \times \cdot) = \nu(\cdot)$. Let us denote the total variation distance between $\mu$ and $\nu$ by $\operatorname{TV}(\mu, \nu)$. We write $\operatorname{spt}(\mu)$ for the support of a measure $\mu$. Any measure $\pi\in \Pi(\mu,\nu)$ admits a $\mu$-a.s. defined disintegration $x\mapsto \pi_x$ satisfying $\pi(A\times B)=\int_A \pi_x(B)\,\mu(dx)$ for all $A,B\in \mathcal{B}(\mathcal{X}).$ 

In this paper we often consider sequences of random elements taking values in some Polish space $(\mathcal{X}, d_{\mathcal{X}})$, as well as the corresponding sequences of \emph{slices} of length $T >1$. We use subscript to enumerate the former and superscript for the latter, i.e. $(x_n)_{n = 0}^\infty \subset \mathcal{X}$ and $(x^n)_{n = 0}^\infty \subset \mathcal{X}^T$. We use $x_{i:j}$ ($x^{i:j}$) to abbreviate the vector $(x_i, \ldots, x_j) \in \mathcal{X}^{j - i + 1}$ (or $(x^i, \ldots, x^j) \in (\mathcal{X}^T)^{j - i + 1}$) for $i \leq j$; if $i > j$ then we set $x_{i:j} := \emptyset$. Analogously we abbreviate $A_s \in \mathcal{B}(\mathcal{X})$ for $s \in \{1, \ldots, k\}$, where $(\mathcal{X}, d_\mathcal{X})$ is a metric space, by $A_{1:k} \in \mathcal{B}(\mathcal{X})^k$.  If $\Phi: \mathcal{X} \to \mathcal{Y}$ is Borel measurable, then we denote $\bigotimes_{s = 1}^k \Phi^{-1}(A_s)$ by $\Phi^{-1}(A_{1:k}) \in \mathcal{B}(\mathcal{X})^k$, where $A_{1:k} \in \mathcal{B}(\mathcal{Y})^k$, and abbreviate $(\Phi(x_1), \ldots, \Phi(x_k)) \in \mathcal{Y}^k$ by $\Phi(x_{1:k})$ for $x_{1:k} \in \mathcal{X}^k$. We write $ab \in \mathcal{X}^{s + t}$ for the concatenation of vectors $a \in \mathcal{X}^s$ and $b \in \mathcal{X}^t$, and analogously write $AB := A \times B \in \mathcal{B}(\mathcal{X}) \times \mathcal{B}(\mathcal{Y})$ for $A \in \mathcal{B}(\mathcal{X})$ and $B \in \mathcal{B}(\mathcal{Y})$.

We equip the Euclidean space $\R^k$, $k\ge 1$, with the canonical scalar product $\cdot$ and the Euclidean norm $\|x\|=(\sum_{l=1}^k |x_l|^2)^{1/2}$.
We denote the diameter of a set $S \subseteq \mathbb{R}^j$ by $\operatorname{diam}(S) := \sup_{x, y \in S} \|x - y\|$. Denote the norm of a set $S$ by $\|S\| := \sup_{x \in S} \|x\|$. Furthermore, we write $B_\delta(z)$ for a closed ball of radius $\delta$ around $z \in \R^k$.

Returning to the Introduction, we write $\widecheck{\bold X} := (X_n)_{n = 1}^N$ for a sequence of random elements on $(\R^d,\|\cdot\|)$; analogously we write $\slices{X} := (X^n)_{n = 1}^N$ for a sequence of slices in $(\R^d)^T$. All random elements are defined on some fixed probability space $(\Omega, \mathcal{F}, \mathbb{P})$. We use bold font to denote sequences of random elements, and regular font for its elements. In generic statements below that hold for $\widecheck{\bold X}$ and $\slices{X}$ (e.g.~concentration inequalities), we use the placeholder $\bold Z = (Z_n)_{n = 1}^N$ for a generic sequence of random elements in $\R^k$, $k \geq 1$. With a slight abuse of notation, we write $\bold Z = (Z_n)_{n = 1}^N \sim \mu$ for $Z_n \sim \mu \text{ for any } n \in \{1, \ldots, N\}$, and $\operatorname{Law}(Z_1, \ldots, Z_N) \in \mathcal{P}((\R^k)^N)$ for the joint law of the sequence $\bold Z$.

For a probability measure $\mu \in \mathcal{P}(\R^k)$ we write $\mathcal{E}_{\alpha, \gamma}(\mu) = \int \exp(\gamma \|x\|^\alpha)\, \mu(dx)$. Similarly, for a random vector $X$ taking values in $\R^k$  we write $\mathcal{E}_{\alpha, \gamma}(X) = \mathbb{E}\, \exp(\gamma \|X\|^\alpha)$.

Lastly we write $C > 0$ for a constant which may increase from line to line; similarly we denote by $c > 0$ a constant which may decrease from line to line. Throughout we specify all non-trivial dependencies of these constants.

\section{Setting and preliminary observations}\label{sec:setting}

\subsection{The adapted Wasserstein distance}\label{sec:adapted_wasserstein}

In this section we recall the basic concepts from adapted optimal transport. 

\begin{definition}\label{def:wasserstein}
For two probability measures $\mu, \nu \in \mathcal{P}_1(\R^k)$, $k \geq 1$ we define the  Wasserstein distance as
$$
\mathcal{W}(\mu, \nu) = \inf_{\pi \in \Pi(\mu, \nu)} \int \|x - y\| \, \pi(dx, dy).
$$
\end{definition}

Throughout this paper we are interested in its modification induced by a specific subset of $\Pi(\mu, \nu)$.

\begin{definition}[cf. {\cite[Lemma 2.2]{bartl2021wasserstein}}]\label{def:bicausal}
Let $\pi \in \Pi(\mu, \nu)$ be a transport plan for some $\mu, \nu \in \mathcal{P}((\R^k)^M)$, $k, M \ge 1$. Then $\pi$ is called \textit{bicausal}, if $\pi_{x_{1:m}, y_{1:m}} \in \Pi(\mu_{x_{1:m}}, \nu_{y_{1:m}})$ almost everywhere with respect to $\pi(dx_{1:m}, dy_{1:m})$ for $m \in \{1, \ldots, M-1\}$ and $\pi^1 \in \Pi(\mu^1, \nu^1)$. Here $\pi_{x_{1:m}, y_{1:m}}$ is defined via the disintegration rule
\begin{equation}\label{eqn:disintegration_definition}
\pi(B_1 \times \ldots \times B_M) = \int_{B_1} \ldots \int_{B_M} \, \pi_{x_{1:M-1}, y_{1:M-1}}(dx_M, dy_M) \ldots \pi^1(dx_1, dy_1),
\end{equation}
for all $B_m \in \mathcal{B}(\R^k \times \R^k)$ for $m \in \{1, \ldots, M\}$, and similarly for $\mu, \nu$. 
We denote the set of bicausal couplings by $\Pi_{\operatorname{bc}}(\mu, \nu)$. 
\end{definition}

Definition \ref{def:bicausal} emphasizes the non-anticipativity of the bicausal couplings, meaning that given $(x_{1:m}, y_{1:m})$, one can only ``see" the conditional laws $\mu_{x_{1:m}}$ and $\nu_{y_{1:m}}$. 

\begin{definition}\label{def:adapted_wasserstein}
For two probability measures $\mu, \nu \in \mathcal{P}_1((\R^k)^M)$ we define the adapted Wasserstein distance as
$$
\mathcal{AW}(\mu, \nu) := \inf_{\pi \in \Pi_{\operatorname{bc}}(\mu, \nu)} \int \sum_{m = 1}^M \|x_m - y_m\| \, \pi(dx, dy).
$$
\end{definition}

To illustrate the difference between $\mathcal{W}$ and $\mathcal{AW}$, we consider the following classical example:
\begin{example}
Let $M = 2$ and $$\mu^N=\frac{1}{2}\delta_{(\epsilon_n, 1)}+\frac{1}{2}\delta_{(-\epsilon_n, -1)}\quad \text{and} \quad \mu=\frac{1}{2}\delta_{(0,1)}+\frac{1}{2}\delta_{(0,-1)}.$$ If $\epsilon_n\to 0$, then $\mathcal{W}(\mu, \mu^N)\to 0$ and $\mathcal{AW}(\mu, \mu^N)=\epsilon_n+1\to 1>0$. To see this, observe that for any bicausal coupling $\pi \in \Pi_{\operatorname{bc}}(\mu, \mu^N)$ one must have $\pi_{x_1 = \pm \varepsilon_n, y_1 = 0} \in \Pi(\delta_{\pm 1}, \frac{1}{2} \delta_1 + \frac{1}{2} \delta_{-1})$, and the corresponding transport cost is at least $1$.
\end{example}

The adapted Wasserstein distance $\mathcal{AW}$ metrizes the adapted weak topology; see e.g.~\cite{bartl2021wasserstein}.

\subsection{Adapted empirical measure}
Let $N \in \mathbb{N}$ and consider a sequence of random elements $\slices{X} = (X^n)_{n = 1}^N \sim \mu \in \mathcal{P}((\mathbb{R}^{d})^T)$ defined on some probability space $(\Omega, \mathcal{F}, \mathbb{P})$. We interpret $\slices{X}$ as a sequence of consecutive blocks of length $T > 1$ sampled from a time series taking values in $\mathbb{R}^d$, so that $T$ is the number of time steps and $d$ is the dimension of the state space (see Figure \ref{fig:1}). We give two examples.

\begin{example}
\begin{itemize}
    \item Let $T = 24$ and let $X^n_t \in \R^d$ denote a cumulative trading volume of a set of $d$ stocks on day $n$ in hour $t$. In this case $X^n$ constitutes a daily slice of hourly trading volumes.
    \item Let $T = 168$ and $X^n_t \in \R$ be an average temperature in \SI{}{\celsius} in week $n$ and hour $t$. In other words, $X^n$ is a temperature profile for a given week. 
\end{itemize}
\end{example}
We aim to estimate the finite dimensional distribution $\mu$ of $T$ consecutive steps given the observations $\slices{X}$. As mentioned in the Introduction, we cannot use the empirical measure $\mu^N = \frac{1}{N} \sum_{n = 1}^N \delta_{X^n}$ for this task, as $\lim_{N \to \infty} \mathcal{AW}(\mu, \mu^N) \neq 0$ in general. Below we formally describe the construction from \cite{backhoff2022estimating} which mitigates this issue by clustering observations.

We partition $\mathbb{R}^d$ uniformly into cubes with edges of length $\Delta_N := N^{-r}$, where
\begin{align}\label{eq:def_r}
r = r(d) := \begin{cases}
    \frac{1}{T + 1}, &d = 1,\\
    \frac{1}{dT}, &d \geq 2,
\end{cases}
\end{align}
and let $\varphi^N: \mathbb{R}^d \to \mathbb{R}^d$ be a Borel measurable function mapping each cube to its center. The central object of the paper is the \textit{adapted empirical measure} 
\begin{equation}\label{eqn:adapted_empirical_measure}
\widehat \mu^N := \frac{1}{N} \sum_{n = 1}^N \delta_{\varphi^N(X^n)},
\end{equation}
where we recall the convention $\varphi^N(X^n)=(\varphi^N(X_1^n), \dots, \varphi^N(X_T^n)).$
By definition, $\widehat \mu^N$ is the empirical measure of the observations $\varphi^N(\slices X) = (\varphi^N(X^n))_{n = 1}^N$. The projection $\varphi^N$ discretizes the sample paths, which allows for non-degenerate conditional distributions of $\widehat \mu^N$, as presented in Figure \ref{fig:empirical_vs_adapted}. The goal of the article is to study convergence rates for $\mathcal{AW}(\widehat \mu^N, \mu)$.

\begin{figure}[H]
    \centering
    \includegraphics[scale=0.5]{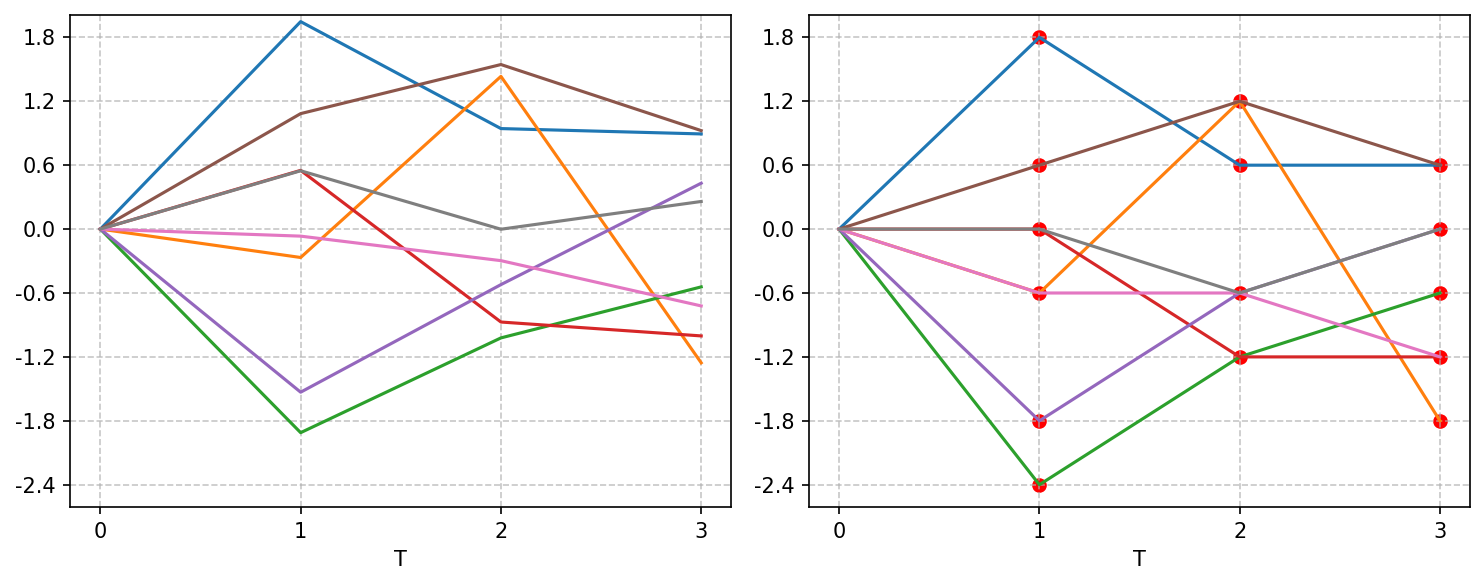}
    \caption{Sample paths of $\mu^N$ (left) and $\widehat \mu^N$ (right) in the case $d = 1$ and $T = 3$. Some trajectories are merged as their values belong to the same interval, hence the corresponding conditional distributions are non-degenerate. }
    \label{fig:empirical_vs_adapted}
\end{figure}

Convergence of $\widehat \mu^N$ is well studied for independent  observations $\slices{X}$, see \cite{backhoff2022estimating,acciaio2024convergence, blanchet2024bounding,hou2024convergence}. However, the i.i.d.~assumption is clearly not satisfied in real-world applications, where the values of the sequence $\slices{X}$ of observations are naturally affected by each other. Typically, this influence decreases as the distance between the observations grows. As is classical in statistics, we introduce a mixing condition to formalize this.

\subsection{Mixing conditions}\label{sec:mixing_conditions}
In this section we first provide a short introduction to the notion of mixing coefficients, and then specify to our setting.

Classically, \textit{mixing coefficients} control the dependence between two $\sigma$-algebras $\mathcal{A}$ and $\mathcal{B}$ for a fixed probability measure $\P$. Famous examples are
\begin{align}
\begin{split}\label{eq:examples}
&\alpha(\mathcal{A}, \mathcal{B}) = \sup_{A \in \mathcal{A}, B \in \mathcal{B}} |\mathbb{P}(A \cap B) - \mathbb{P}(A) \cdot \mathbb{P}(B)|,\\
&\phi(\mathcal{A}, \mathcal{B}) = \sup_{\substack{A \in \mathcal{A}, B \in \mathcal{B}\\\mathbb{P}(B) > 0}} |\mathbb{P}(A \mid B) - \mathbb{P}(A)|;
\end{split}
\end{align}
see \cite{bradley2005basic} for a survey and many other different measures of dependence. One then defines a mixing coefficient for a sequence $\bold Z = (Z_n)_{n = 1}^N$ of random elements by taking $\mathcal{A} = \mathcal{F}_{n+s}^N$ and $\mathcal{B} = \mathcal{F}_1^n$, where $\mathcal{F}_i^j := \sigma(\{Z_n\}_{n = i}^j)$. For \eqref{eq:examples}, this yields
\begin{align*}
&\alpha(s) = \sup_{1 \leq n < N - s + 1} \alpha(\mathcal{F}_{n+s}^N, \mathcal{F}_1^n),\\
&\phi(s) = \sup_{1 \leq n < N - s + 1} \phi(\mathcal{F}_{n+s}^N, \mathcal{F}_1^n).
\end{align*}
One calls $\bold Z$ $\alpha$-mixing if $\alpha(s) \to 0$ or $\phi$-mixing if $\phi(s) \to 0$, as $s \to \infty$. If $\{Z_n\}_{n = 1}^N$ are independent, then $\alpha(s) = \phi(s) = 0$. The mixing conditions above allow for an extension of many classical results in statistics  to non-i.i.d random elements, including a Central Limit Theorem and moment estimates; see \cite{rio2017asymptotic} for details.

One of the main goals in this paper is to derive concentration of measure type results; see e.g. \cite{azuma1967weighted,hoeffding1963probability, talagrand1995concentration, talagrand1996new, ledoux1997talagrand,samson2000concentration} and the references therein. Compared to the i.i.d~case, these are quite naturally more difficult to establish. More specifically we aim to derive an exponential concentration inequality of the form
$$
\mathbb{P}\left(|\varphi(Z_{1:N}) - \mathbb{E}\, \varphi(Z_{1:N})| > \varepsilon\right) \leq \exp(-f(N, \varepsilon))
$$
for some rate function $f(N, \varepsilon)$ and a target function $\varphi$ (which will a function of $\mathcal{AW}$ in our case). Below we will use the martingale differences method to establish such an inequality; 
we refer to \cite[Section 2.2.1-2.2.2]{wainwright2019high} and to the landmark papers \cite{marton1996measure,marton1996bounding, marton1998measure,chatterjee2005concentration,  kontorovich2014concentration, kontorovich2017concentration, kontorovich2008concentration} for a discussion of this approach.
In short, the idea is to decompose $\varphi(Z_{1:N})$ into a sum of martingale differences via the telescoping argument
\begin{align*}
    \varphi(Z_{1:N}) - \mathbb{E}\, \varphi(Z_{1:N}) =\sum_{n=1}^N \E[\varphi(Z_{1:N})|Z_{1:n}] -\E[\varphi(Z_{1:N})|Z_{1:n-1}]
\end{align*}
and then use continuity properties of $\varphi$ to bound the terms on the right hand side. As martingales rely on conditional means, the proof of such a result requires a mixing condition stronger than the $\alpha$-mixing or $\phi$-mixing conditions mentioned above. Indeed, we need to control the distance between the conditional laws $\operatorname{Law}(Z_{n+s} \mid Z_{1:n} = z_{1:n})$ and $\operatorname{Law}(Z_{n+s} \mid Z_{1:n-1} = z_{1:n-1})$ in a uniform sense --- this cannot be guaranteed by $\alpha(s)$ or $\phi(s)$.

Our mixing coefficient is derived from \cite{kontorovich2008concentration}, and is defined as follows:
\begin{definition}\label{def:eta_mixing}
Let $N \in \mathbb{N} \cup \{\infty\}$ and $\bold Z = (Z_n)_{n = 1}^N$ be a sequence of random elements taking values in $\R^k$ and defined on a probability space $(\Omega, \mathcal{F}, \mathbb{P})$. Define
\begin{align}\label{eqn:eta_mixing_definition}
\begin{split}
\eta_{\bold Z}(s) &:= \sup_{\substack{1 \leq n < N - s+1 \\A_{1:n} \in \mathcal{B}(\R^k)^{n}\\\mathbb{P}(Z_{1:n} \in A_{1:n}) > 0}} \operatorname{TV}(\operatorname{Law}(Z_{n+s} \mid Z_{1:n} \in A_{1:n}), \operatorname{Law}(Z_{n+s} \mid Z_{1:n-1} \in A_{1:n-1})),\\
\bar \eta_{\bold Z}(s) &:= \sup_{\substack{1 \leq n < N - s+1 \\A_{1:n} \in \mathcal{B}(\R^k)^{n}\\\mathbb{P}(Z_{1:n} \in A_{1:n}) > 0}} \operatorname{TV}(\operatorname{Law}(Z_{n+s:N} \mid Z_{1:n} \in A_{1:n}), \operatorname{Law}(Z_{n+s:N} \mid Z_{1:n-1} \in A_{1:n-1}))
\end{split}
\end{align}
for $1 \leq s < N$, where we recall that $\operatorname{TV}$ is the total variation distance and we make that convention that $\{n:\,1\le n<\infty\}:=\N$. We call $\eta_{\bold Z}(s)$ the $\eta$-mixing coefficient of $\bold Z$, and $\bar \eta_{\bold Z}(s)$ the tail $\eta$-mixing coefficient.
\end{definition}

Intuitively, the mixing coefficient $\eta_{\bold Z}(s)$ measures how the change of a single variable $Z_n$ influences the distribution of $Z_{n+s}$ in the worst case, given the history of the process $\bold Z$ up to time $(n\!-\!1)$. We emphasize that $\eta_{\bold Z}(s)$ controls the distance between conditional distributions. The reason is that $\operatorname{Law}(Z_{n+s:N} \mid Z_{1:n} = z_{1:n})$ can be approximated by $\operatorname{Law}(Z_{n+s:N} \mid Z_{1:n} \in B_\delta(z_{1:n}))$ for small $\delta > 0$; see Proposition \ref{prop:TV_bound_mixing} for details.

\begin{remark}\label{rem:kontorovich}
As a comparison, the $\eta$-mixing coefficient in \cite{kontorovich2008concentration} is defined as
\begin{equation}\label{eqn:kontorovich_ramanan_eta_mixing}
\widehat \eta_{\bold Z}(n, n+s) := \sup_{\substack{
z_{1:n} \in \mathcal{Z}^n, \; \tilde z_n \in \mathcal{Z}\\\mathbb{P}\left(Z_{1:n} = z_{1:n}\right) > 0\\ \mathbb{P}\left(Z_{1:n} = z_{1:n-1} \tilde z_n\right) > 0}} \operatorname{TV}(\operatorname{Law}(Z_{n+s:N} \mid Z_{1:n} = z_{1:n}), \operatorname{Law}(Z_{n+s:N} \mid Z_{1:n} = z_{1:n-1} \tilde z_n)),
\end{equation}
for $1 \leq n < N$ and $1 \leq s < N - s + 1$, where $\mathcal{Z}$ is a countable space. While similar in spirit, it differs from $\bar \eta_{\bold Z}$ both notationally and conceptually:
\begin{enumerate}
    \item We condition on events $A_{1:n} \in \mathcal{B}(\mathcal{Z})^n$ with $\mathbb{P}(Z_{1:n} \in A_{1:n}) > 0$ rather than on $\{Z_{1:n}=z_{1:n}\}$ only. In consequence, $\bar \eta_{\bold Z}(s)$ is larger. In fact, we emphasize that $\bar \eta_{\bold Z}$ is not equivalent to $\widehat \eta_{\bold Z}$, as shown in Example \ref{ex:difference}. 
    \item We compare conditional probabilities for $Z_{1:n} \in A_{1:n}$ and $Z_{1:n-1} \in A_{1:n-1}$. Equation \eqref{eqn:kontorovich_ramanan_eta_mixing} instead compares two histories of $\bold Z$ that agree up to time $(n\!-\!1)$ and differ at time $n$. In our definition this would correspond to comparing $\operatorname{Law}(Z_{n+s:N} \mid Z_{1:n} \in A_{1:n})$ and $\operatorname{Law}(Z_{n+s:N} \mid Z_{1:n} \in A_{1:n-1}B_n)$ for $A_{1:n} \in \mathcal{B}(\R^k)^n$ and $B_n \in \mathcal{B}(\R^k)$. We adopt the former notation in \eqref{eqn:eta_mixing_definition} as it is more directly aligned with martingale arguments. However, it is important to point out that this difference is simply of a notational nature: indeed, for any $A_{1:n-1} \in \mathcal{B}(\R^k)^{n-1}$,
    \begin{align*}
    &\sup_{\substack{A_n \in \mathcal{B}(\R^k)\\\mathbb{P}(Z_{1:n} \in A_{1:n}) > 0}}\operatorname{TV}\left(\operatorname{Law}(Z_{n+s:N} \mid Z_{1:n} \in A_{1:n}), \operatorname{Law}(Z_{n+s:N} \mid Z_{1:n-1} \in A_{1:n-1})\right)\\
    &\leq \sup_{\substack{A_n, B_n \in \mathcal{B}(\R^k)\\\mathbb{P}(Z_{1:n} \in A_{1:n}) > 0\\\mathbb{P}(Z_{1:n} \in A_{1:n-1}B_n) > 0}} \operatorname{TV}\left(\operatorname{Law}(Z_{n+s:N} \mid Z_{1:n} \in A_{1:n}), \operatorname{Law}(Z_{n+s:N} \mid Z_{1:n} \in A_{1:n-1}B_n)\right)\\
    &\leq 2 \sup_{\substack{A_n \in \mathcal{B}(\R^k)\\\mathbb{P}(Z_{1:n} \in A_{1:n}) > 0}} \operatorname{TV}\left(\operatorname{Law}(Z_{n+s:N} \mid Z_{1:n} \in A_{1:n}), \operatorname{Law}(Z_{n+s:N} \mid Z_{1:n} \in A_{1:n-1})\right),
    \end{align*}
    where the first inequality follows from taking $B_n = \R^k$ and the second follows from the triangle inequality for $\operatorname{TV}$. 
    \item The coefficient  $\widehat \eta_{\bold Z}(n,n+s)$ depends on $n$, while the tail mixing coefficient $\bar \eta_{\bold Z}(s)$ in \eqref{eqn:eta_mixing_definition} only depends on $s$. However, if $\bold Z$ is time-homogeneous, the conditional tail law $\operatorname{Law}(Z_{n+s:N} \mid Z_{1:n})$ is shift-invariant and hence does not depend on $n$ either.
\end{enumerate}
\end{remark}

Defined in this way, the $\eta$-mixing coefficient $\eta_{\bold Z}$ satisfies an information processing inequality (see Proposition \ref{prop:eta_mixing_information_processing} below), that is central for the results in this paper. Indeed, as can be seen from \eqref{eqn:adapted_empirical_measure}, the adapted empirical measure is defined in terms of the projections $\varphi^N(\slices{X}) = (\varphi^N(X^n))_{n = 1}^N$ instead of the original observations $\slices{X} = (X^n)_{n = 1}^N$, and we thus require the mixing properties of $\slices{X}$ to propagate to $\varphi^N(\slices{X})$.

We illustrate Definition \ref{def:eta_mixing} with the following results:

\begin{proposition}\label{prop:independent}
$\bold Z = (Z_n)_{n = 1}^N$ are independent if and only if $\eta_{\bold Z}(s) = 0$ for $s < N$.
\end{proposition}
\begin{proof}
The ``$\Rightarrow$" part is trivial. To prove ``$\Leftarrow$", suppose that $\eta_{\bold Z}(s) = 0$ for all $1\le s < N$. Then
\begin{align*}
\operatorname{Law}(Z_{n+1} \mid Z_{1:n} \in A_{1:n}) &\xlongequal{\eta_{\bold Z}(1) = 0} \operatorname{Law}(Z_{n+1} \mid Z_{1:n-1} \in A_{1:n-1})\\
&\xlongequal{\eta_{\bold Z}(2) = 0} \operatorname{Law}(Z_{n+1} \mid Z_{1:n-2} \in A_{1:n-2})\\
&\cdots\\
&\xlongequal{\eta_{\bold Z}(n) = 0} \operatorname{Law}(Z_{n+1})
\end{align*}
for all $n < N$ and $A_{1:n} \in \mathcal{B}(\R^k)^{n}$, such that $\mathbb{P}(Z_{1:n} \in A_{1:n}) > 0$, and hence $Z_{n+1:N}$ is independent of $(Z_1, \ldots, Z_n)$. This concludes the proof. 
\end{proof}

\begin{remark}\label{rem:markov}
Let $\bold Z$ be a stationary time-homogenenous Markov process taking values in $\R^k$ with transition kernel $K$, meaning that
$$
\operatorname{Law}(Z_{n+1} \mid Z_n = z_n) = K(\cdot \mid z_n)
$$
for $z_n \in \R^k$. Then Definition \ref{def:eta_mixing} simplifies to 
\begin{align}\label{eqn:eta_mixing_markov}
\begin{split}
\eta_{\bold Z}(s) &= \sup_{\substack{1\le n<N-s+1 \\A_{n-1:n} \in \mathcal{B}(\R^k)\\\mathbb{P}(Z_{n-1:n} \in A_{n-1:n}) > 0}} \operatorname{TV}\left(\operatorname{Law}(Z_{n+s} \mid Z_n \in A_n), \operatorname{Law}(Z_{n+s} \mid Z_{n-1} \in A_{n-1})\right)\\
&= \sup_{\substack{A, B \in \mathcal{B}(\R^k)\\ \mathbb{P}(Z_1 \in A) > 0,\; \mathbb{P}(Z_1 \in B) > 0}} \operatorname{TV}\left(K^{(s)}(\cdot \mid A), K^{(s+1)}(\cdot \mid B)\right),
\end{split}
\end{align}
where 
$K^{(s)}(\cdot \mid A) = \operatorname{Law}(Z_{s+1} \mid Z_1 \in A)$ for $A \in \mathcal{B}(\R^k)$ with $\mathbb{P}(Z_1 \in A) > 0$.
\end{remark}

Notably, the asymptotic behavior of $\eta_{\bold Z}(s)$ is linked to classical notions of mixing and ergodicity of Markov processes, see \cite{seneta2006non} and \cite{dobrushin1956central}.

\begin{example}
Recall the definition of $K^{(s)}$ from above and set $K:=K^{(1)}$. Assume that $\bold Z$ is uniformly ergodic (see \cite{roberts2004general} for equivalent definition up to a factor 2), meaning that for $s \in \mathbb{N}$
\begin{equation}\label{eqn:remark_uniformly_ergodic}
\sup_{z_1, z_2 \in \R^k} \operatorname{TV}(K^{(s)}(\cdot \mid z_1), K^{(s)}(\cdot \mid z_2)) \leq C \cdot \rho^{s},
\end{equation}
where $C > 0$ is an absolute constant and $\rho \in (0, 1)$. Then it follows from the semigroup property of $K$ and convexity of $\operatorname{TV}$ (see \cite[Theorem 4.8]{villani2009optimal}) that
\begin{align*}
\eta_{\bold Z}(s) &\overset{\eqref{eqn:eta_mixing_markov}}{=} \sup_{\substack{A, B \in \mathcal{B}(\R^k)\\ \mathbb{P}(Z_1 \in A) > 0,\; \mathbb{P}(Z_1 \in B) > 0}} \operatorname{TV}\left(K^{(s)}(\cdot \mid A), K^{(s+1)}(\cdot \mid B)\right)\\
&\leq \sup_{\substack{A, B \in \mathcal{B}(\R^k)\\ \mathbb{P}(Z_1 \in A) > 0,\; \mathbb{P}(Z_1 \in B) > 0}} \int \operatorname{TV}\left(K^{(s)}(\cdot \mid A), K^{(s)}(\cdot \mid z_2)\right) K(dz_2 \mid B) \\
&\leq \sup_{z_1, z_2 \in \R^k} \operatorname{TV}(K^{(s)}(\cdot \mid z_1), K^{(s)}(\cdot \mid z_2))\\
&\overset{\eqref{eqn:remark_uniformly_ergodic}}{\leq} C \cdot \rho^{s}.
\end{align*}
Hence, uniform ergodicity of a stationary Markov process implies exponential decay of $\eta_{\bold Z}(s)$.
\end{example}

\begin{example}[Processes with seasonality]
Define $\widecheck{\bold X} = (X_n)_{n = 1}^\infty$ taking values in $\{-1, 0, 1\}$ as follows:
\begin{equation}\label{eqn:seasonality_example}
X_0 = 0, \quad X_{n+1} = \mathds{1}_{\{B_n=0\}} \cdot X_n + \mathds{1}_{\{B_n=1\}} \cdot \varepsilon_{n-\tau} + \mathds{1}_{\{B_n=2\}} \cdot \varepsilon_n
\end{equation}
for $n \geq 0$, where $(\varepsilon_n)_{-\infty}^\infty$ is a sequence of i.i.d.~trinomial random variables, $(B_n)_{n = 0}^\infty$ are i.i.d.~ and independent of $(\varepsilon_n)_{-\infty}^\infty$ with distribution $\rho \cdot \delta_{0} + \Theta \cdot \delta_1 + (1 - \rho - \Theta) \cdot \delta_2$ for memory and seasonality parameters $\rho, \Theta \in (0, 1)$ satisfying $\rho + \Theta < 1$. The parameter $\tau \in \N$ is a seasonality lag.

Such a setup is standard, for instance, in stochastic weather generators, where rainfall occurrence is modeled by a Markov chain (e.g.\ dry/wet) and seasonality is modeled by accounting for a seasonal cycle, see \cite{ailliot2015stochastic} and \cite[Section 6.5]{brockwell2002introduction}.
Define the sequence of pairs
$$
X^n := X_{\tau \cdot n + 1: \tau \cdot n + 2},
$$
which represent the weather's distribution at the start of each period. Observe that for all $n, s \in \N$, $X^{n+s} \perp\!\!\!\perp X^{1:n}$ conditionally on the event $\{\max(B_{\tau \cdot n + 2: \tau \cdot (n+s)}) = 2\}$. This happens with probability $1 - (\rho + \Theta)^{\tau \cdot s - 1}$. Hence, for all $n, s \in \N$ and $A \in \mathcal{B}((\R^2)^n)$ with $\mathbb{P}(X^{1:n} \in A) > 0$,
\begin{align}\label{eqn:seasonality_example_prob}
\begin{split}
\mathbb{P}\!\left(X^{n+s} \mid X^{1:n} \in A\right) &= (1 - (\rho + \Theta)^{\tau \cdot s - 1}) \cdot \mathbb{P}\!\left(X^{n+s} \in S \mid \max(B_{\tau \cdot n + 2: \tau \cdot (n+s)}) = 2\right)\\
&\quad + (\rho + \Theta)^{\tau \cdot s - 1} \cdot \mathbb{P}\!\left(X^{n+s} \in S \mid X^{1:n} \in A, \; \max(B_{\tau \cdot n + 2: \tau \cdot (n+s)}) < 2\right),
\end{split}
\end{align}
and consequently
\begin{align*}
\eta_{\slices X}(s) &\leq \sup_{\substack{1 \leq n < N-s+1\\ A_{1:n} \in \mathcal{B}(\R^2)^n\\ \mathbb{P}(X^{1:n} \in A_{1:n}) > 0}} \operatorname{TV}\!\left(\operatorname{Law}(X^{n+s} \mid X^{1:n} \in A_{1:n}), \operatorname{Law}(X^{n+s} \mid X^{1:n-1} \in A_{1:n-1}\right)\\
&\overset{\eqref{eqn:seasonality_example_prob}}{\leq} 2 (\rho + \Theta)^{\tau \cdot s - 1}.
\end{align*}
\end{example}

We now state two immediate consequences of Definition \ref{def:eta_mixing}. Next to the information processing inequality mentioned above, $\eta_{\bold Z}$ also yields a bound on the covariance of elements in $\bold Z.$
To state these results formally, let $N \in \mathbb{N} \cup \{\infty\}$ and $\bold Z = (Z_n)_{n = 1}^N$ be a stochastic process taking values in $\R^k$, $k\ge 1$. Denote the $\eta$-mixing coefficients of $\bold Z$ by $\eta_{\bold Z}(s)$ and $\bar \eta_{\bold Z}(s)$.

\begin{proposition}[Information processing inequality]\label{prop:eta_mixing_information_processing}
For $l \geq 1$ let $\Phi: \R^k \to \R^l$ be a Borel measurable function. Then the $\eta$-mixing coefficients $\eta_{\Phi(\bold Z)}(s)$ and $\bar \eta_{\Phi(\bold Z)}(s)$ of the sequence $\Phi(\bold Z) := (\Phi(Z_n))_{n = 1}^N$ satisfy
\begin{align*}
\eta_{\Phi(\bold Z)}(s) \leq \eta_{\bold Z}(s) \;\; \text{and} \;\; \bar \eta_{\Phi(\bold Z)}(s) \leq \bar \eta_{\bold Z}(s).
\end{align*}
\end{proposition}

\begin{proof}
We present the proof for $\eta_{\Phi(\bold Z)}(s)$ and fix some $1 \leq s < N$. Let $1 \leq n < N - s + 1$, take $A_{1:n} \in \mathcal{B}((\R^l)^n)$ and $B \in \mathcal{B}(\R^l)$. Assume that $\mathbb{P}(\Phi(Z_{1:n}) \in A_{1:n}) > 0$. Then
\begin{align*}
&\mathbb{P}\left(\Phi(Z_{n+s}) \in B \mid  \Phi(Z_{1:n}) \in A_{1:n}\right) - \mathbb{P}\left(\Phi(Z_{n+s}) \in B \mid  \Phi(Z_{1:n-1}) \in A_{1:n-1}\right)\\
&= \mathbb{P}\left(Z_{n+s} \in \Phi^{-1}(B) \mid  Z_{1:n} \in \Phi^{-1}(A_{1:n})\right) - \mathbb{P}\left(Z_{n+s} \in \Phi^{-1}(B) \mid  Z_{1:n-1} \in \Phi^{-1}(A_{1:n-1})\right)\\
&\leq \operatorname{TV}\left(\operatorname{Law}(Z_{n+s} \mid Z_{1:n} \in \Phi^{-1}(A_{1:n})), \operatorname{Law}(Z_{n+s} \mid Z_{1:n-1} \in \Phi^{-1}(A_{1:n-1}))\right)\\
&\leq \eta_{\bold Z}(s),
\end{align*}
where we have used that $\Phi^{-1}(A_{1:n}) \in \mathcal{B}((\R^k)^n)$ and $\Phi^{-1}(B) \in \mathcal{B}(\R^k)$ since $\Phi$ is Borel measurable; additionally $\mathbb{P}(Z_{1:n} \in \Phi^{-1}(A_{1:n})) > 0$ by the choice of $A_{1:n}$. Taking the supremum over $B \in \mathcal{B}(\R^l)$ and then $1 \leq n < N - s + 1, \; A_{1:n} \in \mathcal{B}(\R^l)^n$ we arrive at the desired conclusion $\eta_{\Phi(\bold Z)}(s) \leq \eta_{\bold Z}(s)$. The inequality $\bar \eta_{\Phi(\bold Z)}(s) \leq \bar \eta_{\bold Z}(s)$ is obtained using precisely the same arguments.
\end{proof}

\begin{proposition}[Covariance bound]\label{prop:eta_mixing_covariance}
Let $f: \R^k \to [0, +\infty)$ be a bounded Borel measurable function. Then for any $1 \leq j < i \leq N$ we have
$$
\operatorname{Cov}\,(f(Z_i), f(Z_j)) \leq \eta_{\bold Z}(i - j) \cdot \|f\|_\infty \cdot \mathbb{E}\, f(Z_j).
$$
\end{proposition}
\begin{proof}
Denote the $\eta$-mixing coefficient of $f(\bold Z) := (f(Z_n))_{n = 1}^N$ by $\eta_{f(\bold Z)}$. Applying \eqref{eqn:eta_mixing_definition} with $s = i - j, \; n = j, \; A_{1:n} = \mathbb{R}^{n-1} \{f(Z_j)\}$ we obtain
\begin{align}\label{eqn:eta_mixing_covariance_start}
\begin{split}
\mathbb{E}\,[f(Z_i) \mid f(Z_j)] - \mathbb{E}\,[f(Z_i)]
&\leq \|f\|_\infty \cdot \operatorname{TV}\,(\operatorname{Law}(f(Z_i) \mid f(Z_j)), \operatorname{Law}(f(Z_i)))\\
&\leq \|f\|_\infty \cdot \eta_{f(\bold Z)}(i - j)\\
&\leq \|f\|_\infty \cdot \eta_{\bold Z}(i - j),
\end{split}
\end{align}
where the first inequality follows from
$$
\int f(x) \, \mu(dx) - \int f(x) \, \nu(dx) \leq  \|f\|_\infty \cdot \operatorname{TV}\,(\mu, \nu)
$$
for any probability distributions $\mu, \nu \in \mathcal{P}(\R^k)$, and the final inequality is due to Proposition \ref{prop:eta_mixing_information_processing}. Multiplying both sides of \eqref{eqn:eta_mixing_covariance_start} by $f(Z_j)$ we obtain
$$
\mathbb{E}\,[f(Z_i) f(Z_j) \mid f(Z_j)] - \mathbb{E}\,[f(Z_i)] \cdot f(Z_j) \leq \|f\|_\infty \cdot \eta_{\bold Z}(|i - j|) \cdot f(Z_j).
$$
Taking the expectation on both sides we arrive at the desired conclusion.
\end{proof}

\section{Main results} \label{sec:main_results}

\subsection{Bounds on $\E\left[\mathcal{AW}_p(\mu, \widehat\mu^N)\right]$} \label{sec:main_expectation}

We are now in a position to present the main results of the paper. For the moment bounds on $\mathcal{AW}_p(\mu, \widehat\mu^N)$ we need the following definition.

\begin{definition}\label{def:Lipschitz_kernels}
Let $\mu \in \mathcal{P}((\R^d)^T)$. We say that $\mu$ has $L$-Lipschitz kernels, if there exists a version of the disintegration $x_{1:t} \mapsto \mu_{x_{1:t}}(dx_{t+1}) \in \mathcal{P}(\R^d)$ that is $L$-Lipschitz with respect to $\mathcal{W}$, for all $t \in \{1, \ldots, T-1\}$. Similarly, we say that $\mu$ has continuous kernels if $x_{1:t} \mapsto \mu_{x_{1:t}}(dx_{t+1})$ admits $\mathcal{W}$-continuous version.
\end{definition}

We now state the moment bound for compactly supported case measures.

\begin{theorem}\label{thm:moment_estimate_compact}
Let $\mu \in \mathcal{P}((\mathbb{R}^d)^T)$ be compactly supported with $L$-Lipschitz kernels. Then
\begin{equation}\label{eqn:moment_estimate_compact}
\mathbb{E}\,\mathcal{AW}(\mu, \widehat \mu^N) \leq C \cdot \sqrt{1 + 2 \sum_{s = 1}^{N-1} \eta_{\slices{X}}(s)} \cdot \operatorname{rate}_{\infty}(N),
\end{equation}
where $C > 0$ depends on $L$, $d$, $T$ and $\operatorname{spt}(\mu)$ and
\begin{equation}\label{eqn:moment_estimate_compact_rate}
\operatorname{rate}_{\infty}(N) := \begin{cases}
N^{- 1 / (T + 1)}, &d = 1,\\
N^{- 1 / (2T)} \log (N + 1), &d = 2,\\
N^{- 1 / (dT)}, &d \geq 3.
\end{cases}
\end{equation}
\end{theorem}

\begin{remark}
We observe the influence of the $\eta$-mixing coefficient of $\slices{X}$ on the rate of convergence of $\mathcal{AW}(\mu, \widehat \mu^N)$. For instance, if $\sum_{s = 1}^\infty \eta_{\slices{X}}(s) < \infty$, then the convergence rate is the same as in the independent case, and coincides with \cite[Theorem 1.5]{backhoff2020adapted}. We recall from Proposition \ref{prop:independent} that $\eta_{\slices{X}}(s) = 0$ if $\slices{X}$ are independent.
\end{remark}

To extend this result to the non-compact case, we make a structural assumption on the kernels of $\mu$; see \cite[Definition 2.8]{acciaio2024convergence}.

\begin{definition}\label{def:kernel_decomposition} 
Fix $p \ge 1$, $r \ge 0$, $\alpha > 0$, $\gamma > 0$. We say that $\mu \sim \mathcal{P}((\R^d)^T)$ has \emph{$B$-uniform $p$-th moment (resp. $\mathcal{E}_{\alpha, \gamma}$ moment) noise kernel of growth $r$}, if for any $t \in \{0, \ldots, T-1\}$
$$
\mu_{x_{1:t}} = \operatorname{Law}(f_t(x_{1:t}) + \sigma_t(x_{1:t}) \cdot \varepsilon(x_{1:t})) \qquad \mu(dx_{1:t})\text{-almost everywhere}
$$
for some functions $f_t: (\R^d)^t  \to \R^d$, $\sigma_t: (\R^d)^t \to \mathbb{R}$ and $\varepsilon: (\R^d)^t \to L^p(\Omega, \mathcal{F}, \mathbb{P})$, such that
$$
|\sigma_t(x_{1:t})| \leq B \cdot \|x_{1:t}\|^r \quad \text{and} \quad \sup_{x_{1:t}} \mathbb{E}\,\|\varepsilon(x_{1:t})\|^p \leq B \quad (\text{resp.}\;\sup_{x_{1:t}} \mathcal{E}_{\alpha, \gamma}(\varepsilon(x_{1:t})) \leq B);
$$
for the case $m = 0$ we set $\mu_{x_{1:0}} := \mu^1(dx_1)$, and the corresponding functions $f_0, \sigma_0$ are scalars. \\
We call $x_{1:t} \mapsto \varepsilon(x_{1:t})$ the noise component of $\mu$, and $f_t$ the drift component of $\mu$. 
\end{definition}

When bounding $\E \, \mathcal{W}(\mu,\mu^N)$, it is classically assumed that $\mu$ is $p$-integrable for some $p \in [1, \infty)$; cf.~\cite{fournier2015rate}. The resulting rates then depend on $p$. In the case of the adapted Wasserstein distance, our sample complexity bounds require additional control of the moments of the kernels of $\mu$. To account for this, we make a stronger integrability assumption on the kernels via Definition \ref{def:kernel_decomposition}:

\begin{assumption}\label{asm:uniform_noise_moment}
Let $p\ge 1$, $r \geq 0$ and
$$
q > \begin{cases}
    \max\left(\tfrac{d+1}{d} (r + T - 1), rp + (d + 1) (T - 1) (p - 1)\right),  &d = 1, 2,\\
    \max\big(\tfrac{d}{d-1} (r + T - 1), rp + d (T - 1) (p - 1)\big), &d \geq 3.
\end{cases}
$$
The measure $\mu \in \mathcal{P}_q((\mathbb{R}^d)^T)$ has $L$-Lipschitz kernels and $B$-uniform $p$-th moment noise kernel of growth $r$ with the drift component $f_t(x_{1:t}) := \int x_{t+1}\, \mu_{x_{1:t}}(dx_{t+1})$ for $t \in \{1, \ldots, T-1\}$ and $f_0 := \int x_1 \, \mu^1(dx_1)$.
\end{assumption}

This assumption is a reminiscent of \cite[Setting 2.14]{acciaio2024convergence}, with the difference that the drift component $f_t$ is provided explicitly. We note that if $\mu$ has $L$-Lipschitz kernels, then $f_t$ is $L$-Lipschitz as well, as
$$
\Big|\int x_{t+1} \, \mu_{x_{1:t}}(dx_{t+1}) - \int x_{t+1} \, \mu_{\tilde x_{1:t}}(dx_{t+1})\Big| \leq \mathcal{W}_1(\mu_{x_{1:t}}, \mu_{\tilde x_{1:t}}) \leq L \cdot \|x_{1:t} - \tilde x_{1:t}\|
$$
by the Kantorovich-Rubinstein duality. Assumption \ref{asm:uniform_noise_moment} is satisfied for many classical pricing models, including the Black-Scholes and Heston model with clipped volatility; see \cite[Examples 2.11, 2.12, Remark 2.13]{acciaio2024convergence}.

\begin{theorem}\label{thm:rate_general}
Under Assumption \ref{asm:uniform_noise_moment} with $p \in [1, \infty)$ we have
\begin{equation}\label{eqn:rate_general}
\mathbb{E}\,\mathcal{AW}(\mu, \widehat \mu^N) \leq C \cdot \sqrt{1 + 2 \sum_{s = 1}^{N-1} \eta_{\slices{X}}(s)} \cdot \operatorname{rate}_p(N),
\end{equation}
where $C > 0$ depends on $d, T, p, q, B, L$ and
\begin{equation}\label{eqn:p_rate_function}
\operatorname{rate}_p(N) := N^{-\frac{p-1}{pT}} + \begin{cases}
    N^{-\frac{1}{(d+1)T}}, &d = 1, 2,\\
    N^{-\frac{1}{dT}}, &d \geq 3.
\end{cases}
\end{equation}
\end{theorem}

\begin{remark}
If $\sum_{s = 1}^\infty \eta_{\slices{X}}(s) < \infty$, then we recover \cite[Theorem 2.16]{acciaio2024convergence}.
\end{remark}

\subsection{Concentration inequalities for $\mathcal{AW}(\mu, \widehat{\mu}^N)$}\label{sec:main_concentration}
We now state exponential concentration inequalities for $\mathcal{AW}(\mu, \widehat \mu^N)$. These are derived from the Bounded Differences inequality for mixing sequences, that we prove in Section \ref{sec:concentration}. As before, we start with the compact case.

\begin{theorem}\label{thm:AW_concentration_compact}
Assume that $\mu \in \mathcal{P}((\mathbb{R}^d)^T)$ is compactly supported. Then
\begin{equation}\label{eqn:AW_concentration_compact}
\mathbb{P}\left(\left|\mathcal{AW}(\mu, \widehat \mu^N) - \mathbb{E}\, \mathcal{AW}(\mu, \widehat \mu^N)\right| > \varepsilon\right) \leq 2 \exp\left(-c \cdot N \cdot \frac{\varepsilon^2}{\operatorname{diam}(\operatorname{spt}(\mu))^2 \cdot (1 + \sum_{s = 1}^{N-1} \bar \eta_{\slices X}(s))^2}\right),
\end{equation}
where $c > 0$ depends on $T$.
\end{theorem}

\begin{remark}
If $\sum_{s = 1}^\infty \bar \eta_{\slices X}(s) < \infty$, then Theorem \ref{thm:AW_concentration_compact} recovers the subgaussian rate $\exp(-cN \varepsilon^2)$, coinciding with \cite[Theorem 1.7]{backhoff2022estimating}. This includes the case when $\slices X$ are independent. If $\bar \eta_{\slices X}(s) \leq C \cdot s^{-\beta}$ for $C > 0$ and $\beta \in (0, 1)$, then the concentration bound is $\exp(-c N^{2 \beta - 1} \varepsilon^2)$.
\end{remark}

We now turn to the general case. 

\begin{theorem}\label{thm:AW_concentration_general}
Let $\slices{X} = (X^n)_{n = 1}^N \sim \mu \in \mathcal{P}_1((\mathbb{R}^d)^T)$ be a sequence of random elements. Assume that $\mathcal{E}_{\alpha, \gamma}(\mu) < \infty$ for some $\alpha, \gamma > 0$. Then for any $\varepsilon \geq \Delta_N$
\begin{align}\label{eqn:AW_concentration_general}
\begin{split}
&\mathbb{P}\!\left(\left|\mathcal{AW}(\mu, \widehat \mu^N) - \mathbb{E}\left[\mathcal{AW}(\mu, \widehat \mu^N)\right]\right| > \varepsilon\right)\\
&\leq 2 \exp\!\left(-\frac{c}{(1 + \sum_{s = 1}^{N-1} \bar \eta_{\slices X}(s))^2} N^\frac{\alpha}{\alpha+2} \varepsilon^\frac{2\alpha}{\alpha+2}\right) + \mathcal{E}_{\alpha, \gamma}(\mu) \cdot N\exp\!\left(-c \cdot N^\frac{\alpha}{\alpha+2} \varepsilon^\frac{2\alpha}{\alpha+2}\right),
\end{split}
\end{align}
where $C, c > 0$ depend on $\alpha, \gamma, \mathcal{E}_{\alpha, \gamma}(\mu), d, T$.
\end{theorem}

\begin{remark}
The constraint $\varepsilon \geq \Delta_N$ is not restrictive. In particular, under the conditions of Theorem \ref{thm:rate_general} one gets
\begin{align*}
&\mathbb{P}\!\left(\mathcal{AW}(\mu, \widehat \mu^N) > C \cdot \Bigl(1 + 2 \sum_{s = 1}^{N-1} \eta_{\slices X}(s)\Bigr)^{1/2} \operatorname{rate}_p(N) + \varepsilon\right)\\
&\leq 2 \exp\left(-\frac{c}{(1 + \sum_{s = 1}^{N-1} \bar \eta_{\slices X}(s))^2} N^\frac{\alpha}{\alpha+2} \varepsilon^\frac{2\alpha}{\alpha+2}\right) + \mathcal{E}_{\alpha, \gamma}(\mu) \cdot N\exp\left(-c \cdot N^\frac{\alpha}{\alpha+2} \varepsilon^\frac{2\alpha}{\alpha+2}\right)
\end{align*}
for any $\varepsilon > 0$, as $\Delta_N \leq \operatorname{rate}_p(N)$ and hence will be absorbed by the sample complexity bound.
\end{remark}

\begin{remark}
If $\sum_{s = 1}^\infty \bar \eta_{\slices X}(s) < \infty$, then the second term in \eqref{eqn:AW_concentration_general} dominates the first by a factor of $N$, and we obtain
$$
\mathbb{P}\left(\left|\mathcal{AW}(\mu, \widehat \mu^N) - \mathbb{E}\left[\mathcal{AW}(\mu, \widehat \mu^N)\right]\right| > \varepsilon\right) \leq 3 \mathcal{E}_{\alpha, \gamma}(\mu) \cdot N\exp\left(-c \cdot N^\frac{\alpha}{\alpha+2} \varepsilon^\frac{2\alpha}{\alpha+2}\right).
$$
\end{remark}

\subsection{$\mathcal{AW}$-consistency of $\widehat{\mu}^N$}

We now discuss consistency of $\widehat{\mu}^N$. We start with the compact case.

\begin{theorem}\label{thm:consistency_general}
Let $\mu \in \mathcal{P}((\mathbb{R}^d)^T)$ be compactly supported and let $\slices{X} = (X^n)_{n = 1}^\infty \sim \mu$ satisfy $\sum_{s = 1}^\infty \bar \eta_{\slices X}(s) < \infty$.
\begin{enumerate}
    \item If $\mu$ has continuous kernels (see Def. \ref{def:Lipschitz_kernels}), then
    $$
    \lim_{N \to \infty} \mathcal{AW}(\mu, \widehat \mu^N) = 0 \qquad \mathbb{P}\text{-almost surely}.
    $$
    \item We have
    $$
    \lim_{N\to \infty} \mathcal{AW}(\mu, \widehat \mu^N \ast \mathcal{N}_{\sigma_N}) = 0 \qquad\mathbb{P}\text{-almost surely}
    $$
    for $\sigma_N := \max(\sqrt{\Delta_N}, N^{-\frac{1}{8}})$ and $\mathcal{N}_\sigma := \mathcal{N}(0, \sigma^2 I_{dT})$. 
\end{enumerate}
\end{theorem}\

As mentioned in the Introduction, Theorem \ref{thm:consistency_general} distinguishes two cases: if the disintegrations of $\mu$ are continuous, then the proof of \cite[Theorem 1.3]{backhoff2022estimating} can be extended using the moment estimates and concentration inequalities established in Section \ref{sec:main_expectation} and \ref{sec:main_concentration} for mixing sequences. For general $\mu$, we cannot follow the methodology of  \cite{backhoff2022estimating} directly. In particular, an application of Lusin's theorem to construct approximating measures with continuous kernels seems infeasible. For that reason, we instead convolve the adapted empirical measure with a Gaussian kernel and resort to a different method, similarly to \cite{hou2024convergence}. In fact, the distance between Gaussian-smoothed distributions can be estimated using a covariance bound which follows from Proposition \ref{prop:eta_mixing_covariance}. In either case, we require that the tail $\eta$-mixing coefficients decay fast enough.

We extend Theorem \ref{thm:consistency_general} to the non-compact case in the same way as \cite[Theorem 2.7]{acciaio2024convergence} via compact approximation. To achieve almost-sure convergence, we strengthen the decay assumption on the tail $\eta$-mixing coefficient and the integrability assumption on $\mu$; this allows to derive a version of the Strong Law of Large Numbers necessary in the approximation argument.

\begin{theorem}\label{thm:consistency_general_noncompact}
Let $\mu \in \mathcal{P}_p((\mathbb{R}^d)^T)$ for some $p > 1$ and let $\slices{X} = (X^n)_{n = 1}^\infty \sim \mu$ satisfy
$$
\sum_{s = 1}^\infty \sqrt{\sum_{k = 2^s}^\infty \bar \eta_{\slices X}(k)} < \infty.
$$
Then \textit{(1)} and \textit{(2)} from Theorem \ref{thm:consistency_general} still hold.
\end{theorem}

\begin{remark}
The condition $\sum_{s = 1}^\infty \sqrt{\sum_{k = 2^s}^\infty \bar \eta_{\slices X}(k)} < \infty$ does not follow from $\sum_{s = 1}^\infty \bar \eta_{\slices X}(s) < \infty$. However, it holds, for instance, when $\bar \eta_{\slices X}(s) \leq C \cdot s^{-\alpha}$ for some $\alpha > 1$. 
\end{remark}

\section{Numerical experiments: Estimating $\mathcal{AW}(\mu, \widehat \mu^N)$ for $T = 2$} \label{sec:numerics}

In this section we study bounds on $\mathcal{AW}(\mu, \widehat \mu^N)$ for $T = 2$ numerically. We are particularly interested in showing that an increase of the time lag (and thus a change of the dependence structure) between subsequent observations changes the distribution of $\mathcal{AW}(\mu, \widehat \mu^N)$; cf. Figure \ref{fig:1}.

For this we define $\widecheck{\bold X} := (X_n)_{n = 0}^\infty$ as follows:
$$
X_{n+1} = X_n \cdot B_n + \varepsilon_n \cdot (1 - B_n), \qquad X_0 := \varepsilon_0,
$$
where $(\varepsilon_n)_{n = 0}^\infty$ is a sequence of i.i.d.~trinomial random variables, $(B_n)_{n = 0}^\infty$ are i.i.d.~$\operatorname{Ber}(\rho)$ independent of $(\varepsilon_n)_{n = 0}^\infty$, and $\rho \in (0, 1)$ is a memory parameter. We simulate $\widecheck{\bold X}$ and consider the pairs $\slices X := (X^n)_{n = 0}^\infty$ given via
$$
X^n := X_{D \cdot n + 1: D \cdot n + 2} \sim \mu \in \mathcal{P}(\R^2),
$$
where $D \in \mathbb{N}$ is the distance between consecutive sample paths. Note that the stationary distribution $\mu$ is given by
\begin{align*}
\mu(dx_{1:2}) = \mu^1(dx_1) \otimes \mu_{x_1}(dx_2), \;\text{ where }\; \mu^1 = \tfrac{1}{3} \cdot (\delta_{-1} + \delta_0 + \delta_{1}), \quad \mu_{x_1} = \rho \cdot \delta_{x_1} + (1 - \rho) \cdot \mu^1.
\end{align*}
In this section we compute $\mathcal{AW}(\mu, \widehat \mu^N)$ numerically and see how the convergence depends on $D$: a large value of $D$ corresponds to an ``almost" independent scenario, while smaller values make $X^n$ and $X^{n+1}$ more dependent (in the extreme case $D = 1$ we have $X^n_2 = X^{n+1}_1$). To see this mathematically, we first compute a theoretical bound for $\eta_{\slices X}(s)$.

Take $n, s \in \N$ and observe that $X^{n+s} \perp\!\!\!\perp X^{1:n}$, conditionally on the event $\{\min(B_{D\cdot n+2:D\cdot (n+s)}) = 0\}$. Thus, for any $S \in \mathcal{B}(\R^2)$ and $A \in \mathcal{B}((\R^2)^n)$ with $\mathbb{P}(X^{1:n} \in A) > 0$ we have
\begin{align}\label{eqn:numerics_1}
\begin{split}
\mathbb{P}\left(X^{n+s} \in S \mid X^{1:n} \in A\right) &= (1 - \rho^{Ds-1}) \cdot \mathbb{P}\left(X^{n+s} \in S \mid  \min(B_{D\cdot n+2:D\cdot (n+s)}) = 0\right)\\
&\quad + \rho^{Ds-1} \cdot \mathbb{P}\left(X^{n+s} \in S \mid X^{1:n} \in A, \; \min(B_{D\cdot n+2:D\cdot (n+s)}) = 1\right),
\end{split}
\end{align}
where we have used $\mathbb{P}(\min(B_{D \cdot n+2:D \cdot (n+s)}) = 1) = \rho^{Ds-1}$ and the law of total probability. Therefore,
\begin{equation}\label{eqn:numerics_eta}
|\mathbb{P}\left(X^{n+s} \in S \mid X^n \in A\right) - \mathbb{P}\left(X^{n+s} \in S \mid X^{n-1} \in B\right)| \overset{\eqref{eqn:numerics_1}}{\leq} 2 \rho^{Ds-1},
\end{equation}
and thus $\eta_{\slices X}(s) \leq 2 \rho^{Ds-1}$,
which follows from the former inequality by taking a supremum over $S \in \mathcal{B}(\R^2)$, then $A, B \in \mathcal{B}(\R^2)$ with $\mathbb{P}(X^1 \in A) \wedge \mathbb{P}(X^1 \in B) > 0$, and finally $1 \leq n < N - s + 1$. As a consequence, we anticipate stronger concentration and convergence as $D$ grows. 

To observe the phenomena qualitatively, we take $\rho = 0.99$ and compute $M = 5000$ values of $\mathcal{AW}(\mu, \widehat \mu^N)$ for $M$ independent samples $\slices X$ using linear programming. Recall that
$
\mathbb{E}\,\mathcal{AW}(\mu, \widehat \mu^N) \leq C \cdot \sqrt{1 + 2 \sum_{s = 1}^{N-1} \eta_{\slices X}(s)} \cdot N^{-\frac{1}{3}}
$
according to Theorem \ref{thm:moment_estimate_compact}. As $\rho$ is close to 1, the term $\sqrt{1 + 2 \sum_{s = 1}^{N-1} \eta_{\slices X}(s)}$ grows as $\sqrt{N}$ for $N \lesssim \frac{1}{D}$ per \eqref{eqn:numerics_eta}; for $N > \frac{1}{D}$ it is absorbed by the rate $N^{-\frac{1}{3}}$. Such a behavior is captured by our experiments; see the empirical convergence rate and concentration presented in Figure \ref{fig:aw_rate} and Figure \ref{fig:aw_concentration} respectively.

\begin{figure}[H]
    \centering
    \includegraphics[scale=0.30]{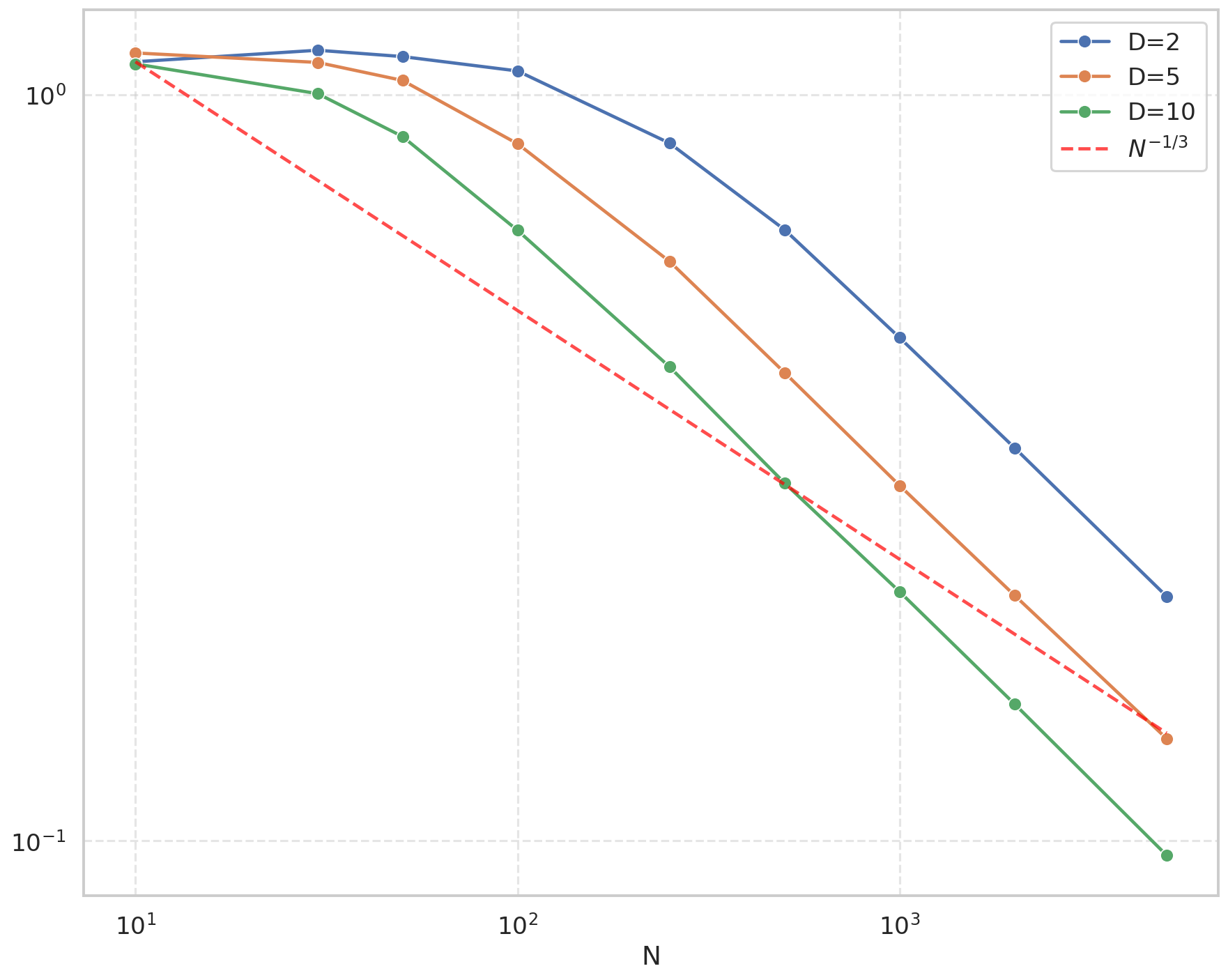}
    \caption{Rate of convergence of $\mathbb{E}\,\mathcal{AW}(\mu, \widehat \mu^N)$ for $D = 2, 5, 10$ compared to the theoretical upper bound $C \cdot \sqrt{1 + 2 \sum_{s = 1}^{N-1} \eta_{\slices X}(s)} \cdot N^{-1/3}$ from Theorem \ref{thm:moment_estimate_compact}. The expected value is estimated using empirical average over $M = 5000$ independent samples.}
    \label{fig:aw_rate}
\end{figure}

\begin{figure}[H]
    \centering
    \includegraphics[scale=0.25]{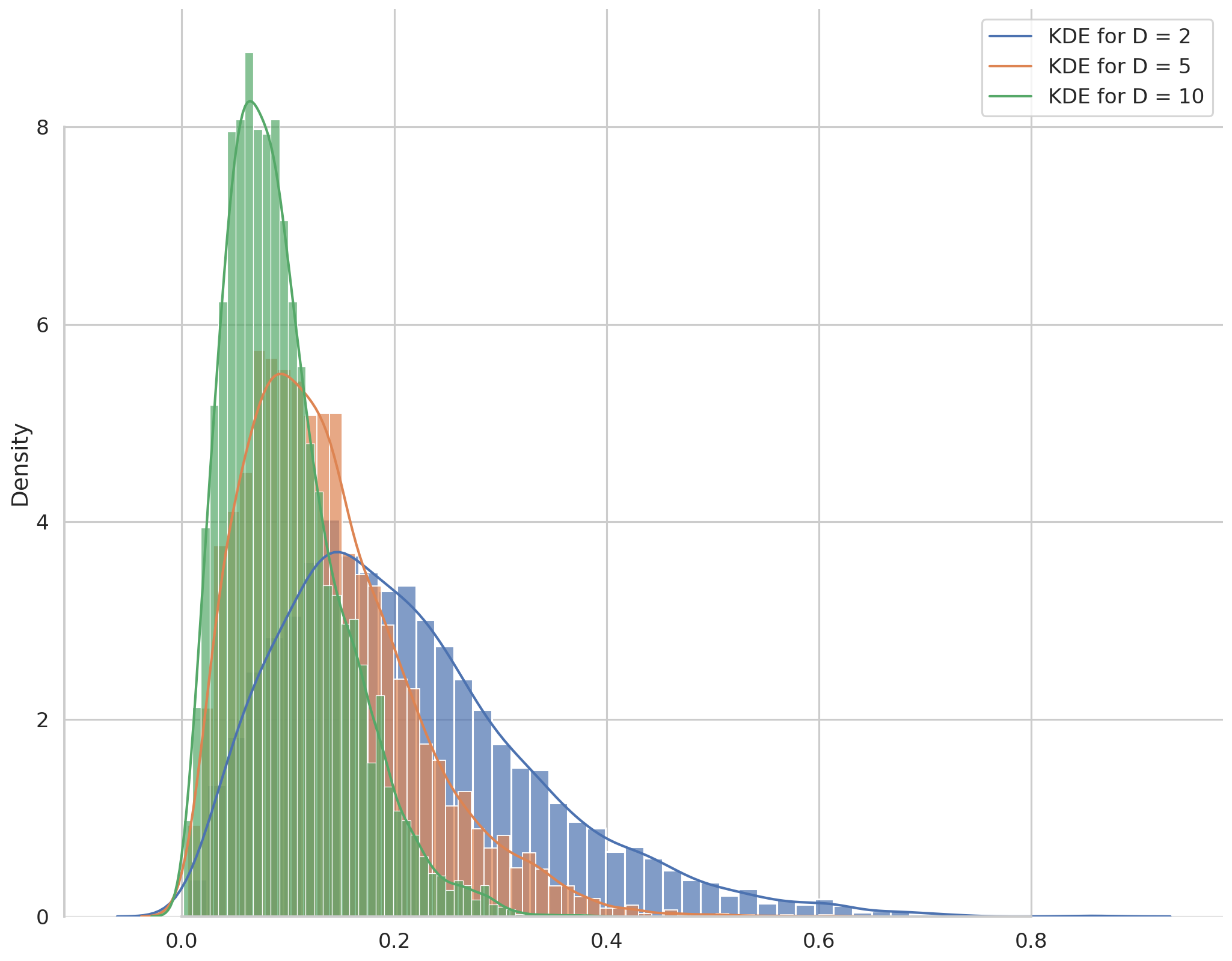}
    \caption{Empirical distribution of $\mathcal{AW}(\mu, \widehat \mu^N)$ for $D = 2, 5, 10$ over $M = 5000$ samples for $N = 5000$. }
    \label{fig:aw_concentration}
\end{figure}

In Figure \ref{fig:aw_rate}, one can see that smaller $D$ implies stronger dependence, which leads to a slower convergence to zero. Similarly, in Figure \ref{fig:aw_concentration}, smaller $D$ implies stronger dependence, which leads to a more skewed distribution with larger tails.

\section{Proofs for Section \ref{sec:main_expectation}}\label{sec:moment_estimate}

As mentioned in the Introduction, our methodology in this section resembles \cite{backhoff2022estimating} with a subtle difference: if $\slices{X}$ are not independent, then the law of the empirical measure of the disintegration is no longer the same as the law of the disintegration of the empirical measure, i.e. $(\mu^N)_G \sim (\mu_G)^{N \mu^N(G)}$ may not hold. Hence it is not possible to use the moment estimates from \cite{fournier2015rate} directly. However, as stated in \cite[Section 7.1]{fournier2015rate}, one can extend the methodology of \cite{fournier2015rate} to the case of mixing observations. To achieve this, we estimate $\mu^N(G) \cdot \mathcal{W}(\mu_G, (\mu^N)_G)$ where $\mu_G(\cdot) := \mu(\cdot \mid x_{1:t} \in G)$ for Borel sets $G$, directly via the Bayes formula and the covariance bound stated in Proposition \ref{prop:eta_mixing_covariance}. 

We first introduce additional notation that we use in the proofs. Recall the definition of $r$ in \eqref{eq:def_r} and denote by $\Phi^N_t$ a partition of $\operatorname{spt}(\mu)$ into products of cubes with edges of length $\Delta_N$; formally,
\begin{align}\label{eqn:Phi_t_definition}
\Phi^N_t := \left\{([0, \Delta_N)^{d})^t + \Delta_N \cdot z\;: \; z \in (\mathbb{Z}^{d})^t\right\} \cap \operatorname{spt}(\mu).
\end{align}
Given a probability measure $\mu \in \mathcal{P}((\mathbb{R}^d)^T)$ and any $G \in \Phi^N_t$, denote by $\mu_G \in \mathcal{P}(\mathbb{R}^d)$ the conditional distribution $\mu(dx_{t+1} \mid  x_{1:t} \in G)$. Given the observations $\slices{X} = (X^n)_{n = 1}^N \sim \mu \in \mathcal{P}((\mathbb{R}^d)^T)$, denote by $\mu^N = \frac{1}{N} \sum_{n = 1}^N \delta_{X^n}$ the empirical probability measure and set
$$
\mu^N_G := \frac{1}{|I_G|} \sum_{n \in I_G} \delta_{X_{t+1}^n},
$$
where $I_G := \{n \in \{1, \ldots, N\}\,:\, X_{1:t}^n \in G\}$, with the convention  $\mu^N_G := \delta_0$ if $I_G=\emptyset.$

We now proceed with the proofs. We start with some auxiliary statements for the proof of Theorem \ref{thm:rate_general}. The first one is similar to \cite[Theorem 14]{fournier2015rate}.
\begin{proposition}\label{prop:probability_bound_mixing}
Let $\slices{X} = (X^n)_{n = 1}^N \sim \mu \in \mathcal{P}((\mathbb{R}^d)^T)$. Then for any set $A \in \mathcal{B}((\mathbb{R}^d)^T)$ we have
$$
\mathbb{E}\,|\mu(A) - \mu^N(A)| \leq \min\left(2 \mu(A), \sqrt{1 + 2 \sum_{s = 1}^{N-1} \eta_{\slices{X}}(s)} \cdot \frac{\sqrt{\mu(A)}}{\sqrt N}\right).
$$
\end{proposition}
\begin{proof}
First, by the triangle inequality and $\mathbb{E}\, \mu^N(A) = \mu(A)$ we have
\begin{align*}
\mathbb{E}\,|\mu(A) - \mu^N(A)| \leq \mu(A) + \mathbb{E}\, \mu^N(A) = 2 \mu(A).
\end{align*}
Next, applying Proposition \ref{prop:eta_mixing_variance} with $f = \mathds{1}_A$, we get
\begin{align*}
\left(\mathbb{E}\,|\mu(A) - \mu^N(A)|\right)^2 \stackrel{\text{(Jensen)}}{\leq} \operatorname{Var}\left[\mu^N(A)\right] \leq \frac{\mu(A)}{N} \cdot \left(1 + 2 \sum_{s = 1}^{N-1} \eta_{\slices{X}}(s)\right).
\end{align*}
Combining the inequalities above completes the proof.
\end{proof}

\begin{lemma}\label{lem:probability_bound_disintegration}
Let $\slices{X} = (X^n)_{n = 1}^N \sim \mu \in \mathcal{P}((\mathbb{R}^d)^T)$. Then for any $G \in \Phi^N_t$ and $A \in \mathcal{B}(\mathbb{R}^d)$ we have
\begin{equation}\label{eqn:probability_bound_disintegration}
\mathbb{E}\,[\mu^N(G) |\mu_G(A) - \mu^N_G(A)|] \leq C \cdot \sqrt{1 + 2 \sum_{s = 1}^{N-1} \eta_{\slices{X}}(s)} \cdot \mu(G) \cdot \min\left(4 \mu_G(A), 2 \frac{\sqrt{\mu_G(A)}}{\sqrt{N \mu(G)}}\right),
\end{equation}
where $C > 0$ is an absolute constant.
\end{lemma}

\begin{proof}
We start by decomposing the integrand. Subtracting and adding $\mu(A \times G)$ and using triangle inequality we obtain
\begin{align}\label{eqn:central_decomposition}
\begin{split}
\mu^N(G) |\mu_G(A) - \mu^N_G(A)|
&= |\mu^N(G) \mu_G(A) - \mu^N(A \times G)|\\
&\leq |\mu^N(G)\mu_G(A) - \mu(A \times G)| + |\mu(A \times G) - \mu^N(A \times G)|\\
&= \mu_G(A) |\mu^N(G) - \mu(G)| + |\mu(A \times G) - \mu^N(A \times G)|.
\end{split}
\end{align}
Applying Proposition \ref{prop:probability_bound_mixing} with the pairs $(\mu, G \times (\mathbb{R}^d)^{(T - t)})$ and $(\mu, G \times A \times (\mathbb{R}^d)^{(T - t - 1)})$ and denoting $\sqrt{1 + 2 \sum_{s = 1}^{N-1} \eta_{\slices{X}}(s)}$ by $C_\eta$, we get the following bound:
\begin{align*}
&\mathbb{E}\,[\mu^N(G) |\mu_G(A) - \mu^N_G(A)|]\\
&\overset{\eqref{eqn:central_decomposition}}{\leq} \mu_G(A) \min\left(2 \mu(G), C_\eta \cdot \frac{\sqrt{\mu(G)}}{\sqrt N}\right) + \min\left(2 \mu(A \times G), C_\eta \cdot\frac{\sqrt{\mu(A \times G)}}{\sqrt N}\right)\\
&\leq \min\left(2 \mu_G(A) \mu(G) + 2 \mu(A \times G), \frac{C_\eta}{\sqrt{N}} \left[\mu_G(A) \sqrt{\mu(G)} + \sqrt{\mu(A \times G)}\right]\right)\\
&= \mu(G) \min\left(4 \mu_G(A), \frac{C_\eta}{\sqrt{N \mu(G)}} \left[\mu_G(A) + \sqrt{\mu_G(A)}\right]\right)\\
&\leq \mu(G) \min\left(4 \mu_G(A), 2 C_\eta \cdot \frac{\sqrt{\mu_G(A)}}{\sqrt{N \mu(G)}}\right),
\end{align*}
where the second inequality follows from $\min(x, y) + \min(a, b) \leq \min(x + a, y + b)$ and the final inequality holds because $x \leq \sqrt{x}$ for $x \in [0, 1]$. The proof is complete as $C_\eta \geq 1$ by non-negativity of $\eta_{\slices{X}}$.
\end{proof}

For the next Lemma, we follow ideas from \cite[Proof of Theorem 1]{fournier2015rate} quite closely.

\begin{lemma}\label{lem:fournier_estimate}
Let $\nu \in \mathcal{P}(\R^d)$ satisfy $\int \|x\|^p \, \nu(dx) \leq 1$ for some $p > 1$. Then for any $M > 0$ we have
\begin{equation}\label{eqn:fournier_estimate}
\sum_{n \geq 0} 2^n \sum_{l \geq 0} 2^{-l} \sum_{F \in \mathcal{P}_l} \min\left(2 \nu(2^n F \cap B_n), \sqrt{\frac{\nu(2^n F \cap B_n)}{M}}\right) \leq C \cdot R_p(M),
\end{equation}
where $C > 0$ depends on $d, p$. The collection $\mathcal{P}_l$ is the natural partition of $(-1,1]^d$ into $2^{dl}$ translations of $(-2^{-l}, 2^{-l}]^d$ and the partition $B_n$ is defined as
\begin{align}\label{eqn:fournier_estimate_partition}
\begin{split}
B_0 &:= (-1, 1]^d, \quad B_n := (-2^n, 2^n]^d \setminus (-2^{n-1}, 2^{n-1}]^d \;\; \text{for} \;\; n \in \mathbb{N},
\end{split}
\end{align}
for $l \geq 0$ and $n \geq 0$, and
\begin{equation}\label{eqn:R}
R_p(u) := \begin{cases}
u^{-\frac{p-1}{p}}, &p < \infty,\\
0, &p = \infty
\end{cases} + \begin{cases}
u^{- 1 / 2}, &d = 1,\\
u^{- 1 / 2} \log (u + 3), &d = 2,\\
u^{- 1 / d}, &d \geq 3.
\end{cases}
\end{equation}
\end{lemma}

\begin{proof}
We first apply the inequality $\sum_i \min(a_i, b_i) \leq \min(\sum_i a_i, \sum_i b_i)$ and the Cauchy-Schwarz inequality to obtain
\begin{align}\label{eqn:fournier_estimate_start}
\begin{split}
\sum_{F \in \mathcal{P}_l} \min\left(2 \nu(2^n F \cap B_n), \sqrt{\frac{\nu(2^n F \cap B_n)}{M}}\right) &\leq \min\left(2 \nu(B_n), \sum_{F \in \mathcal{P}_l} \sqrt{\frac{\nu(2^n F \cap B_n)}{M}}\right)\\
&\overset{|\mathcal{P}_l| = 2^{dl}}{\leq} \min\left(2 \nu(B_n), 2^\frac{dl}{2} \sqrt{\frac{\nu(B_n)}{M}}\right)\\
&\leq 2^{p+1} \cdot \min\left(2^{-pn}, 2^{\frac{dl}{2}} \sqrt{\frac{2^{-pn}}{M}} \right),
\end{split}
\end{align}
where the final inequality follows from $\int \|x\|^p \, \nu(dx) \leq 1$ and Markov's inequality as $\nu(B_n) \leq \mu(\{x \in \R^d\,:\,\|x\| \geq 2^{n-1}\}) \leq 2^{-p(n-1)}$. Consequently,
\begin{align}\label{eqn:fournier_estimate_prefinal}
\begin{split}
&\sum_{n \geq 0} 2^n \sum_{l \geq 0} 2^{-l} \sum_{F \in \mathcal{P}_l} \min\left(2 \nu(2^n F \cap B_n), \sqrt{\frac{\nu(2^n F \cap B_n)}{M}}\right)\\
&\overset{\eqref{eqn:fournier_estimate_start}}{\leq} 2^{p+1} \cdot \sum_{n \geq 0} 2^n \sum_{l \geq 0} 2^{-l} \min\left(2^{-pn}, 2^{\frac{dl}{2}} \sqrt{\frac{2^{-pn}}{M}} \right),
\end{split}
\end{align}
where the last term is the rhs of \cite[Equation (4), proof of Theorem 1]{fournier2015rate} (choosing $p=1, q=p,\mu=\nu, N=M$ in their notation) up to a multiplicative constant which depends on $d, p$. Finally, from \cite[proof of Theorem 1, Steps 1-4]{fournier2015rate} it follows that
\begin{equation}\label{eqn:fournier_estimate_steps}
\sum_{n \geq 0} 2^n \sum_{l \geq 0} 2^{-l} \min\left(2^{-pn}, 2^{\frac{dl}{2}} \sqrt{\frac{2^{-pn}}{M}} \right) \leq C \cdot R_p(M),
\end{equation}
where $C > 0$ depends on $d, p$. The proof is complete, as \eqref{eqn:fournier_estimate_prefinal} and \eqref{eqn:fournier_estimate_steps} imply \eqref{eqn:fournier_estimate}.
\end{proof}

\begin{lemma}\label{lem:wasserstein_distance_bound_disintegration}
Let $\mu \in \mathcal{P}((\mathbb{R}^d)^T)$ satisfy $(\int \|x\|^p \, \mu_G(dx))^{1/p} \leq 1$ for some $p > 1$ and $G \in \Phi^N_t$, and let $\slices{X} \sim \mu$. Then
$$
\mathbb{E}\,[\mu^N(G) \mathcal{W}(\mu_G, \mu^N_G)] \leq C \cdot \sqrt{1 + 2 \sum_{s = 1}^{N-1} \eta_{\slices{X}}(s)} \cdot \mu(G) \cdot R_p(N \mu(G)),
$$
where $C > 0$ depends on $d, p$; the rate function $R_p$ is defined in \eqref{eqn:R}.
\end{lemma}
\begin{proof}
First, \cite[Lemma 5 \& Lemma 6]{fournier2015rate} yield that for any $\mu, \nu \in \mathcal{P}(\mathbb{R}^d)$ 
\begin{equation}\label{eqn:fournier_lemma}
\mathcal{W}(\mu, \nu) \leq C \cdot \sum_{n \geq 0} 2^n \sum_{l \geq 0} 2^{-l} \sum_{F \in \mathcal{P}_l} |\mu(2^n F \cap B_n) - \nu(2^n F \cap B_n)|,
\end{equation}
where $C > 0$ depends on $d$. The partitions $\mathcal{P}_l, B_n$ are defined in \eqref{eqn:fournier_estimate_partition}. Consequently, denoting $\sqrt{1 + 2 \sum_{s = 1}^{N-1} \eta_{\slices{X}}(s)}$ by $C_\eta$ we obtain
\begin{align}
\begin{split}
&\mathbb{E}\,[\mu^N(G) \mathcal{W}(\mu_G, \mu^N_G)]\\
&\leq C \cdot \mathbb{E}\,\left[\mu^N(G) \sum_{n \geq 0} 2^n \sum_{l \geq 0} 2^{-l} \sum_{F \in \mathcal{P}_l} |\mu_G(2^n F \cap B_n) - \mu^N_G(2^n F \cap B_n)|\right]\\
&= C \cdot \sum_{n \geq 0} 2^n \sum_{l \geq 0} 2^{-l} \sum_{F \in \mathcal{P}_l} \mathbb{E}\left[\mu^N(G) \cdot |\mu_G(2^n F \cap B_n) - \mu^N_G(2^n F \cap B_n)|\right]\\
&\overset{\text{(Lemma \ref{lem:probability_bound_disintegration})}}{\leq} C \cdot C_\eta \cdot \mu(G) \cdot \sum_{n \geq 0} 2^n \sum_{l \geq 0} 2^{-l} \sum_{F \in \mathcal{P}_l} \min\left(4 \mu_G(2^n F \cap B_n), 2\frac{\sqrt{\mu_G(2^n F \cap B_n)}}{\sqrt{N \mu(G)}}\right)\\
&\overset{\text{(Lemma \ref{lem:fournier_estimate})}}{\leq} C \cdot C_\eta \cdot \mu(G) \cdot R_p(N\mu(G)),
\end{split}
\end{align}
where $C > 0$ depends on $d, p$ and we have applied Lemma \ref{lem:fournier_estimate} with $\nu=\mu_G$ and $M=N\mu(G)$ in the last inequality. The proof is complete.
\end{proof}

\begin{lemma}\label{lem:sum_mu_G_W_mu_G_mu_N_G}
Let $\mu \in \mathcal{P}(([0, 1]^d)^T)$, and let $\slices X \sim \mu$. Then
$$
\mathbb{E}\Bigg[\sum_{G \in \Phi^N_t} \mu^N(G) \mathcal{W}(\mu_G, \mu^N_G) \Bigg]\leq C \cdot \sqrt{1 + 2 \sum_{s = 1}^{N-1} \eta_{\slices X}(s)} \cdot \operatorname{rate}_\infty(N),
$$ 
where $C > 0$ depends on $d$.
\end{lemma}
\begin{proof}
By Lemma \ref{lem:wasserstein_distance_bound_disintegration} applied with $p = \infty$ and a scaling argument, for any $G \in \Phi^N_t$ we have
\begin{equation}\label{eqn:sum_lemma_1}
\mathbb{E}\,[\mu^N(G) \mathcal{W}(\mu_G, \mu^N_G)] \leq C \cdot \sqrt{1 + 2 \sum_{s = 1}^{N-1} \eta_{\slices{X}}(s)} \cdot \mu(G) \cdot R_\infty(N \mu(G)),
\end{equation}
where $C > 0$ depends on $d$, and $R_\infty$ is defined in \eqref{eqn:R}. Hence, using concavity of $u \mapsto uR_\infty(u)$ we write
\begin{align}\label{eqn:sum_lemma_2}
\begin{split}
&\sum_{G \in \Phi^N_t} \mu(G) \cdot R_\infty(N \mu(G))\\
&= \frac{|\Phi^N_t|}{N} \sum_{g \in \Phi^N_t} \frac{1}{|\Phi^N_t|} \cdot N \mu(G) R_\infty(N \mu(G))\\
&\overset{\text{(Jensen)}}{\leq} \frac{|\Phi^N_t|}{N} \left(\sum_{G \in \Phi^N_t} \frac{N \mu(G)}{|\Phi^N_t|}\right) R_\infty\left(\sum_{G \in \Phi^N_t} \frac{N \mu(G)}{|\Phi^N_t|}\right)\\
&= R_\infty\left(\frac{N}{|\Phi^N_t|}\right) \leq R_\infty\left(\frac{N}{|\Phi^N_{T-1}|}\right) = R_\infty(N \cdot \Delta_N^{d(T-1)}) = \operatorname{rate}_\infty(N),
\end{split}
\end{align}
where we have used that $R_\infty$ is decreasing and $|\Phi^N_{T-1}| = \Delta_N^{-d(T-1)}$, as $\mu$ is supported on $([0, 1]^d)^T$. The claim now follows from \eqref{eqn:sum_lemma_1} and \eqref{eqn:sum_lemma_2}.
\end{proof}

We are now in a position to prove Theorem \ref{thm:moment_estimate_compact} and Theorem \ref{thm:rate_general}:
\begin{proof}[Proof of Theorem \ref{thm:moment_estimate_compact}]
Without loss of generality assume that $\mu$ is supported on $([0, 1]^d)^T$. First, by \cite[Lemma 3.1]{backhoff2022estimating} and the fact that $\mu$ has $L$-Lipschitz kernels we have
\begin{equation}\label{eqn:adapted_dp}
\mathcal{AW}(\mu, \widehat \mu^N) \leq C \cdot \mathcal{W}(\mu^1, (\widehat \mu^{N})^1) + C \cdot \sum_{t = 1}^{T - 1} \int \mathcal{W}(\mu_{y_{1:t}}, \widehat \mu^N_{y_{1:t}}) \, \widehat \mu^N(dy),
\end{equation}
where $C > 0$ depends on $L$, and $(\widehat \mu^{N})^1 := \frac{1}{N} \sum_{n = 1}^N \delta_{\varphi^N(X^n_1)}$ is the empirical measure of the first coordinates of $\slices{X}$. We then proceed with each term of the type $\int \mathcal{W}(\mu_{y_{1:t}}, \widehat \mu^N_{y_{1:t}}) \, \widehat \mu^N(dy)$ separately. Applying \cite[Lemma 3.2]{backhoff2022estimating} we obtain
\begin{equation}\label{eqn:disintegration_bound}
\int \mathcal{W}(\mu_{y_{1:t}}, \widehat \mu^N_{y_{1:t}}) \, \widehat \mu^N(dy) \leq C \cdot \Delta_N + C \cdot \sum_{G \in \Phi^N_t} \mu^N(G) \mathcal{W}(\mu_G, \mu^N_G),
\end{equation}
where 
and $C > 0$ depends on $L$. Hence, by Lemma \ref{lem:sum_mu_G_W_mu_G_mu_N_G} we have
\begin{align*}
\mathbb{E}\Bigg[ \int \mathcal{W}(\mu_{y_{1:t}}, \widehat \mu^N_{y_{1:t}}) \, \widehat\mu^N(dy)\Bigg]
&\overset{\eqref{eqn:disintegration_bound}}{\leq} C \cdot  \Delta_N + C \cdot \mathbb{E}\Bigg[ \sum_{G \in \Phi^N_t} \mu^N(G) \mathcal{W}(\mu_G, \mu^N_G)\Bigg]\\
&\leq C \cdot \Delta_N + C \cdot \sqrt{1 + 2 \sum_{s = 1}^{N-1} \eta_{\slices{X}}(s)} \cdot \operatorname{rate}_{\infty}(N),
\end{align*}
where $C > 0$ depends on $d, L$. Summing up these estimates over $t \in \{1, \ldots, T-1\}$ we conclude
$$
\mathcal{AW}(\mu, \widehat \mu^N) \leq C \cdot \sqrt{1 + 2 \sum_{s = 1}^{N-1} \eta_{\slices{X}}(s)} \cdot \operatorname{rate}_{\infty}(N),
$$
where $C > 0$ depends on $d, L, T$. The proof is complete.
\end{proof}

For general $\mu$ we use a scaling argument together with Assumption \ref{asm:uniform_noise_moment}, similarly to \cite{acciaio2024convergence}. 

\begin{proof}[Proof of Theorem \ref{thm:rate_general}]
As in \cite{{acciaio2024convergence}} we only present the proof for the case $d \geq 3$,  for notational simplicity. The proof for the case for $d\in \{1,2\}$ follows from very similar arguments.

Following the same arguments as in the proof of Theorem \ref{thm:moment_estimate_compact} we need to estimate
$$
\mathbb{E} \Bigg[ \sum_{G \in \Phi^N_t} \mu^N(G) \mathcal{W}(\mu_G, \mu^N_G)\Bigg]
$$
for $t \in \{1, \ldots, T - 1\}$. For this we define the mapping $\Phi: (\R^d)^T \to (\R^d)^T$ as follows:
\begin{align*}
&\Phi(x_{1:T}) := \left(x_{1:t}, \frac{1}{\lambda(x_{1:t})} \frac{x_{t+1} - f(x_{1:t})}{\sigma(x_{1:t})}, x_{t+2:T}\right),\\
&\lambda(x_{1:t}) := \sum_{G \in \Phi^N_t} \lambda_G \cdot \mathds{1}_G(x_{1:t}),\quad \sigma(x_{1:t}) := \sum_{G \in \Phi^N_t} \sigma_G \cdot \mathds{1}_G(x_{1:t}),\quad f(x_{1:t}) := \sum_{G \in \Phi^N_t} f_G \cdot \mathds{1}_G(x_{1:t}),
\end{align*}
where
\begin{align}
\begin{split}
\lambda_G &:= L \cdot \|G\| + \left(\sup_{x_{1:t} \in (\mathbb{R}^d)^t} \mathbb{E} \, \|\varepsilon(x_{1:t})\|^p\right)^\frac{1}{p},\\
\sigma_G &:= \sup_{x_{1:t} \in G} \sigma_t(x_{1:t}) \vee 1,\\
f_G &:= f_t(x_G), \;\; x_G \in \varphi^N(G)
\end{split}
\end{align}
are the constants from \cite[Lemma 3.3]{acciaio2024convergence}.
By Assumption \ref{asm:uniform_noise_moment} we have
\begin{equation}\label{eqn:normalization_properties}
\sigma_G \leq C \cdot (1 + \|G\|^r), \quad \max_{G \in \Phi^N_t} \lambda_G \leq C,
\end{equation}
where $C > 0$ depends on $L, B$. Define now $\bar \mu := (\Phi)_\# \mu, \;\; \bar \mu^N := (\Phi)_\# \mu^N$, and note that
$(\Phi(x_{1:T}))_{1:t} \equiv x_{1:t}$. Thus,
\begin{equation}\label{eqn:normalization_measure_properties}
\bar \mu(dx_{1:t}) = \mu(dx_{1:t}).
\end{equation}
Furthermore, according to \cite[Lemma 3.3]{acciaio2024convergence} we have
\begin{equation}\label{eqn:wasserstein_normalized_p_moment_bound}
\int \|x_{t+1}\|^p \, \bar \mu_G(dx_{t+1}) \leq 1.
\end{equation}
Next, by the scaling property of Wasserstein distance,
\begin{equation}\label{eqn:wasserstein_scaling}
\mathcal{W}(\mu_G, \mu^N_G) = \lambda_G \sigma_G \cdot \mathcal{W}(\bar \mu_G, \bar \mu^N_G) \overset{\eqref{eqn:normalization_properties}}{\leq} C \cdot (1 + \|G\|^r) \cdot \mathcal{W}(\bar \mu_G, \bar \mu^N_G)
\end{equation}
for all $G \in \Phi^N_t$, where $C > 0$ depends on $L, B$. 
To estimate the rhs of \eqref{eqn:wasserstein_scaling}, we use \eqref{eqn:wasserstein_normalized_p_moment_bound} and apply Lemma \ref{lem:wasserstein_distance_bound_disintegration} with $\bar \mu$ and $\Phi(\slices{X}) := (\Phi(X^n))_{n = 1}^N$ to get 
\begin{align}
\begin{split}
\mathbb{E}\,[\bar \mu^N(G) \cdot \mathcal{W}(\bar \mu_G, \bar \mu^N_G)]
&\leq C \cdot \sqrt{1 + 2 \sum_{s = 1}^{N-1} \eta_{\Phi(\slices{X})}(s)} \cdot \bar\mu(G)\cdot  R_p(N \bar\mu(G))\\
&\overset{\substack{\text{(Prop.  \ref{prop:eta_mixing_information_processing})}\\\eqref{eqn:normalization_measure_properties}}}{\leq} C \cdot \sqrt{1 + 2 \sum_{s = 1}^{N-1} \eta_{\slices{X}}(s)} \cdot \mu(G)\cdot R_p(N \mu(G)),
\end{split}
\end{align}
where $C > 0$ depends on $d, p$. Therefore, denoting $\sqrt{1 + 2 \sum_{s = 1}^{N-1} \eta_{\slices{X}}(s)}$ by $C_\eta$, we obtain
\begin{align}\label{eqn:moment_estimate_general_prefinal}
&\mathbb{E}\,\sum_{G \in \Phi^N_t} \mu^N(G) \mathcal{W}(\mu_G, \mu^{N}_G)\nonumber\\
&\overset{\eqref{eqn:normalization_properties}, \eqref{eqn:normalization_measure_properties}, \eqref{eqn:wasserstein_scaling}}{\leq} C \cdot C_\eta \cdot \sum_{G \in \Phi^N_t} \mu(G) \cdot (1 + \|G\|^r) \cdot R_p(N \mu(G))\nonumber\\
&= C \cdot C_\eta \cdot \left(N^{-\frac{1}{d}} \sum_{G \in \Phi^N_t} (1 + \|G\|^r) \cdot \mu(G)^{1 - \frac{1}{d}} + N^{-\frac{p-1}{p}} \sum_{G \in \Phi^N_t} (1 + \|G\|^r) \cdot \mu(G)^{1 - \frac{p-1}{p}}\right),
\end{align}
where $C > 0$ depends on $d, L, B$. Having obtained the same equation as in \cite[Equation (37), proof of Theorem 2.16]{acciaio2024convergence} up to a multiplicative constant and a term $C_\eta = \sqrt{1 + 2 \sum_{s = 1}^{N-1} \eta_{\slices{X}}(s)}$, we follow the proof of \cite[Theorem 2.16, Steps 2-3]{acciaio2024convergence} line by line. More precisely, we estimate \eqref{eqn:moment_estimate_general_prefinal} using \cite[Equation (39), proof of Theorem 2.16]{acciaio2024convergence} and obtain \cite[Equation (40), proof of Theorem 2.16]{acciaio2024convergence}, which leads to a final estimate for $\mathbb{E} \, \mathcal{AW}(\mu, \widehat \mu^N)$.
\end{proof}

\section{Bounded Differences inequality for mixing sequences and proofs for Section \ref{sec:main_concentration} }\label{sec:concentration}
Before proceeding with the proofs of Theorem \ref{thm:AW_concentration_compact} and \ref{thm:AW_concentration_general}, we prove a version of Bounded Differences inequality for mixing sequences.

\subsection{Bounded Differences inequality for mixing sequences}
We start with the following extension of the Bounded Differences inequality:
\begin{lemma}\label{lem:bounded_differences_mixing}
Let $\bold Z = (Z_n)_{n = 1}^N$ be a collection of random elements taking values in $\R^k$, $k \geq 1$. Let $\phi: (\R^k)^N \to \mathbb{R}$ be $L$-Lipschitz with respect to the Hamming distance\footnote{In other words, $\phi$ has $L$-bounded differences.}, i.e.
$$
|\phi(x)-\phi(y)| \leq L \cdot \sum_{n = 1}^N \mathds{1}_{\{x_n \neq y_n\}} \quad \text{for all }x, y \in (\R^k)^N.
$$
Then we have the concentration inequality 
\begin{equation}\label{eqn:bounded_differences_mixing}
\mathbb{P}\left(\left|\phi(Z_{1:N}) - \mathbb{E}\, \phi(Z_{1:N})\right| > \varepsilon\right) \leq 2 \exp\left(-\frac{\varepsilon^2}{2N L^2 (1 + \sum_{s = 1}^{N-1} \bar \eta_{\bold Z}(s))^2}\right).
\end{equation}
\end{lemma}

\begin{remark}
The Bounded Differences (McDiarmid's) inequality for the i.i.d.~observations $\bold Z = (Z_n)_{n = 1}^N$ and a function $\phi: (\R^k)^N \to \R$ with $L$-bounded differences states that
$$
\mathbb{P}\left(\left|\phi(Z_{1:N}) - \mathbb{E}\, \phi(Z_{1:N})\right| > \varepsilon\right) \leq 2 \exp\left(-\frac{2\varepsilon^2}{N L^2}\right),
$$
see \cite[Lemma 1.2]{mcdiarmid1989method}. This coincides with our result up to a multiplicative constant when $\bold Z$ are independent, or, more generally, $\sum_{s = 1}^\infty \bar \eta_{\bold Z}(s) < \infty$. More generally, for $\eta$-mixing observations taking values in a countable space $\mathcal{Z}$, Kontorovich and Ramanan proved that (see \cite[Theorem 1.1]{kontorovich2008concentration})
\begin{equation}\label{eqn:kontorovich_ramanan_concentration}
\mathbb{P}\left(\left|\phi(Z_{1:N}) - \mathbb{E}\, \phi(Z_{1:N})\right| > \varepsilon\right) \leq 2 \exp\left(-\frac{\varepsilon^2}{2N L^2 \max_{1 \leq n \leq N} (1 + \sum_{s = 1}^{N - n} \widehat \eta_{\bold Z}(n, n + s))^2}\right),
\end{equation}
where $\widehat \eta_{\bold Z}(n, n + s)$ was defined in \eqref{eqn:kontorovich_ramanan_eta_mixing} and we use the convention $\sum_{s=1}^0 \eta_{\bold Z}(n, n + s):=0$. Both results are proved using a martingale differences argument in the form of Azuma's inequality; \cite{kontorovich2008concentration} takes an algebraic approach to bound the differences (which cannot be directly extended to uncountable spaces), while we use the simple tensorization property of $\operatorname{TV}$ in the form of Lemma \ref{lem:TV_subadditivity}.

Let us now compare \eqref{eqn:kontorovich_ramanan_concentration} and Lemma \ref{lem:bounded_differences_mixing}. We estimate
\begin{align}\label{eqn:kontorovich_ramanan_eta_stationary}
\begin{split}
\widehat \eta_{\bold Z}(n, n+s) &\overset{\eqref{eqn:kontorovich_ramanan_eta_mixing}}{=} \sup_{\substack{z_{1:n} \in \mathcal{Z}^n, \tilde z_n \in \mathcal{Z}\\\mathbb{P}(Z_{1:n} = z_{1:n}) > 0\\\mathbb{P}(Z_{1:n} = z_{1:n-1} \tilde z_n) > 0}} \operatorname{TV}\left(\operatorname{Law}(Z_{n+s} \mid Z_{1:n} = z_{1:n}), \operatorname{Law}(Z_{n+s} \mid Z_{1:n} = z_{1:n-1} \tilde z_n)\right)\\
&\leq 2 \sup_{\substack{z_{1:n} \in \mathbb{Z}^n\\
\mathbb{P}(Z_{1:n} = z_{1:n}) > 0}} \operatorname{TV}\left(\operatorname{Law}(Z_{n+s} \mid Z_{1:n} = z_{1:n}), \operatorname{Law}(Z_{n+s} \mid Z_{1:n-1} = z_{1:n-1})\right)\\
&\leq 2 \bar \eta_{\bold Z}(s),
\end{split}
\end{align}
where the first inequality follows from triangle inequality for $\operatorname{TV}$. Thus, for countable spaces,
\begin{equation*}
\mathbb{P}\left(\left|\phi(Z_{1:N}) - \mathbb{E}\, \phi(Z_{1:N})\right| > \varepsilon\right) \overset{\eqref{eqn:kontorovich_ramanan_concentration}, \eqref{eqn:kontorovich_ramanan_eta_stationary}}{\leq} 2 \exp\left(-\frac{\varepsilon^2}{2NL^2 (1 + 2 \sum_{s = 1}^{N-1} \bar \eta_{\bold Z}(s))^2}\right),
\end{equation*}
which coincides with our result up to a multiplicative constant.
\end{remark}

\begin{proof}[Proof of Lemma \ref{lem:bounded_differences_mixing}]
We proceed with a martingale differences argument: for $n \in \{1, \ldots, N\}$ define
\begin{equation}\label{eqn:bdd_def}
V_n(z_{1:n}) := \mathbb{E}\,[\phi(Z_{1:N}) \mid Z_{1:n} = z_{1:n}] - \mathbb{E}\,[\phi(Z_{1:N}) \mid Z_{1:n-1} = z_{1:n-1}], \qquad V_0 := 0.
\end{equation}
Fix $n \in \{1, \ldots, N\}$. By the $L$-Lipschitz property of $\phi$ we have
\begin{equation}\label{eqn:bdd_L}
|\phi(z_{1:n-1} \tilde z_{n:N}) - \phi(z_{1:n} \tilde z_{n+1:N})| \leq L
\end{equation}
for all $z_{1:n} \in (\R^k)^n$ and $\tilde z_{n:N} \in (\R^k)^{N-n+1}$, and consequently 
\begin{align}\label{eqn:bounded_differences_subadditivity} 
\begin{split}
|V_n(z_{1:n})| &\overset{\eqref{eqn:bdd_def}}{=} \left|\mathbb{E}\,[\phi(z_{1:n} Z_{n+1:N}) \mid Z_{1:n} = z_{1:n}] - \mathbb{E}\,[\phi(z_{1:n-1} Z_{n:N}) \mid Z_{1:n-1} = z_{1:n-1}]\right|\\
&\overset{\eqref{eqn:bdd_L}}{\leq} L \cdot \left(1 + \left|\mathbb{E}\,[\phi(z_{1:n} Z_{n+1:N}) \mid Z_{1:n} = z_{1:n}] - \mathbb{E}\,[\phi(z_{1:n} Z_{n+1:N}) \mid Z_{1:n-1} = z_{1:n-1}]\right|\right)\\
&\leq L \cdot \Big(1 + \sum_{s = 1}^{N - n} \operatorname{TV}(\operatorname{Law}(Z_{n+s:N} \mid Z_{1:n} = z_{1:n}), \operatorname{Law}(Z_{n+s:N} \mid Z_{1:n-1} = z_{1:n-1}))\Big)
\end{split}
\end{align}
for any $z_{1:n} \in (\R^k)^n$, where the second inequality follows from Lemma \ref{lem:TV_subadditivity} applied with $M = N - n$, $\phi(\cdot) = \phi(z_{1:n}, \cdot)$, $\mu = \operatorname{Law}(Z_{n+1:N} \mid Z_{1:n} = z_{1:n})$ and $\nu = \operatorname{Law}(Z_{n+1:N} \mid Z_{1:n-1} = z_{1:n-1})$. Therefore, applying Proposition \ref{prop:TV_bound_mixing} we conclude
$$
|V_n(z_{1:n})| \overset{\eqref{eqn:bounded_differences_subadditivity}}{\leq} L \cdot \Big(1 + \sum_{s = 1}^{N - n} \bar \eta_{\bold Z}(s)\Big)
$$
$\mathbb{P}(Z_{1:n} \in dz_{1:n})$-almost surely. The proof now follows from \cite[Corollary 2.20]{wainwright2019high} applied with martingale difference sequence $(V_n, \sigma(Z_{1:n}))_{n = 1}^N$.
\end{proof}

\begin{example}[Wasserstein distance concentration]
Let $\mu \in \mathcal{P}(\R^k)$ be supported on the unit ball, and let $\bold Z = (Z_n)_{n = 1}^N \sim \mu$. Applying Lemma \ref{lem:bounded_differences_mixing} with $\phi(z_{1:N}) := \mathcal{W}\left(\mu, \frac{1}{N} \sum_{n = 1}^N \delta_{z_n}\right)$, for which $L = \frac{2}{N}$, we obtain
\begin{align}\label{eq:conc_wasserstein}
\mathbb{P}\left(\left|\mathcal{W}(\mu, \mu^N) - \mathbb{E}\, \mathcal{W}(\mu, \mu^N)\right| > \varepsilon\right) \leq 2 \exp\left(-N \cdot \frac{\varepsilon^2}{8 \cdot (1 + \sum_{s = 1}^{N-1} \bar \eta_{\bold Z}(s))^2}\right),
\end{align}
where $\mu^N := \frac{1}{N} \sum_{n = 1}^N \delta_{Z_n}$, $\bar \eta_{\bold Z}$ is the tail $\eta$-mixing coefficient of $\bold Z$ and $c > 0$ is an absolute constant. If $\bold Z$ are independent, then we recover the subgaussian concentration rate for the Wasserstein distance on a compact space established, for instance, in \cite{fournier2015rate}. 
From this we also conclude that to obtain exponential concentration of $\mathcal{W}(\mu, \mu^N)$ from \eqref{eq:conc_wasserstein}, it is necessary to have $\bar \eta_{\bold Z}(s) \to 0$, as $s \to \infty$; otherwise the rhs of \eqref{eq:conc_wasserstein} goes to 1 when $N \to \infty$.
\end{example}

\begin{remark}
As the previous example shows, to get exponential concentration via Lemma \ref{lem:bounded_differences_mixing}, it is necessary that $\bar \eta_{\bold Z}(s) \to 0, \;\; s \to \infty$. The concentration rate in Lemma \ref{lem:bounded_differences_mixing} depends on the speed of this convergence. For instance, if $\bar \eta_{\bold Z}(s) \leq C \cdot s^{-\beta}$ for some $C, \beta > 0$ and $L = \frac{1}{N}$, then
$$
\mathbb{P}\left(|\phi(Z_{1:N}) - \mathbb{E}\, \phi(Z_{1:N})| > \varepsilon\right) \leq \begin{cases}
2 \exp \left(-c \cdot N^{2 \beta - 1} \cdot \varepsilon^2\right), &\beta \in (0, 1),\\
2 \exp \left(-c \cdot \tfrac{N}{(\log N)^2} \cdot \varepsilon^2\right), &\beta = 1,\\
2 \exp \left(-c \cdot N \cdot \varepsilon^2\right), &\beta > 1
\end{cases}
$$
for some constant $c > 0$ that depends on $\beta$; the faster the decay, the closer the rate to the optimal subgaussian rate, and it is achieved when $\beta > 1$.
\end{remark}

We emphasize that the concentration inequality obtained is quite general and can be used for other cases except for $\mathcal{AW}$. For instance, on a compact space one could estimate the concentration rate of the empirical average or obtain similar results for the Wasserstein distance as the Lipschitz mapping of the observations.

\subsection{Concentration of $\mathcal{AW}(\mu, \widehat \mu^N)$: compact case}
The concentration result for the compact case now follows from Lemma \ref{lem:bounded_differences_mixing}:
\begin{proof}[Proof of Theorem \ref{thm:AW_concentration_compact}]
Let $\widehat \mu^N(x) := \frac{1}{N} \sum_{n = 1}^N \delta_{\varphi^N(x^n)}$ be the adapted empirical measure of $x = (x^1, \ldots, x^N)$ with $x^n \in (\R^d)^T$, and define
\begin{align*}
\phi(x) := \mathcal{AW}\left(\mu, \widehat \mu^N(x)\right).
\end{align*}
We aim to check that $\phi$ satisfies bounded differences property. Take any $x = x^{1:N}$ and $y = y^{1:N}$ with $x^n, y^n \in (\R^d)^T$. Then by the triangle inequality for $\mathcal{AW}$,
\begin{align*}
|\phi(x) - \phi(y)|
&\leq \mathcal{AW}(\widehat \mu^N(x), \widehat \mu^N(y))\\
&\leq \operatorname{diam}(\operatorname{spt}(\mu)) \cdot T \cdot (2T - 1) \cdot \operatorname{TV}(\widehat \mu^N(x), \widehat \mu^N(y))\\
&\leq \operatorname{diam}(\operatorname{spt}(\mu)) \cdot T \cdot (2T - 1) \cdot \frac{1}{N} \sum_{n = 1}^N \mathds{1}_{\{x^n \neq y^n\}},
\end{align*}
where the second inequality follows from \cite[Corollary 2.7]{acciaio2025estimating}. Thus, $\phi$ has $L$-bounded differences with $L \leq \operatorname{diam}(\operatorname{spt}(\mu)) \cdot \frac{T(2T-1)}{N}$. The claim now follows from Lemma \ref{lem:bounded_differences_mixing}.
\end{proof}

\subsection{Concentration of $\mathcal{AW}(\mu, \widehat \mu^N)$: generalization}

We now proceed with the proof of Theorem \ref{thm:AW_concentration_general}.

\begin{proof}[Proof of Theorem \ref{thm:AW_concentration_general}]
For this proof we use a compact-approximation argument, which will also be reused in the consistency results below. For this reason, we start with a preparation step and then outline the proof strategy. \\
According to Lemma \ref{lem:compact_approximation}, for any $R \ge \sqrt{dT}\,\Delta_N$, where $\Delta_N$ is a grid size, there exists a mapping $\kappa_R: \mathbb{R}^d \to B_{2R}(0) \subset \R^d$, such that
\begin{enumerate}
    \item For any probability measure $\nu \in \mathcal{P}((\mathbb{R}^d)^T)$ we have
    $$
    \mathcal{AW}(\nu, \nu_R) \leq \sqrt{T} \int_{\{\|x\| \geq R\}} \|x\| \, \nu(dx),
    $$
    where $\nu_R : = (\kappa_R)_\# \nu$.
    \item On $B_{\frac{R}{2}}(0)$ we have $\kappa_R \circ \varphi^N = \varphi^N \circ \kappa_R$. Consequently, for $R$ large the truncated empirical measures $\widehat{\mu_R}^N$ and $(\widehat\mu^N)_R$ agree with high probability, allowing us to invoke the compact-case result.
\end{enumerate}

Fix $R \geq \max(N^\delta, \sqrt{dT}\Delta_N)$ for some $\delta > 0$, recall the mapping $\kappa_R$ introduced above, and proceed with the proof. The value of $\delta$ will be chosen later. By the triangle inequality for $\mathcal{AW}$, 
\begin{align}
\begin{split}\label{eq:explain_tau}
|\mathcal{AW}(\mu, \widehat \mu^N) - \mathcal{AW}(\mu_R, \widehat{\mu_R}^N)| 
&= |\mathcal{AW}(\mu, \widehat \mu^N) - \mathcal{AW}( \widehat \mu^N,\mu_R) + \mathcal{AW}( \widehat \mu^N,\mu_R) - \mathcal{AW}(\mu_R, (\widehat \mu^N)_R)\\
&\quad+ \mathcal{AW}(\mu_R, (\widehat \mu^N)_R) - \mathcal{AW}(\mu_R, \widehat{\mu_R}^N)| \\
&\leq \mathcal{AW}(\mu, \mu_R) + \mathcal{AW}(\widehat \mu^N, (\widehat \mu^N)_R) + \mathcal{AW}((\widehat \mu^N)_R, \widehat {\mu_R}^N) =: \tau^N_R.
\end{split}
\end{align}
Thus,
\begin{align}\label{eqn:concentration_general_decomposition}
\begin{split}
& |\mathcal{AW}(\mu, \widehat \mu^N) - \mathbb{E}\,\mathcal{AW}(\mu, \widehat \mu^N)| \leq |\mathcal{AW}(\mu_R, \widehat{\mu_R}^N) - \mathbb{E}\,\mathcal{AW}(\mu_R, \widehat{\mu_R}^N)| + \tau^N_R + \mathbb{E}\, \tau^N_R.
\end{split}
\end{align}
The strategy of the proof is to estimate the deviation probability as follows:
\begin{align}\label{eqn:concentration_general_plan}
\begin{split}
&\mathbb{P}\left(\left|\mathcal{AW}(\mu, \widehat \mu^N) - \mathbb{E}\,\mathcal{AW}(\mu, \widehat \mu^N)\right| > 2 \mathbb{E}\,\tau^N_R + \varepsilon\right)\\
&\overset{\eqref{eqn:concentration_general_decomposition}}{\leq} \mathbb{P}\left(\left|\mathcal{AW}(\mu_R, \widehat{\mu_R}^N) - \mathbb{E}\,\mathcal{AW}(\mu_R, \widehat{\mu_R}^N)\right| + \tau^N_R > \mathbb{E}\, \tau^N_R + \varepsilon\right)\\
&\leq \mathbb{P}\left(\left|\mathcal{AW}(\mu_R, \widehat{\mu_R}^N) - \mathbb{E}\,\mathcal{AW}(\mu_R, \widehat{\mu_R}^N)\right| > \varepsilon\right) + \mathbb{P}\left(\tau^N_R > \mathbb{E}\, \tau^N_R\right),
\end{split}
\end{align}
where the second inequality follows from union bound. We proceed by deriving the deviation probabilities for $\mathcal{AW}(\mu_R, \widehat{\mu_R}^N)$ and $\tau^N_R$ separately, along with a bound for $\mathbb{E}\, \tau^N_R$.

\begin{itemize}
    \item We start with $\mathcal{AW}(\mu_R, \widehat{\mu_R}^N)$. Applying Theorem \ref{thm:AW_concentration_compact} with $\mu = \mu_R$ and observations $\kappa_R(\slices X) := (\kappa_R(X^n))_{n = 1}^N$, we obtain
    \begin{align}\label{eqn:concentration_general_compact}
    \begin{split}
    \mathbb{P}\left(|\mathcal{AW}(\mu_R, \widehat{\mu_R}^N) - \mathbb{E}\,\mathcal{AW}(\mu_R, \widehat{\mu_R}^N)| > \varepsilon\right) &\leq 2 \exp\left(-c \cdot \frac{N}{R^2} \cdot \frac{\varepsilon^2}{(1 + \sum_{s = 1}^{N-1} \bar \eta_{\kappa_R(\slices X)}(s))^2}\right) \\
    &\leq 2 \exp\left(-c \cdot \frac{N}{R^2} \cdot \frac{\varepsilon^2}{(1 + \sum_{s = 1}^{N-1} \bar \eta_{\slices X}(s))^2}\right),
    \end{split}
    \end{align}
    where $c > 0$ depends on $T$; the second inequality follows from Proposition \ref{prop:eta_mixing_information_processing} applied with $f = \kappa_R$.
    \item To estimate $\tau^N_R$, recall that $\mathcal{E}_{\alpha, \gamma}(\mu) < \infty$ for some $\alpha ,\gamma > 0$, which in particular implies that $\mu \in \mathcal{P}_{1 + 1/\delta}((\R^d)^T)$. Thus, by Lemma \ref{lem:tau_estimate} applied with $p = 1 + \frac{1}{\delta}$,
    \begin{align}
    &\mathbb{E}\, \tau^N_R \leq \frac{C}{R^{1/\delta}} \leq \frac{C}{N} \leq C \cdot \Delta_N \leq C \cdot \varepsilon,\label{eqn:AW_conc_tau_properties_1}\\
    &\mathbb{P}\!\left(\tau^N_R > \mathbb{E}\, \tau^N_R\right) \leq \mathcal{E}_{\alpha, \gamma}(\mu) \cdot N \exp\!\left(-c \cdot R^\alpha\right),\label{eqn:AW_conc_tau_properties_2}
    \end{align}
    where $C > 0$ depends on $d, T, \alpha, \gamma, \mathcal{E}_{\alpha, \gamma}(\mu), \delta$ and $c > 0$ depends on $\alpha, \gamma$; we have also used that $R \geq N^\delta$ and $\frac{1}{N} \leq \Delta_N \leq \varepsilon$ for \eqref{eqn:AW_conc_tau_properties_1}. 
\end{itemize}

We now choose $R := \sqrt{dT} N^{\frac{1}{\alpha + 2}} \varepsilon^{\frac{2}{\alpha + 2}}$ and consequently set $\delta := \frac{1}{3(\alpha+2)}$. We first check that $R \geq N^\delta$: 
$$
\varepsilon \geq \Delta_N \geq N^{-\frac{1}{3}} \implies \tfrac{R}{\sqrt{dT}} \geq N^{\frac{1}{\alpha + 2}} \varepsilon^{\frac{2}{\alpha + 2}} \geq N^\frac{1}{\alpha+2} \Delta_N^\frac{2}{\alpha+2} \geq N^\frac{1}{3(\alpha+2)} = N^\delta.
$$
As a consequence, since $\Delta_N \leq 1$ we also have $R \geq \sqrt{dT}\Delta_N$. Thus, we write
\begin{align}
\begin{split}
&\mathbb{P}\left(\left|\mathcal{AW}(\mu, \widehat \mu^N) - \mathbb{E}\, \mathcal{AW}(\mu, \widehat \mu^N)\right| > C \cdot \varepsilon \right)\\
&\overset{\eqref{eqn:AW_conc_tau_properties_1}}{\leq} \mathbb{P}\left(\left|\mathcal{AW}(\mu, \widehat \mu^N) - \mathbb{E}\, \mathcal{AW}(\mu, \widehat \mu^N)\right| > 2\mathbb{E}\, \tau^N_R + \varepsilon \right)\\
&\overset{\eqref{eqn:concentration_general_plan}}{\leq}  \mathbb{P}\left(\left|\mathcal{AW}(\mu_R, \widehat{\mu_R}^N) - \mathbb{E}\,\mathcal{AW}(\mu_R, \widehat{\mu_R}^N)\right| > \varepsilon\right) + \mathbb{P}\left(\tau^N_R > \mathbb{E}\,\tau^N_R\right)\\
&\overset{\substack{\eqref{eqn:concentration_general_compact}\\ \eqref{eqn:AW_conc_tau_properties_2}}}{\leq} 2 \exp\left(-c \cdot \frac{N}{R^2} \cdot \frac{\varepsilon^2}{(1 + \sum_{s = 1}^{N-1} \bar \eta_{\slices X}(s))^2}\right) + \mathcal{E}_{\alpha, \gamma}(\mu) \cdot N\exp(-c \cdot R^\alpha)\\
&\leq 2 \exp\left(-\frac{c}{(1 + \sum_{s = 1}^{N-1} \bar \eta_{\slices X}(s))^2} N^\frac{\alpha}{\alpha+2} \varepsilon^\frac{2\alpha}{\alpha+2}\right) + \mathcal{E}_{\alpha, \gamma}(\mu) \cdot N\exp(-c \cdot N^\frac{\alpha}{\alpha+2} \varepsilon^\frac{2\alpha}{\alpha+2}).
\end{split}
\end{align}
where $C, c > 0$ depend on $\alpha, \gamma, \mathcal{E}_{\alpha, \gamma}(\mu), d, T$. The proof is complete.
\end{proof}

\section{Consistency}\label{sec:consistency}

We first focus on the case where $\mu$ has compact support. A standard argument to establish consistency of $\widehat\mu^N$ proceeds as follows: derive a deviation-from-zero inequality for $\mathcal{AW}(\mu, \widehat \mu^N)$, which is summable with respect to $N$, and then apply the Borel-Cantelli Lemma; phrased mathematically, for $\varepsilon > 0$,
\begin{align*}
\mathbb{P}\left(\mathcal{AW}(\mu, \widehat \mu^N) > \varepsilon\right) \leq f(N, \varepsilon) \text{  and } \sum_{N = 1}^\infty f(N, \varepsilon) < \infty \implies \limsup_{N \to +\infty} \mathcal{AW}(\mu, \widehat \mu^N) \leq \varepsilon \end{align*}
$\mathbb{P}$-almost surely. Such a recipe can be implemented if $\mathbb{P}\left(|\mathcal{AW}(\mu, \widehat \mu^N) - \mathbb{E}\, \mathcal{AW}(\mu, \widehat \mu^N)| > \varepsilon\right) \leq f(N, \varepsilon)$ and $\mathbb{E}\, \mathcal{AW}(\mu, \widehat \mu^N) \to 0$ as $N \to +\infty$. The former inequality is stated in Theorem \ref{thm:AW_concentration_compact} assuming appropriate decay of $s \mapsto \eta_{\bold Z}(s)$ and integrability of $\mu$. 

\subsection{Continuous kernels}
To prove consistency in the case when $x_{1:t} \mapsto \mu_{x_{1:t}}$ admits a continuous version, we take similar route as \cite[Lemma 5.3]{backhoff2022estimating}.

\begin{lemma}\label{lem:compact_continuous_kernels_expectation_convergence}
Let $\mu \in \mathcal{P}((\R^d)^T)$ be compactly supported with continuous kernels, and let $\slices X \sim \mu$. Assume that $\sum_{s = 1}^\infty \bar \eta_{\slices X}(s) < \infty$. Then
$$
\mathbb{E}\, \mathcal{AW}(\mu, \widehat \mu^N) \to 0, \;\; N \to +\infty.
$$
\end{lemma}
\begin{proof}
Assume without loss of generality that $\operatorname{spt}(\mu) \subset ([0, 1]^d)^T$ and fix $\varepsilon > 0$. Since $\mu$ has continuous kernels, by \cite[Lemma 5.1, Lemma 5.2]{backhoff2022estimating} there exists $N_0 = N_0(\varepsilon)$, such that
\begin{equation}\label{eqn:consistency_compact_continuous_delta_approx}
\mathcal{AW}(\mu, \widehat \mu^N) \leq \varepsilon + C \cdot \left(\sum_{t = 1}^{T - 1} \sum_{G \in \Phi^N_t} \mu^N(G) \cdot \mathcal{W}(\mu_G, \mu^N_G) + \mathcal{W}(\mu^1, (\widehat \mu^N)^1)\right)
\end{equation}
$\mathbb{P}$-almost surely for any $N \geq N_0(\varepsilon)$, where $C > 0$ depends on $\varepsilon$. Using Lemma \ref{lem:sum_mu_G_W_mu_G_mu_N_G}, we write
\begin{equation}\label{eqn:consistency_compact_continuous_rate}
\mathbb{E}\Big[ \sum_{t = 1}^{T-1} \sum_{G \in \Phi^N_t} \mu^N(G) \cdot \mathcal{W}(\mu_G, \mu^N_G) \Big] \leq C \cdot \sqrt{1 + 2 \sum_{s = 1}^{N-1} \eta_{\slices X}(s)} \cdot \operatorname{rate}_\infty(N) \leq C \cdot \operatorname{rate}_\infty(N),
\end{equation}
where $C > 0$ depends on $d, T$. Therefore
\begin{equation}\label{eqn:consistency_compact_exp}
\mathbb{E} \, \mathcal{AW}(\mu, \widehat \mu^N) \overset{\eqref{eqn:consistency_compact_continuous_delta_approx}, \eqref{eqn:consistency_compact_continuous_rate}}{\leq} \varepsilon + C \cdot \operatorname{rate}_\infty(N),
\end{equation}
for $N \geq N_0(\varepsilon)$, where $C > 0$ depends on $\varepsilon, d, T, \alpha$. Taking $\limsup$ in \eqref{eqn:consistency_compact_exp} and using $\operatorname{rate}_\infty(N) \to 0$, as $N \to \infty$, we obtain
$$
\limsup_{N \to \infty} \mathbb{E}\, \mathcal{AW}(\mu, \widehat \mu^N) \leq \varepsilon 
$$
and thus $\mathbb{E} \, \mathcal{AW}(\mu, \widehat \mu^N) \to 0$ for $N \to \infty$
by letting $\varepsilon \to 0$. This concludes the proof.
\end{proof}

\begin{proof}[Proof of Theorem \ref{thm:consistency_general}, (1)]
Fix $\varepsilon > 0$. According to Theorem \ref{thm:AW_concentration_compact},
\begin{equation}\label{eqn:consistency_compact_continuous_prefinal}
\mathbb{P}\left(|\mathcal{AW}(\mu, \widehat \mu^N) - \mathbb{E}\, \mathcal{AW}(\mu, \widehat \mu^N)| > \varepsilon\right) \leq 2 \exp\left(-c \cdot N \cdot \frac{\varepsilon^2}{(1 + \sum_{s = 1}^{N-1} \bar \eta_{\slices X}(s))^2}\right) \leq 2 \exp \left(-c \cdot N \cdot \varepsilon^2\right),
\end{equation}
where $c > 0$ depends on $T, \operatorname{diam}(\operatorname{spt}(\mu))$. As the right-hand side of \eqref{eqn:consistency_compact_continuous_prefinal} is summable, by Borel-Cantelli Lemma 
\begin{align*}
&\mathbb{P}\left(\limsup_{N \to \infty} |\mathcal{AW}(\mu, \widehat \mu^N) - \mathbb{E}\, \mathcal{AW}(\mu, \widehat \mu^N)| \leq \varepsilon\right) = 1
\end{align*}
and thus 
\begin{align}\label{eqn:consistency_compact_continuous_deviation}
&\implies |\mathcal{AW}(\mu, \widehat \mu^N) - \mathbb{E}\, \mathcal{AW}(\mu, \widehat \mu^N)| \to 0, \;\; N \to \infty,
\end{align}
$\mathbb{P}$-almost surely by letting $\varepsilon \to 0$. The claim then follows from Lemma \ref{lem:compact_continuous_kernels_expectation_convergence} and \eqref{eqn:consistency_compact_continuous_deviation}.
\end{proof}

\subsection{Preliminary results for the case of general kernels}

Following the same strategy without any continuity assumptions on the kernels of $\mu$ seems infeasible to us. In fact, the proof for the i.i.d.~case involves an approximation argument via Lusin's theorem, where (possibly non-continuous) disintegrations are replaced with their continuous modifications, and the observations $\slices X$ are modified accordingly, see \cite[Proof of Theorem 1.3]{backhoff2022estimating}. However it is not obvious how to extend this argument if weak dependency between the observations is introduced.

Thus, in order to establish consistency in the non-i.i.d.~case without continuous kernels, we apply smoothing with a Gaussian kernel $\mathcal{N}(0, \sigma_N^2 I_{dT})$ for $\widehat \mu^N$, where the smoothing parameter $\sigma_N \to 0$ is chosen as a function of the decay of $\eta_{\bold Z}(s)$ as $s \to \infty$, and the grid size. Such a modification allows us to work with smoothed distances, which are well known to improve the convergence speed. Throughout this section we denote $\mathcal{N}(0, \sigma^2 I_{dT})$ by $\mathcal{N}_\sigma$ and its density by $\varphi_\sigma$.
\begin{definition}\label{def:weighted_TV}
For any $\mu, \nu \in \mathcal{P}_1(\R^k)$, $k\ge 1$, define the weighted total variation distance (\cite[Definition 2.2]{hou2024convergence}) as
$$
\operatorname{TV}_1(\mu, \nu) := \int \left(\|x\| + \tfrac{1}{2}\right) |\mu - \nu|(dx) = \inf_{\pi \in \Pi(\mu, \nu)} \int \left(1 + \|x\| + \|y\|\right) \mathds{1}_{\{x \neq y\}} \, \pi(dx, dy).
$$
\end{definition}

\begin{definition}\label{def:weighted_smoothed_TV}
For any probability measures $\mu, \nu \in \mathcal{P}_1((\mathbb{R}^d)^T)$ and $\sigma > 0$ define the weighted smoothed total variation as
$$
\operatorname{TV}^{(\sigma)}_1(\mu, \nu) := \operatorname{TV}_1(\mu \ast \mathcal{N}_\sigma, \nu \ast \mathcal{N}_\sigma).
$$
Similarly, define the smoothed Adapted Wasserstein distance as in \cite[Definition 1]{larsson2025fast}:
$$
\mathcal{AW}^{(\sigma)}(\mu, \nu) := \mathcal{AW}(\mu \ast \mathcal{N}_\sigma, \nu \ast \mathcal{N}_\sigma).
$$
\end{definition}
We now outline the plan for the proof of consistency in more detail. First, by the triangle inequality for $\mathcal{AW}$,
\begin{equation}\label{eqn:AW_consistency_triangle_inequality}
\mathcal{AW}(\mu, \widehat \mu^N \ast \mathcal{N}_{\sigma_N}) \leq \mathcal{AW}(\mu, \mu \ast \mathcal{N}_{\sigma_N}) + \mathcal{AW}^{(\sigma_N)}(\mu, \widehat \mu^N),
\end{equation}
where the first term $\mathcal{AW}(\mu, \mu \ast \mathcal{N}_{\sigma_N})$ is deterministic and converges to zero as long as $\sigma_N \to 0$, see \cite[Theorem 5.5]{hou2024convergence}. To deal with the second term, we use the bound (see \cite[Theorem 3.6, Lemma A.1 - (ii)]{hou2024convergence}) \begin{equation}\label{eqn:AW_leq_TV1}
\mathcal{AW}^{(\sigma)}(\mu, \nu) \leq C \cdot \operatorname{TV}^{(\sigma)}_1(\mu, \nu)
\end{equation}
for compactly supported measures $\mu, \nu \in \mathcal{P}(K^T)$ and $\sigma \in (0, 1]$, where $C > 0$ depends on $T, \operatorname{diam}(K)$. We then proceed to estimate $\operatorname{TV}^{(\sigma)}_1$, as it can be controlled using the covariance bound for the $\eta$-mixing coefficient stated in Proposition \ref{prop:eta_mixing_covariance}. 
Some of our convergence and concentration results extend the estimates from \cite{hou2024convergence} obtained for the i.i.d.~observations.

\subsection{Weighted total variation bounds and concentration}\label{subsec:weighted_total_variation_bounds}

\begin{proposition}\label{prop:weighted_total_variation_bound}
Recall the function $\varphi^N$ from \eqref{eqn:adapted_empirical_measure} and let $\mu \in \mathcal{P}((\mathbb{R}^d)^T)$ be compactly supported. Then we have
\begin{equation}\label{eqn:weighted_total_variation_bound}
\operatorname{TV}^{(\sigma)}_1(\mu, (\varphi^N)_\# \mu) \leq C \cdot (\sigma^{-1} \cdot \Delta_N + \sigma^{-2} \cdot \Delta_N^2)
\end{equation}
for any $\sigma \in (0, 1]$, where $C > 0$ depends on $d, T, \operatorname{diam}(\operatorname{spt}(\mu))$.
\end{proposition}
\begin{proof}
Observe that
\begin{equation}\label{eqn:mu_varphi_mu_definition}
\mu = \int \delta_x \, \mu(dx), \;\; (\varphi^N)_\# \mu = \int \delta_{\varphi^N(x)} \, \mu(dx).
\end{equation}
Then by convexity of $\operatorname{TV}^{(\sigma)}_1$ (note that $\eta \ast \int \mu_x \, \nu(dx) = \int (\eta \ast \mu_x) \, \nu(dx)$ and $\operatorname{TV}_1$ is convex by \cite[Theorem 4.8]{villani2009optimal}), Proposition \ref{prop:TV_1_gaussians_bound} and $\|x - \varphi^N(x)\| \leq \sqrt{dT} \cdot \Delta_N$ 
we have
\begin{align*}
\operatorname{TV}^{(\sigma)}_1(\mu, (\varphi^N)_\# \mu) \overset{\eqref{eqn:mu_varphi_mu_definition}}{\leq} \int \operatorname{TV}^{(\sigma)}_1(\delta_x, \delta_{\varphi^N(x)}) \, \mu(dx) \leq C \cdot (1 + \operatorname{diam}(\operatorname{spt}(\mu))) \cdot (\sigma^{-1} \cdot \Delta_N + \sigma^{-2} \cdot \Delta_N^2),
\end{align*}
where $C > 0$ depends on $d, T$. The proof is complete.
\end{proof}

\begin{lemma}\label{lem:weighted_total_variation_empirical}
Let $\mu \in \mathcal{P}((\mathbb{R}^d)^T)$ be compactly supported, and let $\slices{X} = (X^n)_{n = 1}^N \sim \mu$. Then
$$
\mathbb{E}\big[ \operatorname{TV}^{(\sigma)}_1(\mu, \mu^N)\big] \leq C \cdot \sqrt{1 + 2 \sum_{s = 1}^{N-1} \eta_{\slices{X}}(s)} \cdot \sigma^{-\frac{dT}{2}} \cdot N^{-\frac{1}{2}},
$$
where $C > 0$ depends on $d, T, \operatorname{diam}(\operatorname{spt}(\mu))$.
\end{lemma}
\begin{proof}
Let $q$ and $q^N$ be the densities of $\mu \ast \mathcal{N}_\sigma$ and $\mu^N \ast \mathcal{N}_\sigma$ respectively, i.e.
\begin{align*}
q(x) = \int \varphi_\sigma(x - y) \,\mu(dy)  \quad \text{and}\quad  q^N(x) = \frac{1}{N} \sum_{n = 1}^N \varphi_\sigma(x - X^n).
\end{align*}
This implies in particular
\begin{align}\label{eqn:densities_mu_empirical_ast}
\mathbb{E}\big[ q^N(x)\big] = q(x).
\end{align}
Hence, defining $f(x) := \frac{1}{\|x\|^3 + 1}$ and using Fubini's theorem as well as Cauchy-Schwarz and Jensen's inequality we have 
\begin{align}\label{eqn:weighted_total_variation_empirical}
\begin{split}
\mathbb{E}\big[ \operatorname{TV}^{(\sigma)}_1(\mu, \mu^N) \big]&= \mathbb{E}\Big[\int \left(\|x\| + \tfrac{1}{2}\right) |q(x) - q^N(x)| \, dx\Big]\\
&\leq \sqrt{\int \left(\|x\| + \tfrac{1}{2}\right)^2  f(x) \, dx} \cdot \sqrt{\int \frac{\mathbb{E}[|q(x) - q^N(x)|^2]}{f(x)} dx}\\
&\overset{\eqref{eqn:densities_mu_empirical_ast}}{=} C \cdot \sqrt{\int \frac{\operatorname{Var} \, [q^N(x)]}{f(x)} dx},
\end{split}
\end{align}
where $C > 0$ depends on $d, T$. Applying Proposition \ref{prop:eta_mixing_variance} with $f(\cdot) := \varphi_\sigma(x - \cdot)$ yields
\begin{align}\label{eqn:smooth_TV_variance_qn}
\begin{split}
\operatorname{Var}\,[q^N(x)]
&\leq \frac{1}{N} \cdot (2 \pi \sigma)^{-\frac{dT}{2}} \cdot \mathbb{E}\left[\varphi_\sigma(x - X^1)\right] \cdot \left(1 + 2 \sum_{s = 1}^{N-1} \eta_{\slices X}(s)\right)\\
&\leq \frac{C}{N} \cdot \sigma^{-dT} \cdot \left(1 + 2 \sum_{s = 1}^{N-1} \eta_{\slices{X}}(s)\right) \mathbb{E}\Big[\exp\left(-\frac{\|x - X^1\|^2}{2 \sigma^2}\right)\Big],
\end{split}
\end{align}
where $C > 0$ depends on $d, T$. As $\operatorname{Law}(X^1)$ is compactly supported and $\sigma \in (0, 1]$, we have
\begin{equation}\label{eqn:TV_smooth_sqrt_bound}
\int \mathbb{E}\Big[\exp\left(-\frac{\|x - X^1\|^2}{2 \sigma^2}\right)\Big] \cdot \frac{1}{f(x)} dx \leq C,
\end{equation}
where $C > 0$ depends on $d, T, \operatorname{diam}(\operatorname{spt}(\mu))$. Thus,
\begin{align*}
\mathbb{E}\big[\operatorname{TV}^{(\sigma)}_1(\mu, \mu^N)\big]&\overset{\eqref{eqn:weighted_total_variation_empirical}}{\leq} C \cdot \sqrt{\int \frac{\operatorname{Var} \, [q^N(x)]}{f(x)} dx}\\
&\overset{\eqref{eqn:smooth_TV_variance_qn}}{\leq} C \cdot \sqrt{1 + 2 \sum_{s = 1}^{N-1} \eta_{\slices X}(s)} \cdot \sigma^{-\frac{dT}{2}} \cdot N^{-\frac{1}{2}} \cdot \sqrt{\int \mathbb{E}\Big[\exp\left(-\frac{\|x - X^1\|^2}{2 \sigma^2}\right)\Big] \cdot \frac{1}{f(x)} dx}\\
&\overset{\eqref{eqn:TV_smooth_sqrt_bound}}{\leq} C \cdot \sqrt{1 + 2 \sum_{s = 1}^{N-1} \eta_{\slices X}(s)} \cdot \sigma^{-\frac{dT}{2}} \cdot N^{-\frac{1}{2}},
\end{align*}
where $C > 0$ depends on $d, T, \operatorname{diam}(\operatorname{spt}(\mu))$. The proof is complete.
\end{proof}

\begin{lemma}[Deviation from the mean of smoothed weighted total variation]\label{lem:deviation_compact_TV_sigma_1}
Let $\mu \in \mathcal{P}((\mathbb{R}^d)^T)$ be compactly supported, and let $\slices{X} = (X^n)_{n = 1}^N \sim \mu$. Then for any $\sigma \in (0, 1]$
$$
\mathbb{P}\left(|\operatorname{TV}^{(\sigma)}_1(\mu, \mu^N) - \mathbb{E}\,\operatorname{TV}^{(\sigma)}_1(\mu, \mu^N)| > \varepsilon\right) \leq 2 \exp\left(-c \cdot \sigma^4 N \cdot \frac{\varepsilon^2}{(1 + \sum_{s = 1}^{N-1} \bar \eta_{\slices X}(s))^2}\right),
$$
where $c > 0$ that depends on $d, T, \operatorname{diam}(\operatorname{spt}(\mu))$.
\end{lemma}
\begin{proof}
The proof follows from application of Lemma \ref{lem:bounded_differences_mixing}. For $x = (x^1, \ldots, x^N)$ with $x^n \in (\R^d)^T$ denote $\frac{1}{N} \sum_{n = 1}^N \delta_{x^n}$ by $\mu^N(x)$, and define the mapping $x \mapsto \phi(x) := \operatorname{TV}^{(\sigma)}_1(\mu, \mu^N(x))$. Then for any $x, \hat x \in ((\R^d)^T)^N$
\begin{equation*}
|\phi(x) - \phi(\hat x)| \leq \operatorname{TV}^{(\sigma)}_1(\mu^N(x), \mu^N(\hat x)) \leq \frac{1}{N} \sum_{n = 1}^N \operatorname{TV}^{(\sigma)}_1(\delta_{x^n}, \delta_{\hat x^n}) \leq \frac{C}{N} \cdot \sigma^{-2} \sum_{n = 1}^N \mathds{1}_{\{x^n \neq \hat x^n\}},
\end{equation*}
where $C > 0$ depends on $d, T, \operatorname{diam}(\operatorname{spt}(\mu))$; the first inequality follows from triangle inequality for $\operatorname{TV}^{(\sigma)}_1$, the second inequality holds by convexity of $\operatorname{TV}^{(\sigma)}_1$ and the final inequality follows from Proposition \ref{prop:TV_1_gaussians_bound}. In consequence, $\phi$ is $L$-Lipschitz with respect to the Hamming distance, where $L \leq \frac{C}{N} \sigma^{-2}$. Thus, by Lemma \ref{lem:bounded_differences_mixing}
\begin{align*}
\mathbb{P}\left(|\operatorname{TV}^{(\sigma)}_1(\mu, \mu^N) - \mathbb{E}\,\operatorname{TV}^{(\sigma)}_1(\mu, \mu^N)| > \varepsilon\right) &\leq 2 \exp\left(-c \cdot \sigma^4 N \cdot \frac{\varepsilon^2}{(1 + \sum_{s = 1}^{N-1} \bar \eta_{\slices X}(s))^2}\right),
\end{align*}
where $c > 0$ depends on $d, T, \operatorname{diam}(\operatorname{spt}(\mu))$. The proof is complete.
\end{proof}

\subsection{Consistency of smoothed weighted total variation}
We now build on the results obtained in Section \ref{subsec:weighted_total_variation_bounds}, and prove consistency of $\operatorname{TV}^{(\sigma_N)}_1(\mu, \mu^N)$ for appropriately chosen sequence $\sigma_N \to 0$ as $N \to \infty$.

\begin{lemma}[Consistency of deviation of smoothed weighted total variation]\label{lem:TV_sigma_1_consistency_deviation}
Let $\mu \in \mathcal{P}((\mathbb{R}^d)^T)$ be compactly supported, and let $\slices{X} = (X^n)_{n = 1}^\infty \sim \mu$. Assume that $\sum_{s = 1}^\infty \bar \eta_{\slices X}(s) < \infty$. Let $\sigma_N \in [N^{-\frac{1}{8}}, 1]$. Then
$$
\operatorname{TV}^{(\sigma_N)}_1(\mu, \mu^N) - \mathbb{E}\big[ \operatorname{TV}^{(\sigma_N)}_1(\mu, \mu^N)\big] \to 0, \qquad N \to +\infty,
$$
$\mathbb{P}$-almost surely.
\end{lemma}
\begin{proof}
First, observe that
\begin{equation}\label{eqn:sigma_N_consequence_2}
(\sigma_N)^4 N \geq \sqrt{N}
\end{equation}
due to the choice of $\sigma_N$. Thus, for each $\varepsilon > 0$ by Lemma \ref{lem:deviation_compact_TV_sigma_1} we have
\begin{align*}
&\mathbb{P}\left(\big|\operatorname{TV}^{(\sigma_N)}_1(\mu, \mu^N) - \mathbb{E}\big[\operatorname{TV}^{(\sigma_N)}_1(\mu, \mu^N)\big]\big| \ge \varepsilon\right) \overset{\eqref{eqn:sigma_N_consequence_2}}{\leq} 2 \exp\left(-c \cdot \sqrt{N} \cdot \varepsilon^2\right),
\end{align*}
where $C, c > 0$ depend on $d, T, \operatorname{diam}(\operatorname{spt}(\mu))$; we have also used that $\sum_{s = 1}^\infty \bar \eta_{\slices X}(s) < \infty$. In particular,
$$
\sum_{N = 1}^\infty \mathbb{P}\left(\big|\operatorname{TV}^{(\sigma_N)}_1(\mu, \mu^N) - \mathbb{E}\big[ \operatorname{TV}^{(\sigma_N)}_1(\mu, \mu^N)\big]\big| \ge \varepsilon\right) < \infty.
$$
Therefore, $$\mathbb{P}\left(\limsup_{N \to \infty} \big|\operatorname{TV}^{(\sigma_N)}_1(\mu, \mu^N) - \mathbb{E}\big[ \operatorname{TV}^{(\sigma_N)}_1(\mu, \mu^N)\big]\big| \leq \varepsilon\right) = 1$$ by the Borel-Cantelli Lemma. The statement of the theorem then follows by letting $\varepsilon \to 0$.
\end{proof}

\begin{lemma}[Consistency of smoothed weighted total variation]\label{lem:TV_sigma_1_consistency}
Let $\mu \in \mathcal{P}((\mathbb{R}^d)^T)$ be compactly supported, and let $\slices{X} = (X^n)_{n = 1}^N \sim \mu$. Assume that $\sum_{s = 1}^\infty \bar \eta_{\slices X}(s) < \infty$. Let $\sigma_N = \max(\sqrt{\Delta_N}, N^{-\frac{1}{8}})$. Then
$$
\operatorname{TV}^{(\sigma_N)}_1(\mu, \mu^N) \to 0, \qquad N \to \infty,
$$
$\mathbb{P}$-almost surely.
\end{lemma}
\begin{proof}
Applying Lemma \ref{lem:weighted_total_variation_empirical}, we obtain
\begin{align}\label{eqn:AW_consistency_compact_3_prefinal_expectation}
\begin{split}
\mathbb{E}\big[ \operatorname{TV}^{(\sigma_N)}_1(\mu, \mu^N)\big] &\leq C \cdot \sqrt{1 + 2\sum_{s = 1}^{N-1} \eta_{\slices{X}}(s)} \cdot \sigma_N^{-\frac{dT}{2}} \cdot N^{-\frac{1}{2}} \\
&\leq C \cdot \sigma_N^{-\frac{dT}{2}} \cdot N^{-\frac{1}{2}} \leq C \cdot N^{-\frac{1}{4}} \to 0, \qquad N \to \infty,
\end{split}
\end{align}
where $C > 0$ depends on $d, T, \operatorname{diam}(\operatorname{spt}(\mu))$; the second inequality follows from $\sum_{s = 1}^\infty \bar \eta_{\slices X}(s) < \infty$, and the third inequality holds as $\sigma_N \geq N^{-\frac{r}{2}}$. Moreover, by Lemma \ref{lem:TV_sigma_1_consistency_deviation} we have
\begin{equation}\label{eqn:AW_consistency_compact_3_prefinal}
\operatorname{TV}^{(\sigma_N)}_1(\mu, \mu^N) - \mathbb{E}\big[ \operatorname{TV}^{(\sigma_N)}_1(\mu, \mu^N)\big] \to 0, \qquad N \to \infty,
\end{equation}
$\mathbb{P}$-almost surely. Thus, the claim follows from \eqref{eqn:AW_consistency_compact_3_prefinal_expectation} and \eqref{eqn:AW_consistency_compact_3_prefinal}.
\end{proof}

\subsection{Non-continuous kernels, compact case}

We are now in a position to present the proof of consistency when the kernels of $\mu$ are not assumed to admit $\mathcal{W}$-continuous version. To address this issue we apply smoothing with a Gaussian kernel $\mathcal{N}_{\sigma_N}$, where $\sigma_N \to 0$ is appropriately scaled with respect to the sample size and the mixing properties of $\slices X$. 

\begin{proof}[Proof of Theorem \ref{thm:consistency_general}, (2)]
By \eqref{eqn:AW_consistency_triangle_inequality} and \eqref{eqn:AW_leq_TV1} we have
\begin{equation}\label{eqn:AW_consistency_compact_decomposition}
\mathcal{AW}(\mu, \widehat \mu^N \ast \mathcal{N}_{\sigma_N}) \leq \underbrace{\mathcal{AW}(\mu, \mu \ast \mathcal{N}_{\sigma_N})}_{(1)} + C \cdot \underbrace{\operatorname{TV}^{(\sigma_N)}_1(\mu, \mu^N)}_{(2)} + C \cdot \underbrace{\operatorname{TV}^{(\sigma_N)}_1(\mu^N, \widehat \mu^N)}_{(3)}
\end{equation}
where $C > 0$ depends on $T, \operatorname{diam}(\operatorname{spt}(\mu))$. We now estimate three terms on the right-hand side of \eqref{eqn:AW_consistency_compact_decomposition} separately:
\begin{enumerate}
    \item According \cite[Theorem 5.5]{hou2024convergence}, we have
    \begin{equation}\label{eqn:AW_consistency_compact_1}
    \mathcal{AW}(\mu, \mu \ast \mathcal{N}_{\sigma_N}) \to 0, \;\; N \to +\infty,    
    \end{equation}
    since $\sigma_N \to 0$.
    \item Lemma \ref{lem:TV_sigma_1_consistency} yields
    \begin{equation}\label{eqn:AW_consistency_compact_2}
    \operatorname{TV}^{(\sigma_N)}_1(\mu, \mu^N) \to 0, \;\; N \to +\infty,
    \end{equation}
    $\mathbb{P}$-almost surely.
    \item Observe that $\widehat \mu^N = (\varphi^N)_\# \mu^N$. Applying Proposition \ref{prop:weighted_total_variation_bound} with $\mu = \mu^N$ and using $\sigma_N \geq \sqrt{\Delta_N}$, we get
    \begin{equation}\label{eqn:AW_consistency_compact_3}
    \operatorname{TV}^{(\sigma_N)}_1(\mu^N, \widehat \mu^N) \leq C \cdot (\sigma_N^{-1} \cdot \Delta_N + \sigma_N^{-2} \cdot \Delta_N^2) \leq C \cdot \sqrt{\Delta_N} \to 0, \;\; N \to +\infty,
    \end{equation}
    $\mathbb{P}$-almost surely, where $C > 0$ depends on $d, T, \operatorname{diam}(\operatorname{spt}(\mu))$.
\end{enumerate}
The proof now follows by taking $N \to +\infty$ in \eqref{eqn:AW_consistency_compact_decomposition} and using \eqref{eqn:AW_consistency_compact_1}, \eqref{eqn:AW_consistency_compact_2}, \eqref{eqn:AW_consistency_compact_3}.
\end{proof}

\subsection{Non-compact case}

Concluding the section, we present the proof of consistency for non-compact measures.

\begin{proof}[Proof of Theorem \ref{thm:consistency_general_noncompact}, (1)]
Fix $\varepsilon > 0$ and take $R \geq \sqrt{dT} \Delta_N$. Recall the mapping $\kappa_R$ from Lemma \ref{lem:compact_approximation}. By the triangle inequality for $\mathcal{AW}$, as in \eqref{eq:explain_tau},
\begin{equation}\label{eqn:AW_cons_nonc_cont_dec}
\mathcal{AW}(\mu, \widehat \mu^N) \leq \mathcal{AW}(\mu_R, \widehat{\mu_R}^N) + \tau^N_R,
\end{equation}
where $\mu_R = (\kappa_R)_\# \mu$ and $\tau^N_R$ is defined in \eqref{eq:explain_tau}. According to Lemma \ref{lem:tau_estimate},
\begin{equation}\label{eqn:AW_cons_nonc_cont_1}
\limsup_{N \to \infty} \tau^N_R \leq \frac{C}{R^{p-1}}    
\end{equation}
$\mathbb{P}$-almost surely, where $C > 0$ depends on $T, p, \int \|x\|^p \, d\mu$. Moreover, since $\kappa_R$ is continuous and injective, and $\mu$ has continuous kernels, $\mu_R$ has continuous kernels too. Furthermore, Proposition \ref{prop:eta_mixing_information_processing} ensures $\sum_{s = 1}^\infty \bar \eta_{\kappa_R(\slices X)}(s) < \infty$. Thus, according to Theorem \ref{thm:consistency_general},
\begin{equation}\label{eqn:AW_cons_nonc_cont_2}
\mathcal{AW}(\mu_R, \widehat{\mu_R}^N) \to 0, \quad N \to \infty,
\end{equation}
$\mathbb{P}$-almost surely. Taking $\limsup_{N \to \infty}$ in \eqref{eqn:AW_cons_nonc_cont_dec} and using \eqref{eqn:AW_cons_nonc_cont_1}, \eqref{eqn:AW_cons_nonc_cont_2}, we obtain
$$
\mathbb{P}\!\left(\limsup_{N \to \infty} \mathcal{AW}(\mu, \widehat \mu^N) \leq \tfrac{C}{R^{p-1}}\right) = 1,
$$
where $C > 0$ depends on $T, p$. The claim now follows by letting $R \to \infty$.
\end{proof}

\begin{proof}[Proof of Theorem \ref{thm:consistency_general_noncompact}, (2)]
Fix $\varepsilon > 0$ and take $R \geq \sqrt{dT} \Delta_N$. Recall the mapping $\kappa_R$ from Lemma \ref{lem:compact_approximation}. By the triangle inequality for $\mathcal{AW}$,
\begin{align}\label{eqn:AW_consistency_noncompact_triangle_inequality}
\begin{split}
&\mathcal{AW}(\mu, \widehat \mu^N \ast \mathcal{N}_{\sigma_N})\\
&\quad \leq \mathcal{AW}(\mu_R, \widehat {\mu_R}^N \ast \mathcal{N}_{\sigma_N}) + \mathcal{AW}(\mu, \mu_R) + \mathcal{AW}^{(\sigma_N)}(\widehat{\mu_R}^N, (\widehat \mu^N)_R) + \mathcal{AW}^{(\sigma_N)}((\widehat \mu^N)_R, \widehat \mu^N)\\
&\quad \leq \mathcal{AW}(\mu_R, \widehat {\mu_R}^N \ast \mathcal{N}_{\sigma_N}) + \tau^N_R,
\end{split}
\end{align}
where we have used $\mathcal{AW}^{(\sigma)} \leq \mathcal{AW}$ for the second inequality, and $\tau^N_R$ is defined in \eqref{eq:explain_tau}. According to Lemma \ref{lem:tau_estimate},
\begin{equation}\label{eqn:AW_cons_nonc_sm_1}
\limsup_{N \to \infty} \tau^N_R \leq \frac{C}{R^{p-1}}    
\end{equation}
$\mathbb{P}$-almost surely, where $C > 0$ depends on $T, p, \int \|x\|^p \, d\mu$. Moreover,
\begin{equation}\label{eqn:AW_cons_nonc_sm_2}
\mathcal{AW}(\mu_R, \widehat{\mu_R}^N \ast \mathcal{N}_{\sigma_N}) \to 0, \quad N \to \infty,
\end{equation}
$\mathbb{P}$-almost surely by Theorem \ref{thm:consistency_general} and Proposition \ref{prop:eta_mixing_information_processing}, which ensures $\sum_{s = 1}^\infty \bar \eta_{\kappa_R(\slices X)} < \infty$. Thus, taking $\limsup_{N \to \infty}$ in \eqref{eqn:AW_consistency_noncompact_triangle_inequality} and using \eqref{eqn:AW_cons_nonc_sm_1} and \eqref{eqn:AW_cons_nonc_sm_2}, we obtain
$$
\mathbb{P}\left(\limsup_{N \to \infty} \mathcal{AW}(\mu, \widehat \mu^N \ast \mathcal{N}_{\sigma_N}) \leq \tfrac{C}{R^{p-1}}\right) = 1,
$$
where $C > 0$ depends on $T, p$. The claim then follows by letting $R \to \infty$.
\end{proof}

\section{Auxiliary results and remaining proofs}\label{sec:remaining}

In this section we collect auxiliary results we have used throughout the paper, as well as all remaining proofs.

\begin{example}\label{ex:difference}
Let $N = 3$ and $\bold Z = (Z_n)_{n = 1}^3$ be a collection of random elements on a finite state space $\mathcal{Z} = \{0, 1\}$ with discrete $\sigma$-algebra. For $q \in (0, 1)$ define the joint law of $(Z_1, Z_2, Z_3)$ via
\begin{equation}\label{eqn:example_difference}
Z_1 \sim \operatorname{Ber}(q), \quad \operatorname{Law}(Z_{2:3} \mid Z_1) = \begin{cases}
    \delta_{0} \otimes \operatorname{Ber}(1 - q), &Z_1 = 0,\\
    \operatorname{Ber}(q) \otimes \operatorname{Ber}(q \cdot Z_2 + \tfrac{1}{2} \cdot (1 - Z_2)), &Z_1 = 1.
\end{cases}
\end{equation}
The value of $q$ will be chosen later. To estimate $\eta_{\bold Z}(1)$, we choose $n = 2$ and $A_{1:2} = \mathcal{Z} \times \{1\}$ in the definition of $\eta_{\bold Z}$ and write
\begin{equation}\label{eqn:example_difference_1}
\operatorname{Law}(Z_3 \mid Z_{1:2} \in A_{1:2}) = \operatorname{Law}(Z_3 \mid Z_2 = 1) = \operatorname{Law}(Z_3 \mid Z_1 = 1, \; Z_2 = 1) = \operatorname{Ber}(q),
\end{equation}
where we have used that $Z_2 = 1$ implies $Z_1 = 1$ according to \eqref{eqn:example_difference}. Next,
\begin{align}\label{eqn:example_difference_2}
\begin{split}
\operatorname{Law}(Z_3 \mid Z_1 \in A_1) &= \operatorname{Law}(Z_3) = \mathbb{P}(Z_1 = 1) \cdot \operatorname{Law}(Z_3 \mid Z_1 = 1) + \mathbb{P}(Z_1 = 0) \cdot \operatorname{Law}(Z_3 \mid Z_1 = 0)\\
&\overset{\eqref{eqn:example_difference}}{=} q \cdot (q \cdot \operatorname{Ber}(q) + (1 - q) \cdot \operatorname{Ber}(\tfrac{1}{2})) + (1 - q) \cdot \operatorname{Ber}(1 - q)\\
&= \operatorname{Ber}\!\left(q^3 + \tfrac{1}{2} q\cdot (1-q) + (1-q)^2\right),
\end{split}
\end{align}
where we interpret ``$+$" as a mixture of distributions. Thus, using $\operatorname{TV}(\operatorname{Ber}(p), \operatorname{Ber}(q)) = |p - q|$,
\begin{equation}\label{eqn:example_difference_bar_eta_estimate}
\eta_{\bold Z}(1) \geq \operatorname{TV}\!\left(\operatorname{Law}(Z_3 \mid Z_{1:2} \in A_{1:2}), \operatorname{Law}(Z_3 \mid Z_{1} \in A_{1})\right) \overset{\substack{\eqref{eqn:example_difference_1}\\\eqref{eqn:example_difference_2}}}{\geq} |q^3 + \tfrac{1}{2} q\cdot (1-q) + (1-q)^2 - q|.
\end{equation}
To compute $\widehat \eta_{\bold Z}(2, 2+1)$, observe that only $Z_1 = 1$ has both $Z_2 = 0$ and $Z_2 = 1$ with positive probability. Hence,
\begin{equation}\label{eqn:example_difference_eta_221}
\widehat \eta_{\bold Z}(2, 2+1) = \operatorname{TV}\!\left(\operatorname{Law}(Z_3 \mid Z_1 = 1, \; Z_2 = 0), \operatorname{Law}(Z_3 \mid Z_1 = 1, \; Z_2 = 1)\right) \overset{\eqref{eqn:example_difference}}{=} |\tfrac{1}{2} - q|.
\end{equation}
For $\widehat \eta_{\bold Z}(1, 1+1)$, we recall conditional laws $\operatorname{Law}(Z_{2:3} \mid Z_1)$ from \eqref{eqn:example_difference} and write:
\begin{equation}\label{eqn:example_difference_eta_111}
\begin{split}
&\operatorname{TV}\!\left(\operatorname{Law}(Z_{2:3} \mid Z_1 = 0), \operatorname{Law}(Z_{2:3} \mid Z_1 = 1)\right)\\
&\overset{\eqref{eqn:example_difference}}{=} \operatorname{TV}\!\left(\delta_0 \otimes \operatorname{Ber}(1 - q), \operatorname{Ber}(q) \otimes \operatorname{Ber}(q \cdot Z_2 + \tfrac{1}{2} \cdot (1 - Z_2))\right)\\
&= \operatorname{TV}\!\left(\delta_0 \otimes \operatorname{Ber}(1 - q), (1 - q) \cdot \delta_{0} \otimes \operatorname{Ber}(\tfrac{1}{2}) + q \cdot \delta_1 \otimes \operatorname{Ber}(q) \right)\\
&= \frac{1}{2} \cdot \left(|q - \tfrac{1}{2} (1 - q)| + |1 - q - \tfrac{1}{2} (1 - q)| + q\right) = \tfrac{1}{4} |3q - 1| + \tfrac{1}{4} (1 + q).
\end{split}
\end{equation}
Thus, taking $q = \frac{1}{10}$,
\begin{align*}
&\eta_{\bold Z}(1) \overset{\eqref{eqn:example_difference_bar_eta_estimate}}{\geq} |q^3 + \tfrac{1}{2} q\cdot (1-q) + (1-q)^2 - q| = \tfrac{756}{1000}, \\
&\sup_{n \in \{1, 2\}} \widehat \eta_{\bold Z}(n, n+1) \overset{\substack{\eqref{eqn:example_difference_eta_221}\\\eqref{eqn:example_difference_eta_111}}}{=} \max( |\tfrac{1}{2} - q|, \tfrac{1}{4} |3q - 1| + \tfrac{1}{4} (1 + q)) = \tfrac{9}{20},
\end{align*}
hence $\eta_{\bold Z}(1) > \sup_{n \in \{1, 2\}} \widehat \eta_{\bold Z}(n, n+1)$.
\end{example}

\begin{proposition}\label{prop:eta_mixing_variance}
Fix $N \in \mathbb{N}$ let $\bold Z = (Z_n)_{n = 1}^N$ be a collection of identically distributed random elements taking values in $\R^k$. Let $f: \R^k \to [0, +\infty)$ be a bounded Borel measurable function. Then
\begin{equation}\label{eqn:eta_mixing_variance}
\operatorname{Var}\left[\frac{1}{N} \sum_{n = 1}^N f(Z_n)\right] \leq \frac{1}{N} \cdot \|f\|_\infty \cdot \mathbb{E} \left[f(Z_1)\right] \cdot \left(1 + 2 \sum_{s = 1}^{N-1} \eta_{\bold Z}(s)\right).
\end{equation}
\end{proposition}
\begin{proof}
According to Proposition \ref{prop:eta_mixing_covariance}, for any $i \neq j$ we have
\begin{equation}\label{eqn:covariance_bound_Z}
\operatorname{Cov}\left(f(Z_i), f(Z_j)\right) \leq \eta_{\bold Z}(|i - j|) \cdot \|f\|_\infty \cdot \mathbb{E}\left[f(Z_i)\right].
\end{equation}
Thus,
\begin{align*}
\operatorname{Var}\left[\frac{1}{N} \sum_{n = 1}^N f(Z_n)\right] &= \frac{1}{N} \cdot \operatorname{Var}\left[f(Z_1)\right] + \frac{1}{N^2} \cdot \sum_{i \neq j} \operatorname{Cov}\left(f(Z_i), f(Z_j)\right)\\
&\overset{\eqref{eqn:covariance_bound_Z}}{\leq} \frac{1}{N} \cdot \operatorname{Var}\left[f(Z_1)\right] + \frac{1}{N^2} \cdot \|f\|_\infty \cdot \mathbb{E}\left[f(Z_1)\right] \cdot \sum_{i \neq j} \eta_{\bold Z}(|i - j|)\\
&\leq \frac{1}{N} \cdot \left[\operatorname{Var}\left[f(Z_1)\right] + 2 \sum_{s = 1}^{N-1} \eta_{\bold Z}(s) \cdot \|f\|_\infty \cdot \mathbb{E} \left[f(Z_1)\right]\right]\\
&\leq \frac{1}{N} \cdot \|f\|_\infty \cdot \mathbb{E} \left[f(Z_1)\right] \cdot \left(1 + 2 \sum_{s = 1}^{N-1} \eta_{\bold Z}(s)\right).
\end{align*}
where the second inequality follows from $\sum_{i \neq j} \eta_{\bold Z}(|i - j|) \leq 2N \cdot \sum_{s = 1}^{N-1} \eta_{\bold Z}(s)$, and the final inequality holds as $\operatorname{Var}\left[f(Z_1)\right] \leq \mathbb{E}\left[f(Z_1)^2\right] \leq \|f\|_\infty \cdot \mathbb{E}\left[f(Z_1)\right]$.
\end{proof}

\begin{proposition}\label{prop:borel_sets_approximation}
Let $\mu, \nu \in \mathcal{P}(\R^k)$, and denote by $\mathcal{A} \subseteq \mathcal{B}(\R^k)$ the collection of finite unions of open rectangles with rational endpoints. Then for any $S \in \mathcal{B}(\R^k)$ and any $\varepsilon > 0$ there exists $U \in \mathcal{A}$, such that
\begin{equation}\label{eqn:borel_set_approximation}
\mu(U \Delta S) \vee \nu(U \Delta S) < \varepsilon,
\end{equation}
where $A \Delta B := A \setminus B \cup B \setminus A$ for $A, B \in \mathcal{B}(\R^k)$ is the symmetric difference of $A$ and $B$.
\end{proposition}
\begin{proof}
Fix $\varepsilon > 0$. Since $\mu$ and $\nu$ are Borel probability measures, they are outer regular. Hence, there exists an open set $O \supset S$, such that
\begin{equation}\label{eqn:borel_set_approximation_open}
\mu(O \setminus S) \vee \nu(O \setminus S) = \mu(O \Delta S) \vee \nu(O \Delta S) < \frac{\varepsilon}{2}.
\end{equation}

Next, observe that $O = \cup_{n = 1}^\infty A_n$ for some sequence $\{A_n\}_{n = 1}^\infty \subseteq \mathcal{A}$, which satisfies $A_n \subseteq A_{n+1}$ for $n \in \mathbb{N}$, because $O$ is a countable union of open rectangles with rational endpoints. By monotone convergence, $\mu(A_n) \uparrow \mu(O)$ and $\nu(A_n) \uparrow \nu(O)$, as $n \to \infty$. In particular, choose $N = N(\varepsilon)$, such that
\begin{equation}\label{eqn:borel_set_approximation_union}
\mu(O \setminus A_N) \vee \nu(O \setminus A_N) = \mu(A_N \Delta O) \vee \nu(A_N \Delta O) < \frac{\varepsilon}{2}.
\end{equation}
Then by $A_N\Delta S\subset (A_N\Delta O)\cup (O\Delta S)$ and $(a + b) \vee (c + d) \leq (a \vee c) + (b \vee d)$ we have
\begin{align*}
\mu(A_N \Delta S) \vee \nu(A_N \Delta S) &\leq
[\mu(A_N \Delta O) +\mu(O\Delta S)] \vee [\nu(A_N \Delta O)+\nu(O \Delta S)]\\
&\le \mu(A_N \Delta O) \vee \nu(A_N \Delta O) + \mu(O \Delta S) \vee \nu(O \Delta S) \overset{\substack{\eqref{eqn:borel_set_approximation_open}\\ \eqref{eqn:borel_set_approximation_union}}}{<} \varepsilon.
\end{align*}
Therefore, \eqref{eqn:borel_set_approximation} holds with $U := A_N$.
\end{proof}

\begin{proposition}\label{prop:TV_countable_approx}
Let $\mu, \nu \in \mathcal{P}(\R^k)$. Then
\begin{equation}
\operatorname{TV}(\mu, \nu) = \sup_{S \in \mathcal{A}} |\mu(S) - \nu(S)|,
\end{equation}
where $\mathcal{A}$ consists of finite unions of open rectangles with rational endpoints.
\end{proposition}
\begin{proof}
The ``$\geq$" part is trivial, as $\mathcal{A} \subseteq \mathcal{B}(\R^k)$. The ``$\leq$" part follows from Proposition \ref{prop:borel_sets_approximation}.
\end{proof}

\begin{proposition}\label{prop:TV_bound_mixing}
Fix $N \in \mathbb{N}$ and let $\bold Z = (Z_n)_{n = 1}^N$ be a sequence taking values in $\R^k$. Let $1 \leq s < N$ and $1 \leq n \le N -s $. Then $\P(Z_{1:n}\in dz_{1:n})$-almost everywhere
\begin{equation}\label{eqn:TV_bound_mixing_1}
\operatorname{TV}(\operatorname{Law}(Z_{n+s} \mid Z_{1:n} = z_{1:n}), \operatorname{Law}(Z_{n+s} \mid Z_{1:n-1} = z_{1:n-1})) \leq \eta_{\bold Z}(s).
\end{equation}
Similarly,
\begin{equation}\label{eqn:TV_bound_mixing_2}
\operatorname{TV}(\operatorname{Law}(Z_{n+s:N} \mid Z_{1:n} = z_{1:n}), \operatorname{Law}(Z_{n+s:N} \mid Z_{1:n-1} = z_{1:n-1})) \leq \bar \eta_{\bold Z}(s).
\end{equation}
\end{proposition}
\begin{proof}
We only prove \eqref{eqn:TV_bound_mixing_1} as \eqref{eqn:TV_bound_mixing_2} is obtained using precisely the same arguments.
\newcommand{\Bprodn}[2][]{B^{\otimes}_{#1}(z_{1:#2})}

Fix $A \in \mathcal{B}(\R^k)$ and denote $\prod_{s = 1}^n B_\delta(z_s)$ by $\Bprodn[\delta]{n}$. We first argue that 
\begin{equation}\label{eqn:lebesgue_diff}
\mathbb{P}\left(Z_{n+s} \in A \mid Z_{1:n} \in \Bprodn[\delta]{n}\right) \to \mathbb{P}\left(Z_{n+s} \in A \mid Z_{1:n} = z_{1:n}\right), \;\; \delta \to 0,
\end{equation}
$\P(Z_{1:n}\in dz_{1:n})$-almost everywhere. Indeed, this follows from the Lebesgue differentiation theorem in the form of \cite[Theorem 5.8.8]{bogachev2007measure}, applied with $\nu(dz_{1:n}) = \mathbb{P}(Z_{n+s} \in A, Z_{1:n} \in dz_{1:n})$ and $\mu(dz_{1:n}) = \mathbb{P}(Z_{1:n} \in dz_{1:n})$. We point out that our statement differs from the cited result, as $B_\delta(z_{1:n})$ is replaced with $\Bprodn[\delta]{n}$; such a replacement is possible as $B_\delta(z_{1:n}) \subset \Bprodn[\delta]{n} \subset B_{\sqrt{n} \delta}(z_{1:n})$. Similarly,
\begin{equation}\label{eqn:lebesgue_diff_short}
\mathbb{P}\left(Z_{n+s} \in A \mid Z_{1:n-1} \in \Bprodn[\delta]{n-1}\right) \to \mathbb{P}\left(Z_{n+s} \in A \mid Z_{1:n-1} = z_{1:n-1}\right), \;\; \delta \to 0,
\end{equation}
$\P(Z_{1:n}\in dz_{1:n})$-almost everywhere.

Moreover, $\mathbb{P}(Z_{1:n} \in \Bprodn[\delta]{n}) > 0$ for any $\delta > 0$ $\P(Z_{1:n}\in dz_{1:n})$-almost everywhere, which implies
\begin{equation}\label{eqn:TV_bound_central}
|\mathbb{P}(Z_{n+s} \in A \mid Z_{1:n} \in \Bprodn[\delta]{n}) - \mathbb{P}(Z_{n+s} \in A \mid Z_{1:n-1} \in \Bprodn[\delta]{n-1})| \leq \eta_{\bold Z}(s)
\end{equation}
by definition of $\eta_{\bold Z}.$  Next, using triangle inequality, we get for any $\delta > 0$ 
\begin{align}
\begin{split}
&|\mathbb{P}(Z_{n+s} \in A \mid Z_{1:n} = z_{1:n}) - \mathbb{P}(Z_{n+s} \in A \mid Z_{1:n-1} = z_{1:n-1})|\\
&\leq |\mathbb{P}(Z_{n+s} \in A \mid Z_{1:n} = z_{1:n}) - \mathbb{P}(Z_{n+s} \in A \mid Z_{1:n} \in \Bprodn[\delta]{n})|\\
&\quad+ |\mathbb{P}(Z_{n+s} \in A \mid Z_{1:n} \in \Bprodn[\delta]{n}) - \mathbb{P}(Z_{n+s} \in A \mid Z_{1:n-1} \in \Bprodn[\delta]{n-1})|\\
&\quad+ |\mathbb{P}(Z_{n+s} \in A \mid Z_{1:n-1} \in \Bprodn[\delta]{n-1}) - \mathbb{P}(Z_{n+s} \in A \mid Z_{1:n-1} = z_{1:n-1})|,
\end{split}
\end{align}
hence after letting $\delta \to 0$ and applying \eqref{eqn:lebesgue_diff}, \eqref{eqn:lebesgue_diff_short} and \eqref{eqn:TV_bound_central} we obtain
\begin{equation}\label{eqn:TV_bound_fixed_set}
|\mathbb{P}(Z_{n+s} \in A \mid Z_{1:n} = z_{1:n}) - \mathbb{P}(Z_{n+s} \in A \mid Z_{1:n-1} = z_{1:n-1})| \leq \eta_{\bold Z}(s)
\end{equation}
$\P(Z_{1:n}\in dz_{1:n})$-almost everywhere. We now use Proposition \ref{prop:TV_countable_approx} and take the supremum over $A \in \mathcal{A}$, where $\mathcal{A}$ is the collection of finite unions of open rectangles with rational endpoints. As a result (note that $\mathcal{A}$ is countable, and almost-everywhere inequalities are stable under suprema over countable sets), we obtain \eqref{eqn:TV_bound_mixing_1}, as desired.
\end{proof}

\begin{proposition}\label{prop:TV_osc}
Let $f: \R^k \to \R$ satisfy $\operatorname{osc}(f) := \sup_x f(x) - \inf_x f(x) \leq 1$. Then for any $\mu, \nu \in \mathcal{P}(\R^k)$
\begin{equation}\label{eqn:TV_osc}
\left|\int f(x) \, \mu(dx) - \int f(x) \, \nu(dx)\right| \leq \operatorname{TV}(\mu, \nu).
\end{equation}
\end{proposition}
\begin{proof}
Let $m := \tfrac{1}{2} (\inf_x f(x) + \sup_x f(x))$ and set $g(x) := f(x) - m$. Then $\|g\|_\infty \leq \tfrac{1}{2}$, as $\operatorname{osc}(f) \leq 1$. By the variational characterization of total variation,
\begin{equation*}
\left|\int g(x) \, \mu(dx) - \int g(x) \, \nu(dx)\right| \leq 2 \|g\|_\infty \, \operatorname{TV}(\mu, \nu) \leq \operatorname{TV}(\mu, \nu).
\end{equation*}
Since $\int f(x)\,\mu(dx) - \int f(x)\,\nu(dx) = \int g(x)\,\mu(dx) - \int g(x)\,\nu(dx)$, \eqref{eqn:TV_osc} follows.
\end{proof}

\begin{lemma}[``Subadditivity" of $\operatorname{TV}$]\label{lem:TV_subadditivity}
For $M \in \N$ let $\phi: (\mathbb{R}^k)^M \to \mathbb{R}$ have $1$-bounded differences, i.e.~is a 1-Lipschitz function with respect to the Hamming distance. Let $\mu, \nu \in \mathcal{P}((\R^k)^M)$. Then
\begin{equation}\label{eqn:TV_subadditivity}
\int \phi(x_{1:M}) \, \mu(dx_{1:M}) - \int \phi(x_{1:M}) \, \nu(dx_{1:M}) \leq \sum_{m = 1}^M \operatorname{TV}(\mu(dx_{m:M}), \nu(dx_{m:M})).
\end{equation}
\end{lemma}
\begin{proof}
We argue by induction on $M$. For $M = 1$ the claim follows from
Proposition \ref{prop:TV_osc} (since $1$-bounded differences implies $\operatorname{osc}(\phi) \leq 1$).

Assume that the statement holds for $M - 1$. Define
$$
\psi(x_{2:M}) := \inf_{x_1} \phi(x_{1:M}), \qquad g(x_{1:M}) := \phi(x_{1:M}) - \psi(x_{2:M}).
$$
For each fixed $x_{2:M} \in (\R^k)^{M-1}$, changing only $x_1$ changes $\phi$ by at most 1, hence
\begin{equation}\label{eqn:TV_subadditivity_osc}
0 \leq g(x_{1:M}) \leq 1,
\end{equation}
and $\psi(x_{2:M})$ has 1-bounded differences. Thus, applying Proposition \ref{prop:TV_osc} on $(\R^k)^M$ with $f = g$ and using induction hypothesis for $\psi$, we obtain
\begin{align*}
\int \phi(x_{1:M}) \, \mu(dx_{1:M}) - \int \phi(x_{1:M}) \, \nu(dx_{1:M}) &= \int g(x_{1:M}) \, \mu(dx_{1:M}) - \int g(x_{1:M}) \, \nu(dx_{1:M})\\
&\quad + \int \psi(x_{2:M}) \, \mu(dx_{2:M}) - \int \psi(x_{2:M}) \, \nu(dx_{2:M})\\
&\leq \operatorname{TV}(\mu, \nu) + \sum_{m = 2}^M \operatorname{TV}(\mu(dx_{m:M}), \nu(dx_{m:M}))\\
&= \sum_{m = 1}^M \operatorname{TV}(\mu(dx_{m:M}), \nu(dx_{m:M})),
\end{align*}
as claimed.
\end{proof}

\begin{proposition}[Connection with $\phi$-mixing]\label{prop:phi_mixing}
Fix $N \in \mathbb{N}$ and let $\bold Z = (Z_n)_{n = 1}^N$ be a collection of random elements taking values in $\R^k$. Define 
\begin{equation}\label{eqn:phi_mixing}
\phi_{\bold Z}(s) := \sup_{\substack{1 \leq n < N - s 
+ 1\\A \in \mathcal{B}((\R^k)^n)\\ \mathbb{P}(Z_{1:n} \in A) > 0}} \operatorname{TV}\left(\operatorname{Law}(Z_{n+s} \mid Z_{1:n} \in A), \operatorname{Law}(Z_{n+s})\right)
\end{equation}
for $1 \leq s < N$. Then $\phi_{\bold Z}$ is the $\phi$-mixing coefficient as defined in \eqref{eq:examples} and
$$
\phi_{\bold Z}(s) \leq \sum_{k = s}^{N-1} \eta_{\bold Z}(k).
$$
\end{proposition}
\begin{proof}
Comparing \eqref{eq:examples} with \eqref{eqn:phi_mixing} shows $\phi_{\bold Z}(s)=\phi(s).$
Next, fix $1 \leq s < N$ and $1 \leq n < N - s + 1$. By Proposition \ref{prop:TV_bound_mixing} and the triangle inequality for $\operatorname{TV}$,
\begin{align}\label{eqn:eta_phi_connection}
\begin{split}
&\operatorname{TV}\left(\operatorname{Law}(Z_{n+s} \mid Z_{1:n} = z_{1:n}), \operatorname{Law}(Z_{n+s})\right)\\
&\leq \sum_{k = 1}^{n} \operatorname{TV}\left(\operatorname{Law}(Z_{n+s} \mid Z_{1:k} = z_{1:k}), \operatorname{Law}(Z_{n+s} \mid Z_{1:k-1} = z_{1:k-1})\right)\\
&\leq \sum_{k = 1}^n \eta_{\bold Z}(s + k - 1),
\end{split}
\end{align}
$\mathbb{P}(Z_{1:n} \in dz_{1:n})$-almost surely. Take $A \in \mathcal{B}((\R^k)^n)$, which satisfies $\mathbb{P}(Z_{1:n} \in A) > 0$. Integrating \eqref{eqn:eta_phi_connection} over $\mathbb{P}(Z_{1:n} \in dz_{1:n} \mid Z_{1:n} \in A)$ and using convexity of $\operatorname{TV}$, we obtain
$$
\operatorname{TV}(\operatorname{Law}(Z_{n+s} \mid Z_{1:n} \in A), \operatorname{Law}(Z_{n+s})) \leq \sum_{k = 1}^n \eta_{\bold Z}(s + k - 1),
$$
hence \eqref{eqn:phi_mixing} follows after taking supremum over $A \in \mathcal{B}((\R^k)^n)$ with positive $\mathbb{P}$-probability.
\end{proof}

\begin{proposition}\label{prop:TV_1_gaussians_bound}
For any $\sigma \in (0, 1]$ and $x_1, x_2 \in \mathbb{R}^d$ we have
$$
\operatorname{TV}^{(\sigma)}_1(\delta_{x_1}, \delta_{x_2}) \leq C \cdot (1 + \|x_2\|) \cdot (\sigma^{-1} \cdot \|x_1 - x_2\| + \sigma^{-2} \cdot \|x_1 - x_2\|^2),
$$
where $C > 0$ depends on $d$.
\end{proposition}
\begin{proof}
According to \cite[Example 6.2.3]{murphy2022probabilistic} we have 
\begin{equation}\label{eqn:TV_1_KL_bound}
D_{\operatorname{KL}}(\mathcal{N}(x_1, \sigma^2 I_d) || \mathcal{N}(x_2, \sigma^2 I_d)) \leq C \cdot \sigma^{-2} \cdot \|x_1 - x_2\|^2,
\end{equation}
where $C > 0$ is an absolute constant and $D_{\operatorname{KL}}$ is the \textit{Kullback-Leibler divergence} defined as
$$
D_{\operatorname{KL}}(\mu, \nu) = \begin{cases}
\int \log \frac{d\mu}{d\nu}(x) \, \mu(dx), &\text{if} \;\; \mu \ll \nu,\\
\infty, &\text{otherwise}
\end{cases}
$$
for $\mu, \nu \in \mathcal{P}(\R^d)$. Applying \cite[Theorem 2.1]{bolley2005weighted} with the weight function $\phi(x) = \|x\| + \frac{1}{2}$ yields
\begin{align}
\operatorname{TV}^{(\sigma)}_1(\delta_{x_1}, \delta_{x_2})
&= \int \left(\|x\| + \tfrac{1}{2}\right) \cdot |\varphi_\sigma(x - x_1) - \varphi_\sigma(x - x_2)| \, dx\nonumber\\
&\leq \left(\tfrac{5}{2} + \log \int e^{2 \|x\|} \, \varphi_\sigma(x - x_2) \, dx\right)\nonumber\\
&\quad \cdot \left(\sqrt{D_{\operatorname{KL}}(\mathcal{N}(x_1, \sigma^2 I_d) || \mathcal{N}(x_2, \sigma^2 I_d))} + D_{\operatorname{KL}}(\mathcal{N}(x_1, \sigma^2 I_d) || \mathcal{N}(x_2, \sigma^2 I_d))\right)\nonumber\\
&\overset{\eqref{eqn:TV_1_KL_bound}}{\leq} C \cdot \left(\tfrac{5}{2} + \log \int e^{2 \|x\|} \, \varphi_\sigma(x - x_2) \, dx\right) \cdot (\sigma^{-1} \cdot \|x_1 - x_2\| + \sigma^{-2} \cdot \|x_1 - x_2\|^2),\label{eqn:TV_1_gaussians_bound}
\end{align}
where $C > 0$ is an absolute constant. To complete the proof, it suffices to bound the second term in \eqref{eqn:TV_1_gaussians_bound}. Using the triangle inequality, monotonicity of $x \mapsto e^{x}$, and $\sigma \leq 1$ we obtain
\begin{align*}
\log \int e^{2 \|x\|} \varphi_\sigma(x - x_2) \, dx
&\leq 2 \sigma \|x_2\| + \log \int e^{2 \sigma \|x\|} \varphi_1(x) \, dx\\
&\leq 2 \|x_2\| + \log \int e^{2 \|x\|} \varphi_1(x) \, dx\\
&\leq C \cdot (1 + \|x_2\|),
\end{align*}
where $C > 0$ depends on $d$. 
\end{proof}

The following result also appears in the proof of \cite[Lemma 7.1, Step 1]{acciaio2024convergence} in a slightly different form. We modify the approximation argument, such that a projection to a compact subset is injective, and hence preserves continuity of kernels of the original distribution. 

\begin{lemma}[Compact approximation]\label{lem:compact_approximation}
Let $R \geq \sqrt{d}\Delta_N$. Then there exists the mapping $\kappa_R: \R^d \to B_{2R}(0) \subset \R^d$, which is continuous and injective and satisfies
\begin{enumerate}
    \item For any $\mu \in \mathcal{P}((\R^d)^T)$
    $$
    \mathcal{AW}(\mu, \mu_R) \leq \sqrt{T} \int_{\{\|x\| \ge R\}} \|x\| \, \mu(dx),
    $$
    where $\mu_R := (\kappa_R)_\# \mu$.
    \item On $B_{\frac{R}{2}}(0) \subset \R^d$ we have $\kappa_R \circ \varphi^N = \varphi^N \circ \kappa_R$.
\end{enumerate}
\end{lemma}
\begin{proof}
Fix $R > 0$ and define $\kappa_R: \R^d \to B_{2R}(0)$ as follows:
$$
\kappa_R(x) := \begin{cases}
    0, &x = 0,\\
    \frac{g_R(\|x\|)}{\|x\|} \cdot x, &x \neq 0,
\end{cases}\quad\text{where}\quad g_R(r) := \begin{cases}
    r, &r \leq R,\\
    2R - R \exp\!\left(1 - \frac{r}{R}\right), &r > R.
\end{cases}
$$
Note that $\|\kappa_R(x)\| = g_R(\|x\|) \uparrow 2R,$ as $\|x\| \to \infty$, and
\begin{align}
\kappa_R(x) = x \quad &\text{for } x \in B_{R}(0) \subset \R^d,\label{eqn:compact_approximation_properties_1}\\
|g_R(r)| \leq r \quad &\text{for all } r \geq 0.\label{eqn:compact_approximation_properties_2}
\end{align}
Let $\mu \in \mathcal{P}((\R^d)^T)$. Since $\kappa_R$ is injective and acts coordinatewise, the coupling
$\pi_R := (\operatorname{Id}, \kappa_R)_\#\mu$ is bicausal, i.e.\ $\pi \in \Pi_{\operatorname{bc}}(\mu, \mu_R)$. Thus,
\begin{align*}
\mathcal{AW}(\mu, \mu_R) &\leq \int \sum_{t = 1}^T \|x_{t} - y_{t}\| \, \pi_R(dx, dy) =\int \sum_{t = 1}^T \|x_t - \kappa_R(x_t)\| \, \mu(dx)\\
&\overset{\eqref{eqn:compact_approximation_properties_1}}{\leq} \int_{\{\|x\| \ge R\}} \sum_{t = 1}^T \|x_t - \kappa_R(x_t)\| \,\mu(dx),
\end{align*}
where the second inequality also uses that
$\{\|x\| < R\}$ implies $\|x_t\| < R$ for all $t$, hence $\kappa_R(x_t) = x_t$. Moreover, by \eqref{eqn:compact_approximation_properties_2},
$\|x_t - \kappa_R(x_t)\| = \|x_t\| - \|\kappa_R(x_t)\| \le \|x_t\|$, so by Cauchy--Schwarz,
\[
\sum_{t=1}^T \|x_t - \kappa_R(x_t)\|
\le \sum_{t=1}^T \|x_t\|
\le \sqrt{T}\,\Big(\sum_{t=1}^T \|x_t\|^2\Big)^{1/2}
=\sqrt{T}\,\|x\|.
\]
This proves \emph{(1)}.

For \emph{(2)}, let $x \in B_{\frac{R}{2}}(0)$. Since $\kappa_R(x) = x$ by \eqref{eqn:compact_approximation_properties_1}, it suffices to show that $\varphi^N(x) \in B_R(0)$. Using the quantization error bound $\|\varphi^N(x) - x\| \le \tfrac{1}{2}\sqrt{d}\, \Delta_N$ and $R \ge \sqrt d\,\Delta_N$,
\begin{equation}\label{eqn:compact_approximation_properties_3}
\|\varphi^N(x)\| \leq \|x\| + \|\varphi^N(x) - x\| \leq \tfrac{R}{2} + \tfrac{1}{2}\sqrt{d} \Delta_N \leq R.
\end{equation}
Thus,
$$
\varphi^N(\kappa_R(x)) \overset{\eqref{eqn:compact_approximation_properties_1}}{=} \varphi^N(x) \overset{\substack{\eqref{eqn:compact_approximation_properties_1}\\\eqref{eqn:compact_approximation_properties_3}}}{=} \kappa_R(\varphi^N(x)),
$$
as claimed. The proof is complete.
\end{proof}

\begin{lemma}[Price-to-pay if $\kappa_R \circ \varphi^N \neq \varphi^N \circ \kappa_R$]\label{lem:AW_mismatch_error}
Let $R \geq \sqrt{d}\Delta_N$, and let $\kappa_R$ be as in Lemma \ref{lem:compact_approximation}. Then for any $x:=(x^1, \dots, x^N)$ with $x^n \in (\R^d)^T$,
\begin{equation}
\mathcal{AW}\!\left((\widehat \mu^N)_R(x), \widehat{\mu_R}^N(x)\right) \leq C\,\frac{R}{N} \sum_{n = 1}^N \mathds{1}_{\{\|x^n\|\ge R/2\}},
\end{equation}
where $C > 0$ depends on $T$ and
\begin{equation*}
(\widehat \mu^N)_R(x) := \frac{1}{N} \sum_{n = 1}^N \delta_{\kappa_R(\varphi^N(x^n))}, \qquad \widehat{\mu_R}^N(x) := \frac{1}{N} \sum_{n = 1}^N \delta_{\varphi^N(\kappa_R(x^n))}.
\end{equation*}
\end{lemma}
\begin{proof}
We first bound the transport cost uniformly in order to control $\mathcal{AW}$ by total variation. Since $\kappa_R(\R^d) \subseteq B_{2R}(0)$, we have $\|\kappa_R(\varphi^N(z))\| \leq 2R$ for all $z \in \R^d$. Moreover, using the quantization error bound
$\|\varphi^N(u) - u\| \le \tfrac{1}{2}\sqrt{d}\,\Delta_N \le \tfrac{R}{2}$,
\begin{equation}\label{eqn:AW_mismatch_minor}
\|\varphi^N(\kappa_R(z))\| \leq \|\kappa_R(z)\| + \|\varphi^N(\kappa_R(z)) - \kappa_R(z)\| \leq \tfrac{5}{2}R
\end{equation}
for all $z \in \R^d$. Thus, for any pair of points $x, y \in (\R^d)^T$,
\begin{equation}\label{eqn:AW_mismatch_cost_bound}
\sum_{t = 1}^T \|\kappa_R(\varphi^N(x_t)) - \varphi^N(\kappa_R(y_t))\| \leq T \cdot \Big(\sup_{z \in \R^d} \|\kappa_R(\varphi^N(z))\| + \sup_{z \in \R^d} \|\varphi^N(\kappa_R(z))\| \Big) \overset{\eqref{eqn:AW_mismatch_minor}}{\leq} 5RT.
\end{equation}
Hence, according to \cite[Corollary 2.7]{acciaio2025estimating},
\begin{equation}\label{eqn:compact_AW_leq_TV}
\mathcal{AW}\!\left((\widehat \mu^N)_R(x), \widehat{\mu_R}^N(x)\right) \overset{\eqref{eqn:AW_mismatch_cost_bound}}{\leq} 5RT\,(2T-1)\,\operatorname{TV}\!\left((\widehat \mu^N)_R(x), \widehat{\mu_R}^N(x)\right).
\end{equation}
Next, $\kappa_R$ commutes with $\varphi^N$ on $B_{\frac{R}{2}}(0) \subset \R^d$ by Lemma \ref{lem:compact_approximation}; hence
\begin{equation}\label{eqn:AW_mismatch_commute}
\kappa_R(\varphi^N(x)) \neq \varphi^N(\kappa_R(x)) \implies \exists t:\ \|x_t\| \geq \tfrac{R}{2} \implies \|x\| \geq \tfrac{R}{2}.
\end{equation}
Finally, by convexity of $\operatorname{TV}$,
\begin{equation}\label{eqn:TV_bound_kappa}
\operatorname{TV}\!\left((\widehat \mu^N)_R(x), \widehat{\mu_R}^N(x)\right) \leq \frac{1}{N} \sum_{n = 1}^N \mathds{1}_{\{\kappa_R(\varphi^N(x^n)) \neq \varphi^N(\kappa_R(x^n))\}} \overset{\eqref{eqn:AW_mismatch_commute}}{\leq} \frac{1}{N} \sum_{n = 1}^N \mathds{1}_{\{ \|x^n\|\ge R/2\}}.
\end{equation}
The proof now follows from \eqref{eqn:compact_AW_leq_TV} and \eqref{eqn:TV_bound_kappa}.
\end{proof}

\begin{lemma}[Error term estimate for compact approximation]\label{lem:tau_estimate}
Let $R \geq \sqrt{dT}\,\Delta_N$, and let $\kappa_R$ be as in Lemma \ref{lem:compact_approximation}. For $N \in \N \cup \{\infty\}$ let $\slices X = (X^n)_{n = 1}^N \sim \mu \in \mathcal{P}_p((\R^d)^T)$ for some $p > 1$, and define
\begin{equation}\label{eqn:def_tau}
\tau^N_R := \mathcal{AW}(\mu, \mu_R) + \mathcal{AW}(\widehat{\mu_R}^N, (\widehat \mu^N)_R) + \mathcal{AW}((\widehat \mu^N)_R, \widehat \mu^N).
\end{equation}
Then the following holds:
\begin{enumerate}
    \item The expected value has the bound
    $$
    \mathbb{E}\,\tau^N_R \leq \frac{C}{R^{p-1}},
    $$
    where $C > 0$ depends on $T, p, \int \|x\|^p \, d\mu$.
    \item If $\mathcal{E}_{\alpha, \gamma}(\mu) < \infty$, then
    $$
    \mathbb{P}\!\left(\tau^N_R \geq \mathbb{E}\, \tau^N_R\right) \leq \mathcal{E}_{\alpha, \gamma} \cdot N \exp\!\left(-c \cdot R^\alpha\right),
    $$
    where $c > 0$ depends on $\alpha, \gamma$.
    \item If $\sum_{s = 1}^\infty \sqrt{\sum_{k = 2^s}^\infty \bar \eta_{\slices X}(k)} < \infty$, then
    $$
    \limsup_{N \to \infty} \tau^N_R \leq \frac{C}{R^{p-1}},
    $$
    where $C > 0$ depends on $T, p, \int \|x\|^p \, d\mu$.
\end{enumerate}
\end{lemma}
\begin{proof}
We start by estimating two of the terms in \eqref{eqn:def_tau}. According to Lemma \ref{lem:compact_approximation},
\begin{align}\label{eqn:tau_est_3}
\begin{split}
&\mathcal{AW}(\mu, \mu_R) \le \sqrt{T} \int_{\{\|x\| \ge R\}} \|x\| \, \mu(dx)\\
&\mathcal{AW}((\widehat \mu^N)_R, \widehat \mu^N) \le \frac{\sqrt{T}}{N} \sum_{n = 1}^N \|\varphi^N(X^n)\| \cdot \mathds{1}_{\{ \|\varphi^N(X^n)\|\ge R\}}.
\end{split}
\end{align}
Using $R \geq \sqrt{dT}\,\Delta_N$, we write the quantization error bound as $\|\varphi^N(x) - x\| \leq \tfrac{1}{2}\,\sqrt{dT}\,\Delta_N \leq \tfrac12 \, R$ for all $x \in (\R^d)^T$. Thus,
\begin{equation}\label{eqn:tau_est_varphi_simp}
\|\varphi^N(x)\| \ge R \implies \|x\| \ge \tfrac{1}{2} \, R \implies \|\varphi^N(x)\| \leq \|x\| + \tfrac{1}{2}\,R \leq 2 \|x\|,
\end{equation}
and consequently
\begin{equation}\label{eqn:tau_est_term}
\mathcal{AW}((\widehat \mu^N)_R, \widehat \mu^N) \overset{\substack{\eqref{eqn:tau_est_3}\\\eqref{eqn:tau_est_varphi_simp}}}{\leq} \frac{2\sqrt{T}}{N} \sum_{n = 1}^N \|X^n\| \cdot \mathds{1}_{\{\|X^n\|\ge R/2\}}.
\end{equation}

Next, controlling the middle term in \eqref{eqn:def_tau} by Lemma \ref{lem:AW_mismatch_error} and using \eqref{eqn:tau_est_3} and \eqref{eqn:tau_est_term}, we obtain
\begin{align}\label{eqn:tau_dec}
\begin{split}
\tau^N_R &\overset{\substack{\eqref{eqn:tau_est_3}\\\eqref{eqn:tau_est_term}}}{\leq} \sqrt{T} \int_{\{\|x\| \geq R\}} \|x\| \, \mu(dx) + C \, \frac{R}{N} \sum_{n = 1}^N \mathds{1}_{\{\|X^n\|\ge R/2\}} + \frac{2\sqrt{T}}{N} \sum_{n = 1}^N \|X^n\| \cdot \mathds{1}_{\{\|X^n\|\ge R/2\}}\\
&\leq \sqrt{T} \int_{\{\|x\| \geq R\}} \|x\| \, \mu(dx) + \frac{C}{N} \sum_{n = 1}^N \|X^n\| \cdot \mathds{1}_{\{\|X^n\|\ge R/2\}},
\end{split}
\end{align}
where $C > 0$ depends on $T$, and the second inequality follows from $\tfrac{R}{2} \mathds{1}_{\{\|x\|\ge R/2\}} \leq \|x\| \mathds{1}_{\{\|x\|\ge R/2\}}$. As $\mu$ is $p$-integrable, we conclude by H\"older's and Markov inequality
\begin{equation}\label{eqn:tau_holder_markov}
\int_{\{\|x\| \geq R\}} \|x\| \, \mu(dx) \leq \mathbb{E}[|X^1|^p]^\frac{1}{p} \cdot \mathbb{P}\!\left(\|X^1\| \geq R\right)^\frac{p-1}{p} \leq \frac{\mathbb{E}[|X^1|^p]}{R^{p-1}},
\end{equation}
and therefore
\begin{equation}
\mathbb{E}\, \tau^N_R \overset{\substack{\eqref{eqn:tau_dec}\\\eqref{eqn:tau_holder_markov}}}{\leq} \frac{C}{R^{p - 1}},
\end{equation}
where $C > 0$ depends on $T$ and $\int \|x\|^p \, d\mu$; this proves \emph{(1)}.

To see \emph{(2)}, observe that $\tau^N_R \geq \mathbb{E} \, \tau^N_R$ can only happen if $\kappa_R(\varphi^N(X^n)) \neq \varphi^N(X^n)$ or $\kappa_R(\varphi^N(X^n)) \neq \varphi^N(\kappa_R(X^n))$ for some $n$. This implies $\|X^n\| \geq \tfrac{R}{2}$, as $\kappa_R$ is the identity and $\kappa_R \circ \varphi^N = \varphi^N \circ \kappa_R$ on $B_{\frac{R}{2}}(0) \subseteq \R^d$. Thus, by union bound
$$
\mathbb{P}\!\left(\tau^N_R \geq \mathbb{E}\, \tau^N_R\right) \leq N\, \P(\|X^1\|\ge R/2)\le \mathcal{E}_{\alpha, \gamma} \cdot N \exp\!\left(-c \cdot R^\alpha\right),
$$
where $c > 0$ depends on $\alpha, \gamma$. Next, recall the $\phi$-mixing coefficient defined in \eqref{eqn:phi_mixing}. By Proposition \ref{prop:phi_mixing} with $f(\slices X) := (f(X^n))_{n = 1}^N$ and by Proposition \ref{prop:eta_mixing_information_processing} applied with $f(x) = \|x\| \cdot \mathds{1}_{\{\|x\|\ge R/2\}}$ we have
\begin{equation}\label{eqn:tau_SLLN_coeff}
\phi_{f(\slices X)}(s) \leq \sum_{k = s}^{\infty} \eta_{f(\slices X)}(k) \leq \sum_{k = s}^{\infty} \eta_{\slices X}(k).
\end{equation}
Thus, by the Marcinkiewicz–Zygmund SLLN for $\phi$-mixing sequences (see \cite[Equation (1)]{kuczmaszewska2011strong}) applied with $f(\slices X)$ (note that $\sum_{s = 1}^\infty \phi^\frac{1}{2}_{f(\slices X)}(2^s) < \infty$ by assumption) we have
\begin{equation}\label{eqn:tau_slln}
\frac{1}{N} \sum_{n = 1}^N \|X^n\| \cdot \mathds{1}_{\{\|X^n\|\ge R/2\}} \to \int_{\{\|x\| \geq \frac{R}{2}\}} \|x\| \, \mu(dx), \qquad N \to \infty,
\end{equation}
$\mathbb{P}$-almost surely. Therefore,
$$
\limsup_{N \to \infty} \tau^N_R \overset{\substack{\eqref{eqn:tau_dec}\\\eqref{eqn:tau_slln}}}{\leq} C \cdot \int_{\{\|x\| \geq \frac{R}{2}\}} \|x\| \, \mu(dx) \overset{\eqref{eqn:tau_holder_markov}}{\leq} \frac{C}{R^{p - 1}},
$$
where $C > 0$ depends on $T, p, \int \|x\|^p\, d\mu$. The proof is complete.
\end{proof}

\bibliography{bib}
\bibliographystyle{siam}

\end{document}

%% file: figures/observations.tikz
\newcommand{\RectWithBrace}[6]{%
  \draw[thick] (#1,#3) rectangle ({#1 + #2},#4);

  \pgfmathsetmacro{\cellwidth}{1}
  \foreach \i [count=\n] in {#5} {
    \pgfmathsetmacro{\xpos}{#1+\n*\cellwidth}
    \ifnum\n<#2
      \draw[thick] (\xpos,#3) -- (\xpos,#4);
    \fi
    \node at ({#1+(\n-0.5)*\cellwidth}, {(#3+#4)/2}) {\small $\i$};
  }

  \draw [decorate, decoration={calligraphic brace, amplitude=8pt}]
    (#1,#4+0.05) -- ({#1 + #2},#4+0.05)
    node [midway, yshift=0.6cm] {\small $#6$};
}

\begin{tikzpicture}[>=Stealth, scale=0.68]

  \def\xmax{18}
  \def\ymax{0}

  \draw[step=0.5cm, lightgray!70, very thin] (0,0) grid (\xmax,\ymax);

  \draw[->, thick] (0,0) -- (\xmax+0.3,0) node[below right] {time};
  \foreach \x in {0,1,...,12} {
    \draw[thick] (\x,0) -- (\x,-0.15); 
    \node[below] at (\x,-0.15) {\small \x}; 
  }
  \node[below] at (13, -0.30) {\ldots};

  \RectWithBrace{2}{3}{0}{1}{X^1_1,\ldots,X^1_T}{X^1}
  \RectWithBrace{8}{3}{0}{1}{X^2_1,\ldots,X^2_T}{X^2}
  \RectWithBrace{14}{3}{0}{1}{X^n_1,\ldots,X^n_T}{X^n}

\end{tikzpicture}